\documentclass[12pt] {article}
\usepackage{float}
\usepackage{delarray}
\usepackage{amsmath, latexsym, amsfonts, amssymb, amsthm, amscd,enumerate,hyperref}
\usepackage{graphicx}
\usepackage{color}
\usepackage{amsmath, latexsym, amsfonts, amssymb, amsthm, amscd,multirow}
\usepackage{mathrsfs}

\allowdisplaybreaks

\renewcommand{\leq}{\leqslant}
\renewcommand{\geq}{\geqslant}

\newcommand{\N}{\mathbb{N}}

\newcommand{\R}{\mathbb{R}}
\newcommand{\C}{\mathbb{C}}

\newcommand{\1}{1\!\!{\sf I}}

\newcommand{\tr}{\operatorname{tr}}
\newcommand{\Tr}{\operatorname{Tr}}

\newtheorem{theorem}{Theorem}
\newtheorem{proposition}[theorem]{Proposition}
\newtheorem{remark}[theorem]{Remark}
\newtheorem{lemma}[theorem]{Lemma}

\newtheorem{corollary}[theorem]{Corollary}

\theoremstyle{definition}
\newtheorem{example}[theorem]{Example}

\title{Strong convergence of  tensor products of independent G.U.E.  matrices}
\author{Serban T. Belinschi and Mireille Capitaine}
\date{}
\begin{document}
\maketitle

\begin{abstract}
Given tuples of properly normalized independent $N\times N$ G.U.E. matrices $(X_N^{(1)},\dots,X_N^{(r_1)})$ and $(Y_N^{(1)},\dots,Y_N^{(r_2)})$, we show that the tuple
$(X_N^{(1)}\otimes I_N,\dots,X_N^{(r_1)}\otimes I_N,I_N\otimes Y_N^{(1)},\dots,I_N\otimes Y_N^{(r_2)})$ of $N^2\times N^2$ random matrices converges strongly as $N$ 
tends to infinity. It was shown by Ben Hayes 
that this result implies that the Peterson-Thom conjecture is true.
\end{abstract}

\section{Introduction}
In the fifties,  Wigner proposed in quantum mechanics to replace the selfadjoint Hamiltonian operator in an infinite dimensional Hilbert
space by an ensemble of very large Hermitian matrices. Let us present the  most frequently used Gaussian matrix ensembles G.U.E. 
(Gaussian Unitary Ensemble, so named because their law is invariant under conjugation by unitary matrices). A G.U.E.  matrix $W_N$ 
of size $N$ is a selfadjoint matrix such that the entries $( W_N)_{i,j}$ of $ W_N$ are centered Gaussian  random variables satisfying:
 $\{( W_N)_{i,i}\colon1\le i\le N\}\cup\{\Re( W_N)_{i,j},\Im( W_N)_{i,j}\colon1\le j<i\le N\}$ are independent and 
 $( W_N)_{i,i},1\le i\le N$, and $\sqrt{2}\Re( W_N)_{i,j},\sqrt{2}\Im(W_N)_{i,j}$, $1\le j<i\le N$ have all variance equal to 1.
We call $X_N=\frac{W_N}{\sqrt{N}}$ a standard {\em normalized} G.U.E. The famous theorem of Wigner \cite{Wigner1, Wigner2} states that  
the empirical distribution of the eigenvalues of $X_N$ converges weakly, in probability, to the standard semicircular distribution $\mu_{sc}$ where 
${\rm d}\mu_{sc}(t)= \frac{1}{2\pi} \sqrt{4-t^2} \1_{[-2,2]}(t){\rm d}t$. The semicircular distribution appeared also as the central limit distribution in 
Voiculescu's free probability theory, developped in the eighties (see \cite{boxplus}). This occurrence  hinted at a closer relation between free probability 
and random matrices. In the early 90’s, Voiculescu  \cite{V} made this concrete by showing that also freeness shows up asymptotically in the random matrix world.
Indeed, one of the results of \cite{V} states that  an $r$-tuple of independent standard normalized G.U.E.-distributed matrices  $X_N^{(1)},\ldots, X_N^{(r)}$ 
are asymptotically free, that is, for every noncommutative polynomial $P$ in $r$ variables, one has
$$
\mathbb{E}[{\rm tr}_N P(X_N^{(1)},\ldots, X_N^{(r)})]\underset{N\rightarrow +\infty}\rightarrow\tau (P(s_1, \ldots, s_r))
$$
where ${\rm tr}_N$ stands for the normalized trace on $M_N(\mathbb C)$ and $s_1, \ldots, s_r$ are free standard semicircular variables in a $C^*$-probability space 
$(\mathcal A,\tau)$ (for an introduction to free probability we refer to \cite{VDN}). A noncommutative random variable $x$ in $(\mathcal A,\tau)$ is called a standard 
semicircular variable if its distribution with respect to $\tau$ is $\mu_{sc}$; that is, if $x=x^*$ (i.e. $x$ is selfadjoint) and for any $k\in\mathbb N$, 
$\tau(x^k)=\int t^k\,{\rm d}\mu_{sc}(t)$. In \cite{HiP} and \cite{T}, the authors showed that the convergence in Voiculescu's result actually holds almost
surely, that is almost surely, for every noncommutative polynomial $P$ in $r$ variables,
\begin{equation}\label{ascv}
\tr_N P(X_N^{(1)},\ldots, X_N^{(r)})\underset{N\rightarrow +\infty}\rightarrow\tau (P(s_1, \ldots, s_r)).
\end{equation}
Later, Haagerup and Thorbj{\o}rnsen \cite{HT} proved a strong version of \eqref{ascv}, in the G.U.E. case, namely a convergence for the operator norm:  almost surely, for every
noncommutative polynomial $P$ in $r$ variables,
\begin{equation}\label{HTs}
\lim_{N\rightarrow +\infty}\left\|P(X_N^{(1)},\ldots, X_N^{(r)})\right\|=\left\|P(s_1,\ldots, s_r)\right\|.
\end{equation}
Now, let $X_N=\{X_N^{(i)},i=1,\ldots r_1\}$ and $Y_N=\{ Y_N^{(j)},i=1,\ldots r_2\}$ be independent tuples of independent normalized standard G.U.E. $N\times N$ matrices. Let 
$({\cal A},\tau)$ be a ${C}^*$-probability space equipped with a faithful tracial state and ${\bf s}=\{{\bf s}_i,i=1,\dots, r_1\}$ and ${\bf t}=\{{\bf t}_i,i=1,\dots,r_2\}$  be 
(possibly different) free semicircular systems, that is tuples of free standard semicircular variables,  in $({\cal A},\tau)$. Denote by $I_N$ the $N\times N$ identity matrix and by 
$1_{\bf s}$ the unit of $C^*({\bf s}_1,\dots,{\bf s}_r)$, the unital ${C}^*$-algebra generated by the free semicirculars ${\bf s}_1,\dots,{\bf s}_r$.  It is straightfoward to deduce 
from \eqref{ascv} that almost surely  for any noncommutative polynomial $P$  in $r_1+r_2$ variables, one has
\begin{eqnarray}
\lefteqn{\lim_{N\to\infty}(\tr_N\otimes\tr_N)\!\left[P(X_N^{(1)}\otimes I_N,\dots,X_N^{(r_1)}\otimes I_N,I_N\otimes Y_N^{(1)},\dots,I_N\otimes Y_N^{(r_2)})\right]}\nonumber\\
&\quad\quad\quad\quad\ =&(\tau\otimes\tau)\!\left[P({\bf s}_1\otimes1_{\bf t},\dots,{\bf s}_{r_1}\otimes1_{\bf t},1_{\bf s}\otimes{\bf t}_1,\dots,1_{\bf t}\otimes{\bf t}_{r_2})\right] \label{asymptoticfreeness}.
\end{eqnarray}
In this paper, we prove the following strong convergence. 
\begin{theorem}\label{mainresult}
 Almost surely, for any noncommutative polynomial $P$  in $r_1+r_2$ variables, 
\begin{eqnarray}
\lefteqn{\lim_{N\to\infty}\left\|P(X_N^{(1)}\otimes I_N,\dots,X_N^{(r_1)}\otimes I_N,I_N\otimes Y_N^{(1)},\dots,I_N\otimes Y_N^{(r_2)})\right\|}\nonumber\\
&\quad\quad\quad\quad\ =&\left\|P({\bf s}_1\otimes1_{\bf t},\dots,{\bf s}_{r_1}\otimes1_{\bf t},1_{\bf s}\otimes{\bf t}_1,\dots,1_{\bf s}\otimes{\bf t}_{r_2})\right\|_{\rm min} \label{scv}.
\end{eqnarray}
\end{theorem}
One views $P({\bf s}_1\otimes1_{\bf t},\dots,{\bf s}_{r_1}\otimes1_{\bf t},1_{\bf s}\otimes{\bf t}_1,\dots,1_{\bf s}\otimes{\bf t}_{r_2})\in
C^*({\bf s}_1,\dots,{\bf s}_{r_1})\otimes_{\rm min} C^*({\bf t}_1,\dots,{\bf t}_{r_2})$, the completion of the algebraic tensor product $C^*({\bf s}_1,\dots,{\bf s}_{r_1})\otimes C^*({\bf t}_1,
\dots,{\bf t}_{r_2})$ with respect to $\|\cdot\|_{\rm min}$, the minimal, or spatial, norm on $C^*({\bf s}_1,\dots,{\bf s}_r)\otimes C^*({\bf s}_1,\dots,{\bf t}_{r_2})$ - 
see \cite[Section 2.1]{H} and \cite[Section II.9.1]{Bruce}; since in this paper we consider only the algebraic and the spatial, minimal tensor product norm on our tensor products of 
$C^*$-algebras - or von Neumann algebras - from now on we suppress the ``min" subscript in our norm notations.)
\\

In the wake of Voiculescu's discovery, random matrix theory became a powerful tool in the study of operator algebras. The option of modeling operator algebras 
asymptotically by random matrices led to new results on von Neumann algebras, in particular on the free group factors \cite{VDN}.
Our investigation is motivated by a result of Ben Hayes \cite{H}, showing that a conjecture about the structure of certain finite von Neumann algebras is implied by a strong 
convergence result for tuples of random matrices. Specifically, the conjecture (see \cite[Conjecture 1]{H}) is the following. Assume $r\in\mathbb N,r>1,$ is given. Denote 
by $\mathbb F_r$ the free group with $r$ free generators, and by $L(\mathbb F_r)$ the free group von Neumann algebra, that is, the von Neumann algebra generated by the 
left regular representation of $\mathbb F_r$ in the space $B(\ell^2(\mathbb F_r))$ of bounded linear operators on the Hilbert space $\ell^2(\mathbb F_r)$. Assume that $Q$ is a 
von Neumann subalgebra of $L(\mathbb F_r)$ which is diffuse (meaning that it contains no minimal projections) and amenable (meaning that there exists a conditional expectation 
$E\colon B(\ell^2(\mathbb F_r))\to Q$). Then there exists a unique maximal amenable von Neumann subalgebra $P$ of $L(\mathbb F_r)$ such that $Q\subseteq P$. 
(For the terminology related to the structure of von Neumann algebras, we refer to \cite{AP} - specifically in our case to \cite[Definition 2.4.12]{AP} and \cite[Definition 10.2.1 and
Proposition 10.2.2]{AP}.) This conjecture, cited in \cite{H}, is known as the Peterson-Thom conjecture \cite{PeT}. Ben Hayes proved in \cite[Theorem 1.1]{H} that if 
\eqref{scv} holds    for $r_1=r_2$, then the Peterson-Thom conjecture is true as well. Note that previous works established strong convergence of matrices $X_N^{(1)}\otimes I_M,
\dots,X_N^{(r)}\otimes I_M,I_N\otimes Y_M^{(1)},\dots,I_N\otimes Y_M^{(r)}$, where the dimension of the G.U.E. matrices $Y_M^{(i)}$'s is $M$ and $M=o(N^{1/4})$ in \cite{Pi}, 
$M=o(N^{1/3})$ in \cite{CGP}, and $M=o(N/(\log N)^3)$ in \cite{BBV}. This does not suffice for the purpose of \cite{H}, which requires $M=N$. Since the appearance of our paper on 
arXiv, Bordenave and Collins \cite{BoC} have obtained independently another proof of the Peterson-Thom conjecture via Hayes' theorem, by showing the strong convergence of tensor 
products of independent Haar-distributed unitary random matrices instead of G.U.E. matrices.
\\

Our approach is very similar to that of \cite{HT, S}. Therefore, we will recall the first ideas of the proof of \eqref{HTs}.  First,  the minoration
$$
\liminf_{N\rightarrow +\infty}\left\|P(X_N^{(1)},\ldots, X_N^{(r)})\right\|\geq\left\|P(s_1,\ldots, s_r)\right\|
$$ 
comes rather easily from \eqref{ascv}. So, the main difficulty is the proof of the reverse inequality:
\begin{equation}\label{limsupHT}
\limsup_{N\rightarrow +\infty}\left\|P(X_N^{(1)},\ldots, X_N^{(r)})\right\|\leq\left\|P(s_1,\ldots, s_r)\right\|
\end{equation}
Section 2 and Proposition 7.3 of \cite{HT}  established that, in order to prove \eqref{limsupHT}, it is sufficient to prove that,  for all $m\in \N$, all self-adjoint matrices $a_0,\ldots,a_r$
of size $m\times m$ and all $\epsilon>0$, almost surely for all large $N$,
\begin{equation}\label{incluHT}
{\rm sp}\left(a_0\otimes I_N +\sum_{i=1}^ra_i \otimes X_N^{(i)}\right)\subset {\rm sp}\left(a_0\otimes I_N+\sum_{i=1}^ra_i \otimes s_i\right)+(-\epsilon,\epsilon). 
\end{equation}
Here, ${\rm sp}(T)$ denotes the spectrum of the operator $T$. The analysis of the spectrum of $\Sigma_N = a_0\otimes I_N +\sum_{i=1}^ra_i \otimes X_N^{(i)}$ was done using
the matrix-valued Cauchy transform $G_N(\lambda) = \mathbb{E}[({\rm id}_{M_m(\mathbb C)}\otimes \tr_N)[(\lambda\otimes I_N -\Sigma_N)^{-1}]], \lambda \in M_m(\mathbb C), 
\Im(\lambda)$ positive definite; the proof of \eqref{incluHT} required sharp estimates of the rate of convergence of $g_N(z)=\tr_m G_N(zI_m)$ to $g(z)=\tr_m G(zI_m)$ where 
$G(\lambda)=({\rm id}_{M_m(\mathbb C)}\otimes \tau)[(\lambda\otimes I_N -\Sigma)^{-1}]$, $\Sigma_ := a_0\otimes 1_{\cal A}+\sum_{i=1}^r a_i \otimes s_i$, $z\in \mathbb C
\setminus\mathbb R$. In \cite{HT}, such sharp estimates were deduced from an  equation and an approximate equation respectively, satisfied by $G$ and $G_N$.  By necessity, our 
method to prove such sharp estimates dealing now with the tensorized G.U.E. matrices involved in Theorem \ref{mainresult} is different and is based on a series expansion of the 
resolvents viewed as rational noncommutative functions, successively in each of their variables. Thus, first, thanks to results from \cite{Yin}, we re-phrase Theorem \ref{mainresult} in 
terms of the Cayley transforms of the selfadjoint variables involved in order to deal with a bounded sequence of matrices (see Theorem \ref{scunitary}). Then a series expansion 
around infinity of the resolvents and an expansion in $1/N^2$ of expectations of normalized traces of polynomials in Cayley transforms of tuples of independent G.U.E. matrices allow 
us to obtain a precise formula for $g_N(z)-g(z)$, for  large $|z|$ (here $g_N$ and $g$ are obtained by replacing the G.U.E. matrices and the semicircular operators in the above 
$\Sigma_N$ and  $\Sigma$ by the tensorized Cayley transforms of G.U.E. and semicircular operators involved in Theorem \ref{scunitary}). Since this formula we obtained for large $|z|$ 
involves functions that  can be analytically extended to $\mathbb C\setminus\mathbb R$, we can deduce that it holds on the whole $\mathbb{C} \setminus \mathbb R$. Our proof relies 
heavily on the explicit  asymptotic expansion of smooth functions in polynomials in independent  G.U.E. matrices obtained  in the beautiful work of  Parraud \cite{P}.  We use Parraud's 
formulae in an essential way in our work, but with the free difference quotient replaced by the difference-differential operator, via the natural identification between the two operations.

We end the introduction with a brief outline of the rest of the paper. In Section \ref{Sec:WU} we introduce the main objects of interest for our study and re-phrase Theorem 
\ref{mainresult} in terms of the Cayley transforms of the selfadjoint variables involved (see Theorem \ref{scunitary}). In Section \ref{HTSapproach} we explain in detail how the 
methods introduced by Haagerup, Thorbj{\o}rnsen, and Schultz for proving strong convergence of random matrices \cite{HT,S} apply to prove Theorem \ref{scunitary}. In Section 
\ref{Sec:nctools} we introduce the necessary tools from noncommutative analysis, and prove a number of auxiliary results about them. This section follows largely
\cite{KVV,P,VFE5,Coalg}. Section \ref{S:AsyEx} is dedicated to the extension of Parraud's work \cite{P} to rational functions.  Section \ref{Sec:pfHT} contains the proof of the 
main step \eqref{estimdiffeqno} indicated in Section \ref{HTSapproach}. 

\noindent {\bf Acknowledgements} We thank F\'elix Parraud for useful discussions on his formulae.
The current version of the paper benefits from numerous helpful questions and comments from Ben Hayes, David Jekel, and Mikael de la Salle, to whom we are grateful.

\section{From G.U.E. to unitary matrices}\label{Sec:WU}
In this section, we reformulate Theorem \ref{mainresult} as strong convergence result of some unitary matrices built from G.U.E. matrices. This reformulation is based on the following 
result of S. Yin \cite{Yin}.
\begin{theorem}\label{polrat}\cite{Yin}
Let $(\mathcal A_n, \tau_n)$, $(\mathcal A, \tau)$ be faithful tracial $C^*$-probability spaces.
{Assume} that  {$x(n) = (x_1(n),\ldots, x_r(n))\in (\mathcal A_n, \tau_n)$ strongly converges to $x = (x_1,\ldots, x_r)$} in $(\mathcal A, \tau)$, that is, for any polynomial $P$ in $2r$ 
non-commuting indeterminates,
$$
\lim_{n\rightarrow +\infty} \tau_n(P(x(n), x(n)^*))=\tau(P(x,x^*)),
$$
$$
\lim_{n\rightarrow +\infty} \|P(x(n), x(n)^*)\|_{\mathcal A_n}=\|P(x,x^*)\|_{\mathcal A}.
$$
Assume moreover that the tuple $(x, x^*)$ lies in the domain of {a rational function $R$}. {Then} $R(x(n), (x(n))^*)$ is well defined  for sufficiently large n and we have 
{$$\lim_{n\rightarrow +\infty} \tau_n(R(x(n), x(n)^*))=\tau(R(x,x^*)),$$}
{$$\lim_{n\rightarrow +\infty} \|R(x(n), x(n)^*)\|_{\mathcal A_n}=\|R(x,x^*)\|_{\mathcal A}.$$}
\end{theorem}
Let 
$$
X=(X_1,\ldots,X_{r_1}),
$$
$$
Y=(Y_1,\ldots,Y_{r_2}), 
$$
be independent tuples of independent normalized G.U.E. matrices. Let $({\cal A},  \tau)$ be a ${C}^*$-probability space equipped with a faithful tracial state and ${\bf s}=
\{{\bf s}_i,i=1,\dots, r_1\}$ and ${\bf t}=\{{\bf t}_i,i=1,\dots,r_2\}$  be  free semicircular systems  in $({\cal A},\tau)$. Consider the Cayley transform 
$\Psi\colon\overline{\mathbb C}\to\overline{\mathbb{C}}$ given by $\Psi(z)=\frac{z+i}{z-i}$. This is an automorphism of the extended complex plane 
$\overline{\mathbb C}=\mathbb C\cup\{\infty\}$ which sends the complex lower half-plane $\mathbb C^{-}$ onto the unit disk $\mathbb D$, infinity to the complex number 1, and the 
extended real line $\mathbb R\cup\{\infty\}$ onto the unit circle $\mathbb T$. Its inverse with respect to composition  $\Psi^{\langle-1\rangle}\colon\overline{\mathbb C}\to
\overline{\mathbb C}$ is given by $\Psi^{\langle-1\rangle}(w)=i\frac{w+1}{w-1}$. By continuous functional calculus, one evaluates $\Psi$ on bounded selfadjoint linear operators on 
Hilbert spaces and this gives rise to  unitary operators. Conversely, if $U$ is a unitary operator whose spectrum does not contain $1$, then $\Psi^{\langle-1\rangle}(U)$ is a 
bounded selfadjoint operator. Set 
$$
(\Psi(X_1),\ldots,\Psi(X_{r_1}))=(U_1,\ldots, U_{r_1})=\underline{U}_{r_1},
$$
$$
(\Psi({\bf s}_1),\ldots,\Psi({\bf s}_{r_1}))=(u_1,\ldots, u_{r_1})=\underline{u}_{r_1},
$$
$$
(\Psi(Y_1),\ldots,\Psi(Y_{r_2}))=(V_1,\ldots, V_{r_2})=\underline{V}_{r_2},
$$
$$
(\Psi({\bf t}_1),\ldots,\Psi({\bf t}_{r_2}))= (v_1,\ldots, v_{r_2})=\underline{v}_{r_2},
$$
(due to a lack of space, we find it convenient sometimes to use an underline in order to denote a vector of objects, with the index denoting the length of the vector).
In order to prove Theorem \ref{mainresult}, it is sufficient, thanks to Yin's Theorem \ref{polrat}, to prove the following 
\begin{theorem}\label{scunitary}
For any selfadjoint polynomial $P$ in $r_1+r_2$ variables and their adjoints, 
$\| P(\underline{U\otimes I_N}_{r_1}, \underline{I_N \otimes V}_{r_2},\allowbreak\underline{U^*\otimes I_N}_{r_1},\underline{I_N \otimes V^*}_{r_2})\|$ converges to
$\| P(\underline{u\otimes 1_{\cal A}}_{r_1},\underline{1_{\cal A}\otimes v}_{r_2},\underline{u^*\otimes1_{\cal A}}_{r_1},\underline{1_{\cal A}\otimes v^*}_{r_2})\|$.
\end{theorem}
Indeed, as we have seen above, 
$\Psi^{\langle-1\rangle}(U)=i\frac{U+1}{U-1}$ and $\frac{U+1}{U-1}\otimes I=I\otimes I +2 (U\otimes I -I\otimes I)^{-1}$. A noncommutative monomial $M$ of degree $n$ in $r_1+
r_2$ indeterminates is written as $M=\alpha\mathfrak X_{i_1}\mathfrak X_{i_2}\cdots\mathfrak  X_{i_n}$ for some $i_1,\dots,i_n\in\{1,\dots,r_1+r_2\}$. Then
\begin{eqnarray*}\lefteqn{M(X_1\otimes I_N,\ldots, X_{r_1}\otimes I_N, I_N \otimes Y_1,\ldots,I_N \otimes Y_{r_2})}\\
&=&\alpha \left(\prod_{j\colon i_j\in\{1,\dots,r_1\}}X_{i_j}\otimes 1\right)\left( \prod_{k\colon i_k\in\{r_1+1,\dots,r_1+r_2\}} 1\otimes Y_{i_k-r_1}\right)\\
&=&\alpha\left(\prod_{j\colon i_j\in\{1,\dots,r_1\}}i[I\otimes I +2 (U_{i_j}\otimes I -I\otimes I)^{-1}]\right)\\
& & \mbox{}\times\left( \prod_{k\colon i_k\in\{r_1+1,\dots,r_1+r_2\}}i[I\otimes I +2 (I\otimes V_{i_k-r_1} -I\otimes I)^{-1}]\right).
\end{eqnarray*} 
Thus $M(X_1\otimes I_N,\ldots, X_{r_1}\otimes I_N, I_N \otimes Y_1,\ldots,I_N \otimes Y_{r_2})$ is a rational function in $U_1\otimes I_N,\ldots, U_{r_1}\otimes I_N, I_N \otimes V_1,
\ldots,I_N \otimes V_{r_2}.$  Since a polynomial is a sum of monomials, this implies that any noncommutative polynomial in $X_1\otimes I_N,\ldots, X_{r_1}\otimes I_N, I_N \otimes Y_1,
\ldots,I_N \otimes Y_{r_2}$  is a rational function of  $U_1\otimes I_N,\ldots, U_{r_1}\otimes I_N, I_N \otimes V_1,\ldots,I_N \otimes V_{r_2}.$ Therefore, if we are able to prove the 
strong convergence of $U_1\otimes I_N,\ldots, U_{r_1}\otimes I_N, I_N \otimes V_1,\ldots,I_N \otimes V_{r_2}$ (that is Theorem \ref{scunitary}, then we will deduce the strong 
convergence of $X_1\otimes I_N,\ldots, X_{r_1}\otimes I_N, I_N \otimes Y_1,\ldots,I_N \otimes Y_{r_2}$ (that is Theorem \ref{mainresult}) by Theorem  \ref{polrat}.\\

Therefore, in the following, we will focus on the proof of  Theorem \ref{scunitary}.

\section{Haagerup, Thorbj{\o}rnsen and Schultz's approach}\label{HTSapproach}

We present here the arguments of \cite{HT} and \cite{S} which still hold in our framework of Theorem \ref{scunitary} and point out the result which requires new ideas and techniques. 

First, it is straightfoward to deduce from Theorem \ref{polrat} that almost surely, for any noncommutative polynomial $P$  in $r_1+r_2$ variables and their adjoints, 
$$
{\rm tr}_N\otimes{\rm tr}_N \left[ P(
\underline{U\otimes I_N}_{r_1},\underline{I_N \otimes V}_{r_2},\underline{U^*\otimes I_N}_{r_1},\allowbreak\underline{I_N \otimes V^*}_{r_2})\right]
$$
\begin{equation}\label{lib}
\rightarrow_{N\rightarrow +\infty} \tau \otimes \tau \left[P\left( 
\underline{u\otimes 1_\mathcal A}_{r_1},\underline{1_\mathcal A\otimes v}_{r_2},\underline{u^*\otimes1_\mathcal A}_{r_1},\underline{1_\mathcal A\otimes v^*}_{r_2}\right)\right]
\end{equation}
(recall the notations introduced before Theorem \ref{scunitary}).  It turns out that in proving Theorem \ref{scunitary}, the minoration almost surely, for any $P$,
$$
\liminf_{N\rightarrow +\infty} \left\|P(
\underline{U\otimes I_N}_{r_1},\underline{I_N \otimes V}_{r_2},\underline{U^*\otimes I_N}_{r_1},\underline{I_N \otimes V^*}_{r_2})\right\|
$$
$$
\geq\left\|P\left(\underline{u\otimes1_\mathcal A}_{r_1},\underline{1_\mathcal A\otimes v}_{r_2},\underline{u^*\otimes 1_\mathcal A}_{r_1},
\underline{1_\mathcal A \otimes v^*}_{r_2}\right)\right\|
$$
should follows rather easily (sticking to  the proof of  Lemma 7.2 in \cite{HT}). So, the main difficulty is the proof of the reverse inequality: almost surely for any $P$,
$$
\limsup_{N\rightarrow +\infty} \left\|P(
\underline{U\otimes I_N}_{r_1},\underline{I_N \otimes V}_{r_2},\underline{U^*\otimes I_N}_{r_1},\underline{I_N \otimes V^*}_{r_2})\right\|
$$
\begin{equation}\label{ri}
\leq \left\|P\left(\underline{u\otimes1_\mathcal A}_{r_1},\underline{1_\mathcal A\otimes v}_{r_2},\underline{u^*\otimes1_\mathcal A}_{r_1},
\underline{1_\mathcal A \otimes v^*}_{r_2}\right)\right\|.
\end{equation}
Thanks to a linearization trick (following \cite[Section 2]{HT} and the proof of Proposition 7.3 from the same \cite{HT}), in order to prove \eqref{ri}, it suffices to prove:
\begin{proposition}\label{inclu2} 
For all $m\in\mathbb N$, all  matrices $\xi=\xi^*, \gamma_1,\dots,\gamma_{r_1},\beta_1,\ldots,\beta_{r_2}\in M_m(\mathbb C)$ and all $\varepsilon>0$, almost surely, for all large 
$N$, we have 
\begin{equation} \label{spectre3} 
{\rm sp}\left(\xi\otimes I_N\otimes I_N+S_U+S_V\right)\subset{\rm sp}\left(\xi\otimes1_{\cal A}\otimes1_{\cal A}+\mathcal{S}_u+\mathcal{S}_v\right)+(-\varepsilon,\varepsilon),
\end{equation}
where  
$$S_U=\sum_{i=1}^{r_1}\left(\gamma_i\otimes U_i\otimes I_N+\gamma_i^*\otimes U_i^*\otimes I_N\right)=2\Re\sum_{i=1}^{r_1} \gamma_i \otimes U_i\otimes I_N,$$
$$S_V=\sum_{i=1}^{r_2}\left(\beta_i \otimes I_N\otimes V_i+\beta_i^*\otimes I_N \otimes V_i^*\right)=2\Re\sum_{i=1}^{r_2} \beta_i \otimes I_N\otimes V_i,$$
$$\mathcal{S}_u=\sum_{i=1}^{r_1}\left(\gamma_i \otimes u_i\otimes 1_{\mathcal{A}}+\gamma_i^*\otimes u_i^*\otimes 1_{\mathcal{A}}\right)=2\Re\sum_{i=1}^{r_1}\gamma_i \otimes u_i\otimes 1_{\mathcal{A}},$$
$$
\mathcal{S}_v=\sum_{i=1}^{r_2}\left(\beta_i\otimes1_\mathcal A\otimes v_i+\beta_i^*\otimes1_\mathcal A\otimes v_i^*\right)=2\Re\sum_{i=1}^{r_2}\beta_i\otimes1_\mathcal A\otimes v_i.
$$
Here, ${\rm sp}(T)$ denotes the spectrum of the operator $T$, $I_N$ the identity matrix and $1_{\cal A}$ denotes the unit of ${\cal A}$.
\end{proposition}
As the risk of confusion is small, from here on we denote both the unit of the von Neumann algebra $\mathcal A$ and the complex number one by $1$, preserving 
$I$ with a subscript exclusively for the unit of the matrix algebra of matrices of the size indicated by the subscript. Set 
\begin{eqnarray}  
\lefteqn{S_N=\xi\otimes I_N\otimes I_N+\sum_{i=1}^{r_1}\gamma_i\otimes U_i\otimes I_N+\sum_{i=1}^{r_1}\gamma_i^*\otimes U_i^*\otimes I_N}\label{defSN}\\
& &\mbox{}+\sum_{j=1}^{r_2}\beta_j\otimes I_N\otimes V_j+\sum_{j=1}^{r_2}\beta_j^*\otimes I_N\otimes V_j^*=\xi\otimes I_N\otimes I_N+S_U+S_V\nonumber
\end{eqnarray} and 
$$
\mathcal{S}=\xi\otimes1\otimes 1+\sum_{i=1}^{r_1}(\gamma_i\otimes u_i\otimes 1+\gamma_i^*\otimes
u_i^*\otimes 1)+\sum_{j=1}^{r_2}(\beta_j\otimes 1\otimes {v_j}+\beta_j^*\otimes 1\otimes {v_j^*})
$$\vskip-.8truecm
$$
=\xi\otimes1\otimes 1+\mathcal{S}_u+\mathcal{S}_v.\quad\quad\quad\quad\quad\quad\quad\quad\quad\quad\quad\quad\quad\quad\quad\quad\quad\quad\quad\quad
$$
The proof of Proposition \ref{inclu2} requires sharp estimate of  $g_N(z)- g(z)$ where for $z\in \mathbb{C}\setminus \mathbb{R}$, 
\begin{equation}\label{defgn}
g_N(z) =\mathbb{E}(\mathrm{tr}_m \otimes\mathrm{tr}_N\otimes\mathrm{tr}_N)[ (zI_m \otimes I_N \otimes I_N - S_N)^{-1}]
\end{equation}
and 
\begin{equation}\label{defg} 
g(z)=(\mathrm{tr}_m\otimes\tau\otimes\tau)[ (zI_m \otimes 1\otimes 1- \mathcal{S})^{-1}].
\end{equation}
This estimate is detailed in Section 6. More precisely we are going to establish that there exists a polynomial $Q$ with nonnegative coefficients
such that, for $z \in \C \setminus \mathbb{R}$, 
\begin{equation} \label{estimdiffeqno}
\left|g_N(z)- g(z)-\frac{E(z)}{N^2}\right|\leq  \frac{Q(\vert \Im  z\vert^{-1})}{N^4},
\end{equation}
where $E(z)$ is the Stieltjes transform of a distribution $\hat \Lambda$ whose support is included in the spectrum of $\mathcal S$ and $\hat \Lambda(1)=0$.

 Now, we explain, for the reader's convenience because of differences arising from the dimension of our tensor matrices,  how \eqref{spectre3} can be deduced from 
\eqref{estimdiffeqno} by the  approaches introduced in \cite{HT,S}. Using the Stieltjes inversion formula for measures and compactly supported distributions (see 
\cite[Theorem 5.4]{S}), one obtains that, for any $\varphi$ in ${\cal C}^\infty (\mathbb R,\mathbb R)$ with compact support, 
$$
\mathbb{E} [\tr_m \otimes \tr _N \otimes \tr _N(\varphi(S_N))]-\tr_m\otimes \tau \otimes \tau(\varphi(\mathcal{S}))-\frac{\hat \Lambda(\varphi)}{N^2}
$$
$$
=\frac{i}{2\pi }\lim _{y\rightarrow 0^+} \int _\mathbb R\varphi (x)\left[\epsilon_N(x+iy)-\epsilon_N(x-iy)\right]\,{\rm d}x,
$$
where $\epsilon_N(z)= g_N(z)-g(z)-\frac{E(z)}{N^2}$ satisfies, according to \eqref{estimdiffeqno}, for any $z\in\mathbb C\setminus\mathbb R$, 
\begin{equation*}\label{estimgdif}
\vert \epsilon_N(z)\vert \leq \frac{Q(\vert \Im z \vert ^{-1})}{N^4}.
\end{equation*} 
We refer the reader to the Appendix of \cite{CD07}, where it is proved using the ideas of \cite{HT} that if $h$ is an analytic function on $\mathbb C\setminus\mathbb R$ which satisfies
\begin{equation*}\label{nestimgdif}
\vert h(z)\vert \leq Q(\vert \Im z \vert ^{-1})(|z|+K)^\alpha
\end{equation*} 
for some polynomial $Q$ with nonnegative coefficients and degree $k$, then there exists a polynomial $\tilde Q$ such that
\begin{eqnarray*}
\lefteqn{\limsup _{y\rightarrow 0^+}\left\vert \int _\R\varphi (x)h(x+iy)\,{\rm d}x\right\vert}\\
&\quad\quad \leq & \int _\mathbb R\int _0^{+\infty }\vert (1+D)^{k+1}\varphi(x)\vert(|x|+\sqrt{2}t+K)^\alpha \tilde Q(t)\exp(-t)\,{\rm d}t{\rm d}x,
\end{eqnarray*}
where $D$ stands for the derivative operator. Hence, dealing with $h(z) =N^4\epsilon_N(z)$ (respectively $x\mapsto \varphi(-x)$ instead of $\varphi$ and $h(z) =N^4\epsilon_N(-z)$), 
we deduce that there exists $C>0$ such that for all large $N$,
\begin{equation*} \label{majlimsup1} 
\limsup _{y\rightarrow 0^+}\left\vert \int _\R \varphi (x)\epsilon_N(x\pm iy)\,{\rm d}x\right\vert \leq \frac{C}{N^4},
\end{equation*} 
and then 
\begin{equation}\label{StS}
\mathbb E[{\tr}_m\otimes{\rm tr}_N\otimes\tr_N(\varphi(S_N))]-{\rm tr}_m\otimes\tau\otimes\tau(\varphi(\mathcal S))-\frac{\hat\Lambda(\varphi)}{N^2}=O\left(\frac1{N^4}\right).  
\end{equation}
Let $\rho\in{\cal C}^\infty(\mathbb R,\mathbb R)$ be such that $\rho\geq0$, its support is included in $[-1,1]$ and $\int_\mathbb R \rho (x)\,{\rm d}x=1$. Let $0 < \delta < 1$. 
Define for any $x\in \mathbb{R}$, 
$$
\rho _{\frac{\delta }{2}}(x)=\frac{2}{\delta}\rho\left(\frac{2x}{\delta }\right).
$$
Set 
$$
K(\delta )=\{x,{\rm dist}(x,{\rm sp}(\mathcal{S}))\leq\delta\}
$$ 
and define for any $x\in\mathbb{R}$, 
$$
f_\delta (x)=\int _\mathbb{R} \1 _{K(\delta)}(y)\rho _{\frac{\delta }{2}}(x-y)\,{\rm d}y.
$$
The function $f_{\delta}$ is in ${\cal C}^\infty (\mathbb{R}, \mathbb{R})$, $f_{\delta}\equiv 1$ on $K(\frac{\delta }{2})$; its support is included in $K(2\delta )$. 
Since there exists $C$ such that the spectrum of $\mathcal{S}$ is included in $[-C;C]$,  the support of $f_{\delta}$ is included in $[-C-2;C+2]$. Thus, according to \eqref{StS}, 
\begin{equation}
\mathbb{E}[{\rm tr}_m\otimes{\rm tr}_N\otimes{\rm tr}_N(f_{\delta}(S_n))]-{\rm tr}_m\otimes\tau\otimes\tau(f_{\delta}(\mathcal{S}))-\frac{\hat \Lambda (f_{\delta})}{N^2}
=O_{\delta }\left(\frac{1}{N^4}\right).
\end{equation}
 Since $\hat \Lambda(1)=0$,  the function $\psi_{\delta}\equiv 1-f_{\delta}$ also satisfies
\begin{equation} 
\mathbb E[{\rm tr}_m\otimes{\rm tr}_N\otimes\tr_N(\psi_{\delta}(S_N))]-{\rm tr}_m\otimes\tau\otimes\tau
\left(\psi_{\delta}(\mathcal{S})\right)-\frac{\hat\Lambda(\psi_{\delta})}{N^2}=O_{\delta }\left(\frac{1}{N^4}\right). 
\end{equation}
Now, since $\psi _\delta\equiv 0$ on the spectrum of $\mathcal{S}$, we deduce that 
\begin{equation}\label{psi2} 
\mathbb{E} [\tr_m\otimes \tr _N \otimes \tr_N(\psi_\delta(S_N))]=O_{\delta }\left(\frac{1}{N^4}\right).
\end{equation}
Let $m_\delta(S_N)$ be the number of eigenvalues of $S_N$ outside $K_{2\delta}$.
Since $\psi_\delta\geq\1 _{\mathbb{R}\setminus K({2\delta})}$,
we have $m_\delta(S_N) \leq \Tr_m\otimes \Tr _N \otimes \Tr_N(\psi_\delta(S_N))$, hence
$$\mathbb{P}(m_\delta(S_N)\geq 1)\leq\mathbb{E}[m_\delta(S_N)]\leq mN^2\mathbb{E}[\tr_m\otimes \tr _N \otimes \tr_N(\psi_\delta(S_N))]=O_{\delta }\left(\frac{1}{N^2}\right).$$
 Thus, by the Borell-Cantelli lemma,  
almost surely for all large $N$, the spectrum of $S_N$ is included in $K(2\delta )=\{ x, {\rm dist}(x, {\rm sp}(\mathcal{S}))\leq 2\delta \}$. 
Since this result holds for any  $m\times m$  matrices $\xi$, $\gamma$,  $\beta$, the proof of Proposition \ref{inclu2} is complete.\\

Thus, the difficulty of the paper relies on the  proof of \eqref{estimdiffeqno} to which the rest of the article is dedicated.
\\

In the following section, we first introduce preliminary notions and results in noncommutative analysis that are needed for our purposes.

\section{Noncommutative analysis}\label{Sec:nctools}

\subsection{Noncommutative polynomials}\label{pol}

Consider the star-algebra $\mathbb C\langle{\bf x}_1,\dots,{\bf x}_r\rangle=\mathbb C\langle\underline{\bf x}_r\rangle$ of noncommutative polynomials
in $r$ noncommuting selfadjoint indeterminates ${\bf x}_1,\dots,{\bf x}_r$. This is the algebra of the free monoid with $r$ free generators, and the star (adjoint) operation 
is the $\mathbb R$-linear extension of $(\lambda{\bf x}_{i_1}{\bf x}_{i_2}\cdots{\bf x}_{i_n})^*=\overline{\lambda}{\bf x}_{i_n}\cdots{\bf x}_{i_2}{\bf x}_{i_1}$,
$\lambda\in\mathbb C,n\in\mathbb N,i_1,i_2,\dots,i_n\in\{1,\dots,r\}$. On this algebra we consider Voiculescu's {\em free difference quotient}
\cite{VFE5,Coalg}:
\begin{equation}\label{FDQ}
\partial_j\colon\mathbb C\langle\underline{\bf x}_r\rangle\to\mathbb C\langle\underline{\bf x}_r\rangle\otimes\mathbb C\langle\underline{\bf x}_r\rangle,
\quad\partial_j1=0,\partial_j{\bf x}_k=\delta_{j,k}1\otimes1,
\end{equation}
extended by $\mathbb C$-linearity and the Leibniz rule: $\partial_j(PQ)=(\partial_jP)(1\otimes Q)+(P\otimes1)(\partial_jQ)$.
One can iterate $\partial_j$: for $k,l\in\mathbb N$ (possibly zero), is defined the obvious way:
$$
{\rm id}_{\mathbb C\langle \underline{\bf x}_r\rangle}^{\otimes k}\otimes\partial_j\otimes{\rm id}_{\mathbb C\langle \underline{\bf x}_r\rangle}^{\otimes l}\colon
\mathbb C\langle \underline{\bf x}_r\rangle^{\otimes k}\otimes \mathbb C\langle \underline{\bf x}_r\rangle\otimes\mathbb C\langle \underline{\bf x}_r\rangle^{\otimes l}
\to\mathbb C\langle \underline{\bf x}_r\rangle^{\otimes k+2+l}.
$$
We introduce the notation ${}^k\!\partial_j^l={\rm id}_{\mathbb C\langle \underline{\bf x}_r\rangle}^{\otimes k}\otimes\partial_j\otimes
{\rm id}_{\mathbb C\langle \underline{\bf x}_r\rangle}^{\otimes l}$, so that $\partial_j={}^0\!\partial_j^0$.

A direct verification shows that the free difference quotient thus defined obeys the usual chain rule $\partial_j(P(Q_1(\underline{\bf x}_r),\dots,Q_r(\underline{\bf x}_r)
))=\sum_{k=1}^r(\partial_k P)(Q_1(\underline{\bf x}_r),\allowbreak\dots,Q_r(\underline{\bf x}_r))\circ \partial_jQ_k(\underline{\bf x}_r)$. 
The symbol $\circ$ refers to the linear extension of the operation $(A\otimes B)\circ(C\otimes D)=AC\otimes DB$ - this is the natural multiplication
of the algebra $\mathbb C\langle\underline{\bf x}_r\rangle\otimes\mathbb C\langle\underline{\bf x}_r\rangle^{\rm op}$ -  and is motivated by the view of
$\mathbb C\langle\underline{\bf x}_r\rangle\otimes\mathbb C\langle\underline{\bf x}_r\rangle^{\rm op}$ as a space of linear maps from 
$\mathbb C\langle\underline{\bf x}_r\rangle$ to itself: $(A\otimes B)(X)=AXB$. (As the reader has surely guessed, $\mathbb C\langle\underline{\bf x}_r\rangle^{\rm op}$
denotes the {\em opposite } algebra structure on $\mathbb C\langle\underline{\bf x}_r\rangle$: one has $\mathbb C\langle\underline{\bf x}_r\rangle^{\rm op}=\mathbb C\langle
\underline{\bf x}_r\rangle$ as complex vector spaces, but the multiplication on $\mathbb C\langle\underline{\bf x}_r\rangle^{\rm op}$ is given by $P^{\rm op}\cdot Q^{\rm op}=
(QP)^{\rm op}$. The superscript op is only relevant when the multiplicative structure of our algebra is considered, otherwise there is no difference between $P$ and $P^{\rm op}$.) 
This motivates us to introduce for any $P\in\mathbb C\langle \underline{\bf x}_r\rangle$ the operator
$$
\mathsf{ev}_P\colon\mathbb C\langle \underline{\bf x}_r\rangle\otimes\mathbb C\langle \underline{\bf x}_r\rangle\to\mathbb C\langle \underline{\bf x}_r\rangle
$$
by the requirement that it is linear and that $\mathsf{ev}_P(A\otimes B)=A(\underline{\bf x}_r)P(\underline{\bf x}_r)B(\underline{\bf x}_r)$. (One uses the
identification of $\mathbb C\langle \underline{\bf x}_r\rangle\otimes\mathbb C\langle \underline{\bf x}_r\rangle\subseteq\mathscr L(\mathbb C\langle \underline{\bf x}_r\rangle,
\mathbb C\langle \underline{\bf x}_r\rangle),$ where, as discussed above, the linear map is given by the linear extension of $A\otimes B\mapsto[X\mapsto AXB]$.) A 
``multilinear'' version is 
$$
\mathsf{ev}_{P_1,\dots,P_s}\colon\underbrace{\mathbb C\langle\underline{\bf x}_r\rangle\otimes\cdots\otimes\mathbb C\langle\underline{\bf x}_r\rangle}_{s+1\text{ times}}
\to\mathbb C\langle \underline{\bf x}_r\rangle,
$$
given by $\mathsf{ev}_{P_1,\dots,P_s}(A_1\otimes\cdots\otimes A_{s+1})=A_1(\underline{\bf x}_r)P_1(\underline{\bf x}_r)A_2(\underline{\bf x}_r)\cdots P_s(\underline{\bf x}_r)
A_{s+1}(\underline{\bf x}_r)$. We will seldom need any other case except $P=1$ / $P_1=\cdots=P_s=1$. In this last case we write $\mathsf{ev}_{1,\dots,1}=\mathsf{ev}_{1^s}$.

The map 
$$
\mathsf{flip}\colon
\mathbb C\langle\underline{\bf x}_r\rangle\otimes\mathbb C\langle\underline{\bf x}_r\rangle\to\mathbb C\langle\underline{\bf x}_r\rangle\otimes\mathbb C\langle\underline{\bf x}_r\rangle
$$
is defined as the linear extension of $\mathsf{flip}(A\otimes B)=B\otimes A$. Sometimes we will need to write this map on different algebras:
$$
\mathsf{flip}\colon
\mathbb C\langle\underline{\bf x}_r\rangle\otimes\mathbb C\langle\underline{\bf y}_t\rangle\to\mathbb C\langle\underline{\bf y}_t\rangle\otimes\mathbb C\langle\underline{\bf x}_r\rangle
$$
given by the linear extension of $\mathsf{flip}(A(\underline{\bf x }_r)\otimes B(\underline{\bf y}_t))=B(\underline{\bf y}_t)\otimes A(\underline{\bf x}_r)$.
We can also define ${}^k\mathsf{flip}^l={\rm id}_{\mathbb C\langle \underline{\bf x}_r\rangle}^{\otimes k}\otimes\mathsf{flip}\otimes
{\rm id}_{\mathbb C\langle \underline{\bf x}_r\rangle}^{\otimes l}$, where
$$
{}^k\mathsf{flip}^l\colon\mathbb C\langle\underline{\bf x}_r\rangle^{\otimes k+2+l}\to\mathbb C\langle\underline{\bf x}_r\rangle^{\otimes k+2+l}.
$$
With this notation, $\mathsf{flip}={}^0\mathsf{flip}^0$. Sometimes it might be convenient to write 
$(\mathsf{flip}\otimes\mathsf{flip})(A\otimes B\otimes C\otimes D)=B\otimes A\otimes D\otimes C$ for ${}^0\mathsf{flip}^2\circ{}^2
\mathsf{flip}^0$, or variations on this theme, but we will try to do so only when the risk of confusion is small.

Finally, we define
$$
\#\colon\mathbb C\langle \underline{\bf x}_r\rangle\otimes\mathbb C\langle \underline{\bf y}_t\rangle\to\mathbb C\langle \underline{\bf x}_r;\underline{\bf y}_t\rangle
$$
by the linear extension of $\#(A(\underline{\bf x}_r)\otimes B(\underline{\bf y}_t))=A(\underline{\bf x}_r)B(\underline{\bf y}_t)$. 
The main difference between $\#$ and $\mathsf{ev}_1$ is that the latter does not increase the number of variables; $\#$ is an embedding of the tensor product in 
the algebraic free product, 
$$
\#\colon\mathbb C\langle\underline{\bf x}_r\rangle\otimes\mathbb C\langle\underline{\bf y}_t\rangle\to\mathbb C\langle\underline{\bf x}_r\rangle*\mathbb C\langle\underline{\bf y}_t
\rangle\simeq\mathbb C\langle\underline{\bf x}_r;\underline{\bf y}_t\rangle,$$
the algebra of the free semigroup with $r+t$ free generators, while $\mathsf{ev}_1$ is essentially the multiplication operation. We generally denote
$$
\underbrace{\#*\cdots*\#}_{s-1\text{ times}}\colon\underbrace{\mathbb C\langle\underline{\bf x}_r\rangle\otimes\cdots\otimes\mathbb C
\langle\underline{\bf x}_r\rangle}_{s\text{ times}}\to\underbrace{\mathbb C\langle\underline{\bf x}_r\rangle*\cdots*\mathbb C\langle\underline{\bf x}_r\rangle}_{s\text{ times}};
$$
As usual, we view $\mathbb C\langle \underline{\bf x}_r\rangle*\mathbb C\langle\underline{\bf x}_r\rangle*\cdots*\mathbb C\langle\underline{\bf x}_r\rangle\simeq
\mathbb C\langle \underline{\bf x}_r^1;\underline{\bf x}_r^2;\dots;\underline{\bf x}_r^s\rangle$. Then $\#*\#*\cdots*\#$ is expressed as
$(\#*\#*\cdots*\#)(A_1(\underline{\bf x}_r)\otimes A_2(\underline{\bf x}_r)\otimes\cdots\otimes A_s(\underline{\bf x}_r))=
A_1(\underline{\bf x}_r^1)A_2(\underline{\bf x}_r^2)\cdots A_s(\underline{\bf x}_r^s)$. Sometimes it will be convenient to write
$\#^{s-1}$ for $\underbrace{\#*\#*\cdots*\#}_{s-1\text{ times}}$, $s\ge2$. (Note that one can define a priori a more general version of $\#$, which 
we could denote by $\mathsf{ev}_{P(\underline{\bf x}_r,\underline{\bf y}_t)}\circ\#$ for a fixed $P$, given by $(\mathsf{ev}_{P(\underline{\bf x}_r,\underline{\bf y}_t)}\circ\#)
(A\otimes B)=A(\underline{\bf x}_r)P(\underline{\bf x}_r,\underline{\bf y}_t)B(\underline{\bf y}_t)$, but such considerations will not be relevant in our paper.)

\begin{remark}
There are algebras of operators (bounded or not) which are star-isomorphic to $\mathbb C\langle\underline{\bf x}_r\rangle$.
For example, the unital star-algebra ${}^*\mathrm{Alg}\{1,{\bf s}_1,\dots,{\bf s}_r\}$ generated by an $r$-tuple of standard free semicircular
random variables ${\bf s}_1,\dots,\allowbreak{\bf s}_r\in\mathcal A$ and the unit $1$ of $\mathcal A$ is star-isomorphic to $\mathbb C\langle\underline{\bf x}_r\rangle$
via the linear extension of the map $\alpha{\bf s}_{i_1}{\bf s}_{i_2}\cdots{\bf s}_{i_n}\mapsto\alpha\mathbf x_{i_1}\mathbf x_{i_2}\cdots\mathbf x_{i_n}$.
This is in fact true for any free $r$-tuple of diffuse selfadjoint random variables, but not only. We call {\em algebraically free} any $r$ tuple of 
selfadjoint random variables which generate a unital star-subalgebra of a von Neumann algebra containing the unit of the von Neumann algebra, which is star-isomorphic 
to $\mathbb C\langle\underline{\mathbf x}_r\rangle$. All of the above operations extend via this isomorphism to algebras generated by
algebraically free operators. Moreover, in a sufficiently large von Neumann algebra (for instance one that includes the free group factor $L(\mathbb F_\infty)$),
the set of algebraically free $r$-tuples of selfadjoints is norm-dense in the following sense: if $t_1,\dots,t_r$ are not algebraically free, then one picks a standard free semicircular
system $(\mathbf s_1,\dots,\mathbf s_r)$, free from $(t_1,\dots,t_r)$. For any $\varepsilon>0$, the tuple $(t_1+\varepsilon\mathbf s_1,\dots,t_r+\varepsilon\mathbf s_r)$ is 
algebraically free (see \cite{Coalg}), and converges in norm to $(t_1,\dots,t_r)$ as $\varepsilon\searrow0$. We make this assumption about $\mathcal A$.
\end{remark}

\begin{remark}
With the exception of the density argument in the above remark, all considerations in this subsection were purely algebraic. 
The existence of a star (adjoint) operation is not a requirement for either of $\partial_j,\mathsf{flip},\mathsf{ev}_{\cdot},\#$ to exist/be well-defined.
Moreover, non-selfadjoint elements in a von Neumann algebra may as well be algebraically free, by requiring only that the unital {\em algebra} (not star-algebra) 
they generate is {\em isomorphic} (not star-isomorphic) to the unital algebra generated by $r$ noncommuting indeterminates.
\end{remark}

For future reference, we write down here also how $\partial_j$ acts on $\mathbb C\langle\underline{\bf x}_r\rangle^{\rm op}$.
For a given monomial, one has $(\mathbf{x}_{i_1}\mathbf{x}_{i_2}\cdots\mathbf{x}_{i_n})^{\rm op}=\mathbf{x}_{i_n}^{\rm op}\cdots\mathbf{x}_{i_2}^{\rm op}
\mathbf{x}_{i_1}^{\rm op}$. When we write a polynomial $P\in\mathbb C\langle\underline{\bf x}_r\rangle$ as an element 
$P^{\rm op}\in\mathbb C\langle\underline{\bf x}_r\rangle^{\rm op}$, we mean that each of its monomials is written as just above: 
if $P=\sum_k\alpha_k \mathbf{x}_{i_1^k}\mathbf{x}_{i_2^k}\cdots\mathbf{x}_{i_{n_k}^k}$,
then $P^{\rm op}=\sum_k\alpha_k(\mathbf{x}_{i_1^k}\mathbf{x}_{i_2^k}\cdots\mathbf{x}_{i_{n_k}^k})^{\rm op}
=\sum_k\alpha_k \mathbf{x}_{i_{n_k}^k}^{\rm op}\cdots\mathbf{x}_{i_2^k}^{\rm op}\mathbf{x}_{i_1^k}^{\rm op}.$
When acting on $\mathbb C\langle\underline{\bf x}_r\rangle^{\rm op}$, the free difference quotient $\partial_j$ ``differentiates in the direction'' of $\mathbf{x}_j^{\rm op}$.
Thus, $\partial_j(\mathbf{x}_{i_1}\mathbf{x}_{i_2}\cdots\mathbf{x}_{i_n})^{\rm op}=
\partial_j(\mathbf{x}_{i_n}^{\rm op}\cdots\mathbf{x}_{i_2}^{\rm op}\mathbf{x}_{i_1}^{\rm op})=\sum_{k\colon i_k=j}
\mathbf{x}_{i_n}^{\rm op}\cdots\mathbf{x}_{i_{k+1}}^{\rm op}\otimes\mathbf{x}_{i_{k-1}}^{\rm op}\cdots\mathbf{x}_{i_1}^{\rm op}\allowbreak
=\sum_{k\colon i_k=j}(\mathbf{x}_{i_{k+1}}\cdots\mathbf{x}_{i_n})^{\rm op}\otimes(\mathbf{x}_{i_1}\cdots\mathbf{x}_{i_{k-1}})^{\rm op}$.
We recognize here the $\mathsf{flip}$ operation applied to $\partial_j(\mathbf{x}_{i_1}\mathbf{x}_{i_2}\cdots\mathbf{x}_{i_n}).$
That is, we have shown that $\partial_j$ acts on $\mathbb C\langle\underline{\bf x}_r\rangle^{\rm op}$ as $\mathsf{flip}\circ\partial_j$ acts on 
$\mathbb C\langle\underline{\bf x}_r\rangle.$ Let us write down the Leibniz rule in this context. 
If $P,Q\in\mathbb C\langle\underline{\bf x}_r\rangle$ are two given polynomials, then $\partial_j(PQ)^{\rm op}
=\partial_j(Q^{\rm op}P^{\rm op})=(\partial_j(Q^{\rm op}))(1\otimes P^{\rm op})+(Q^{\rm op}\otimes1)(\partial_j(P^{\rm op}))
=(\mathsf{flip}\circ\partial_jQ)^{\rm op}(1\otimes P)^{\rm op}+(Q\otimes1)^{\rm op}(\mathsf{flip}\circ\partial_jP)^{\rm op}
=[(1\otimes P)(\mathsf{flip}\circ\partial_jQ)]^{\rm op}+[(\mathsf{flip}\circ\partial_jP)(Q\otimes1)]^{\rm op}.$

The free difference quotient extends to a much larger family of noncommutative functions $f$ that includes polynomial functions, with the meaning 
of $\partial_jf$ being quite obvious (the reader can easily verify this fact for ``most'' analytic functions that are locally norm limits of 
sequences of polynomials - since the notion of analyticity may become rather complicated in infinitely dimensional spaces, we do not
provide details here, but refer the curious reader to, for instance, \cite{Dineen}, for a thorough discussion of analytic functions on 
Banach, and more generally, locally convex, topological vector spaces). Some specific cases have already been considered in
detail, and we mention here 
\begin{equation}\label{exps}
\partial_je^{P}=\int_0^1e^{aP\otimes1}(\partial_jP)e^{1\otimes(1-a)P}\,{\rm d}a,\quad P\in\mathbb C\langle\underline{\mathbf x}_r\rangle,
\end{equation}
from \cite{P}, and 
\begin{eqnarray}\label{invs}
\partial_j(z-P)^{-1}= \left((z-P)^{-1}\otimes1\right)(\partial_jP)\left(1\otimes(z-P)^{-1}\right),\quad P\in\mathbb C\langle\underline{\mathbf x}_r\rangle,
\end{eqnarray}
from \cite{Coalg} (see, in particular, Sections 1 and 3), which, again, the reader can verify by hand. Applications of the chain rule and sum/product rules
extend these formulas to much larger classes of functions. We do not elaborate any further because these facts will hardly play any role in our paper
in this shape, but will become very important under the shape of the computation rules for the difference-differential operator, as it is seen below.

\subsection{Noncommutative functions}\label{sec:ncf}
In this subsection we  follow very closely \cite{KVV}, to which we refer for a thorough and detailed introduction to the subject. 
We do not introduce noncommutative functions in maximum generality, but in a context only as general as needed for our purposes. Thus, given $d\in\mathbb N,d\ge1$, and
a $C^*$-algebra $\mathcal B$, a {\em noncommutative set} over $\mathcal B^d$ is a family $\Omega=\coprod_{n\in\mathbb N}\Omega_n$
such that 
\begin{enumerate}
\item $\Omega_n\subseteq M_n(\mathcal B)^d$ for all $n\in\mathbb N$;
\item If $\underline{X}_d\in\Omega_n,\underline{Y}_d\in\Omega_m$, then $\underline{X}_d\oplus\underline{Y}_d=\left(
\begin{bmatrix} X_1 & 0 \\ 0 & Y_1\end{bmatrix},\dots,\begin{bmatrix} X_d & 0 \\ 0 & Y_d\end{bmatrix}\right)\in\Omega_{n+m},m,n\in\mathbb N$;
\item If $\underline{X}_d\in\Omega_n$ and $U\in M_n(\mathbb C)$ is a unitary, then 
$U\underline{X}_dU^*=(UX_1U^*,\dots,\allowbreak UX_dU^*)\in \Omega_n$.
\end{enumerate}
Given a (possibly, but not necessarily, different) $C^*$-algebra $\tilde{\mathcal B}$, a {\em noncommutative function} is a family $\mathsf f=\{\mathsf f_n\}_{n\in\mathbb N}$ 
such that 
\begin{enumerate}
\item $\mathsf f_n\colon\Omega_n\to M_n(\tilde{\mathcal B})$ for all $n\in\mathbb N$;
\item If $\underline{X}_d\in\Omega_n,\underline{Y}_d\in\Omega_m$, then $\mathsf f_{n+m}(\underline{X}_d\oplus\underline{Y}_d)=
\mathsf f_n(\underline{X}_d)\oplus\mathsf f_m(\underline{Y}_d),m,n\in\mathbb N$;
\item If $\underline{X}_d\in\Omega_n$ and $S\in M_n(\mathbb C)$ is an invertible matrix such that $S\underline{X}_dS^{-1}\in\Omega_n$, 
then $\mathsf f_n(S\underline{X}_dS^{-1})=S\mathsf f_n(\underline{X}_d)S^{-1}$.
\end{enumerate}
We call $n$ the {\em level} of the noncommutative function/set. The following example will be useful.
\begin{example}\label{ExAFC}
Consider a simply connected open set $G\subseteq\mathbb C$. Then any analytic function $f\colon G\to\mathbb C$ is the first level of 
a noncommutative function taking values in $\mathbb C_{\rm nc}:=\coprod_n M_n(\mathbb C)$. Its natural domain of definition is 
$\coprod_n\{Z\in M_n(\mathbb C)\colon{\rm sp}(Z)\subseteq G\}$ and it is defined via the analytic functional calculus:
$$
f_n(Z)=\frac{1}{2\pi i}\int_\gamma f(\zeta)(\zeta I_n-Z)^{-1}\,{\rm d}\gamma(\zeta),
$$
for some simple closed curve $\gamma$ in $G$ and surrounding exactly once the spectrum ${\rm sp}(Z)$ of $Z$.
This is easily seen to be an extension of the polynomial evaluation.
\end{example}

An important property of noncommutative functions is that the derivative at a given level $n$ is often recoverable from evaluation at level $2n$. Specifically: if 
$\underline{X}_d\in\Omega_n,\underline{Y}_d\in\Omega_m$, and $\underline{B}_d\in M_{n\times m}(\mathcal B)^d$ is such that $\begin{bmatrix}\underline{X}_d&\underline{B}_d\\
0&\underline{Y}_d\end{bmatrix}\in\Omega_{n+m}$ and $\mathsf f$ is locally bounded on slices, then $\mathsf f_{n+m}$ on upper triangular matrices is of the form 
$$
\mathsf f_{n+m}\left(\begin{bmatrix} \underline{X}_d & \underline{B}_d \\ 0 & \underline{Y}_d\end{bmatrix}\right)
=\begin{bmatrix} \mathsf f_n(\underline{X}_d) & \Delta\mathsf f_{n,m}(\underline{X}_d,\underline{Y}_d)(\underline{B}_d) \\ 0 & \mathsf f_m(\underline{Y}_d)\end{bmatrix}.
$$
If there is an open set around zero such that 
$\begin{bmatrix} \underline{X}_d & \underline{B}_d \\ 0 & \underline{Y}_d\end{bmatrix}\in\Omega_{n+m}$ for all $\underline{B}_d$ in that set, then
$M_{n\times m}(\mathcal B)^d\ni\underline{B}_d\mapsto\Delta\mathsf f_{n,m}(\underline{X}_d,\underline{Y}_d)(\underline{B}_d)\in M_{n\times m}(\mathcal B)$
is a $\mathbb C$-linear map. If $m=n$, then $\Delta\mathsf f_{n,n}(\underline{X}_d,\underline{Y}_d)(\underline{X}_d-\underline{Y}_d)
=\mathsf f_n(\underline{X}_d) -\mathsf f_n(\underline{Y}_d)$, and if $\underline{X}_d=\underline{Y}_d$, then 
$\Delta\mathsf f_{n,n}(\underline{X}_d,\underline{X}_d)(\underline{B}_d)=\mathsf f_{n}'(\underline{X}_d)(\underline{B}_d)$, the usual 
Fr\'echet derivative of the Banach space-valued map $\mathsf f_n$  (the dependence in $\underline{X}_d$ and $\underline{Y}_d$
of $\Delta\mathsf f_{n,m}(\underline{X}_d,\underline{Y}_d)(\underline{B}_d)$ is of a nature that is very similar to the dependence of $\mathsf f_{n}$ on $\underline{X}_d$, or of
$\mathsf f_{m}$ on $\underline{Y}_d$). A rather spectacular fact that follows from this is that locally bounded noncommutative functions defined on sets 
that are ``thick'' enough (open sets, for instance) are automatically Fr\'echet analytic (see \cite[Corollary 7.6]{KVV}). The linear operator $\Delta\mathsf f_{n,m}(\underline{X}_d,
\underline{Y}_d)$ is called the (first) {\em difference-differential} operator.

This is important for us for several reasons. First, one easily notes that 
any $\underline{B}_d=(B_1,B_2,\dots,B_d)\in M_{n\times m}(\mathcal B)^d$ is written as $\underline{B}_d=B_1e_1+B_2e_2+\cdots+B_de_d$,
with $e_j=(0,\dots,1\otimes I_{m},\dots,0)$ having the identity on position $j$ and zero everywhere else. Thus, it makes sense to define 
$$
\Delta_j\mathsf f_{n,m}(\underline{X}_d,\underline{Y}_d)(B)=\Delta\mathsf f_{n,m}(\underline{X}_d,\underline{Y}_d)(Be_j)
=\Delta\mathsf f_{n,m}(\underline{X}_d,\underline{Y}_d)(0,\dots,B,\dots,0),
$$
the $j^{\rm th}$ partial difference-differential operator $\Delta_j\mathsf f_{n,m}(\underline{X}_d,\underline{Y}_d)\colon M_{n\times m}(\mathcal B)\to M_{n\times m}(\mathcal B)$.
If $m=n$ and $\underline{X}_d=\underline{Y}_d$, then $\Delta_j\mathsf f_{n,n}(\underline{X}_d,\underline{X}_d)$ is just the classical partial derivative in the $j^{\rm th}$ coordinate.
Second, one may repeat the above for larger matrices:
\begin{eqnarray*}
\lefteqn{\mathsf f_{n+m+p}\left(\begin{bmatrix} \underline{X}_d & \underline{B}_d & 0 \\ 0 & \underline{Y}_d & \underline{C}_d\\0 & 0 & \underline{Z}_d\end{bmatrix}\right)=}\\
&&\begin{bmatrix}
\mathsf f_n(\underline{X}_d)&\Delta\mathsf f_{n,m}(\underline{X}_d,\underline{Y}_d)(\underline{B}_d)&\Delta^2\mathsf f_{n,m,p}(\underline{X}_d,\underline{Y}_d,\underline{Z}_d)
(\underline{B}_d,\underline{C}_d)\\ 0 & \mathsf f_m(\underline{Y}_d) & \Delta\mathsf f_{m,p}(\underline{Y}_d,\underline{Z}_d)(\underline{C}_d)\\
0 & 0 & \mathsf f_p(\underline{Z}_d) \end{bmatrix},\quad
\end{eqnarray*}
where $(\underline{B}_d,\underline{C}_d)\mapsto\Delta^2\mathsf f_{n,m,p}(\underline{X}_d,\underline{Y}_d,\underline{Z}_d)(\underline{B}_d,\underline{C}_d)$ is a {\em bi}linear
correspondence from $M_{n\times m}(\mathcal B)^d\allowbreak\times M_{m\times p}(\mathcal B)^d$ to $M_{n\times p}(\mathcal B)$, and so on, as one considers larger and larger
matrices (again, $\Delta^2\mathsf f_{n,n,n}(\underline{X}_d,\underline{X}_d,\underline{X}_d)$ is the classical second derivative of $\mathsf f_n$). Generally,
\begin{eqnarray}
\lefteqn{\!\!\!\mathsf f_{m_1+\cdots+m_{n+1}}\left(\begin{bmatrix} \underline{X}^{(1)}_d & \underline{B}_d^{(1)} &\cdots & 0 & 0 \\ 
0 & \underline{X}_d^{(2)} & \cdots & 0 & 0\\ 
\vdots & \vdots & \ddots & \vdots &\vdots\\
0 & 0 & \cdots & \underline{X}_d^{(n)} & \underline{B}_d^{(n)}\\
0 & 0 & \cdots & 0 & \underline{X}_d^{(n+1)}\end{bmatrix}\right)=}\nonumber\\
&&\!\!\!\begin{bmatrix}
\mathsf f_{m_1}(\underline{X}^{(1)}_d)& \cdots &  \Delta^{n}\mathsf f_{m_1,\dots,m_{n+1}}(\underline{X}^{(1)}_d,\dots,\underline{X}^{(n+1)}_d)(\underline{B}^{(1)}_d,\dots,
\underline{B}^{(n)}_d)\\ 
\vdots & \ddots & \vdots \\
0 & \cdots & \mathsf f_p(\underline{X}^{(n+1)}_d) \end{bmatrix}\!\!.\label{DiffDiff}
\end{eqnarray}
One observes that $\Delta^n\mathsf f(\underline{X}^{(1)}_d,\dots,\underline{X}^{(n+1)}_d)$ makes sense for tuples $\underline{X}^{(j)}_d$
of all sizes in the domain of the initial $\mathsf f$, and, based on the properties of $\mathsf f$ itself when defined on
block matrices as above, the correspondence $(\underline{X}^{(1)}_d,\dots,\underline{X}^{(n+1)}_d)\mapsto
\Delta^n\mathsf f(\underline{X}^{(1)}_d,\dots,\underline{X}^{(n+1)}_d)$ must have a ``nice'' behavior, similar to that of $\mathsf f$.
Indeed, for any fixed $n\ge1$, $\Delta^n\mathsf f$ is a map on sets $\coprod_{m_1,\dots,m_{n+1}\in\mathbb N}\Omega_{m_1}\times\cdots\times\Omega_{m_{n+1}}$ taking 
values in the space of $n$-linear continuous maps from $M_{m_1\times m_2}(\mathcal B)\times M_{m_2\times m_3}(\mathcal B)\times\cdots\times M_{m_n\times m_{n+1}}(\mathcal B)
$ to $M_{m_1\times m_{n+1}}(\mathcal B),m_1,\dots,m_{n+1}\in\mathbb N$. This is a particular case of a noncommutative map of order $n$ (with $\mathsf f$ being a 
map of order zero), and such maps have properties that are very similar to those of $\mathsf f$, including accepting the application of the free difference quotient in
any of its $n+1$ tuples of variables, which they ``carry'' with them in upper triangular matrices in a manner similar to that of $\mathsf f$, and analyticity under
very mild conditions of local boundedness. Conceptually, this is quite similar to the case of noncommutative functions of order zero, but the notations 
become quickly very cumbersome, so we do not provide any details here, but refer the reader to \cite[Chapter 3]{KVV} - particularly Section 3.2 - and
\cite[Section 7.4]{KVV} for a clear and comprehensive presentation of the properties of analytic higher order noncommutative functions.

The existence and properties of the difference differential operator
allows one to write power series expansions for noncommutative functions: according to \cite[Theorems 7.2 and 7.4]{KVV}, if $\mathsf f$ is locally bounded (which will
always be the case for the functions we deal with), then 
\begin{align}\label{class}
\frac{1}{K!}\frac{d^K}{dt^K}&\mathsf f_n\left.(\underline{Y}_d+t\underline{Z}_d)\right|_{t=0}
=\Delta^K\mathsf f_{n,\dots,n}(\underbrace{\underline{Y}_d,\dots,\underline{Y}_d}_{K+1\text{ times}})(\underbrace{\underline{Z}_d,\dots,\underline{Z}_d}_{K\text{ times}}),\\
\mathsf f_{n}\left(\underline{X}_d\right)&=\sum_{l=0}^\infty\Delta^l\mathsf f_{n,\dots,n}(\underbrace{\underline{Y}_d,\dots,\underline{Y}_d}_{l+1\text{ times}})
(\underbrace{\underline{X}_d-\underline{Y}_d,\dots,\underline{X}_d-\underline{Y}_d}_{l\text{ times}}),\label{TT}
\end{align}
where \eqref{class} happens for all $K,n\in\mathbb N$, $\underline{Y}_d\in\Omega_n$, $\underline{Z}_d\in M_{n}(\mathcal B)^d$ (recall that each $\Omega_n$
is assumed now to be open in the usual, norm topology of $M_{n}(\mathcal B)^d$), and for any $\varepsilon>0$ and any open ball $\Upsilon\subseteq\Omega_n$ centered at 
$\underline{Y}_d$, \eqref{TT} converges absolutely and uniformly on the set 
$$
\Upsilon_\varepsilon:=\{\underline{X}_d\in\Upsilon\colon\underline{Y}_d+(1+\varepsilon)(\underline{X}_d-\underline{Y}_d)\in\Upsilon\}:
$$
one has 
\begin{equation}\label{speed}
\sum_{l=0}^\infty\sup_{\underline{X}_d\in\Upsilon_\varepsilon}\big\|\Delta^l\mathsf f_{n,\dots,n}(\underbrace{\underline{Y}_d,\dots,\underline{Y}_d}_{l+1\text{ times}})
(\underbrace{\underline{X}_d-\underline{Y}_d,\dots,\underline{X}_d-\underline{Y}_d}_{l\text{ times}})\big\|_{M_n(\mathcal B)^d}<\infty.
\end{equation}
Following \cite{KVV}, we refer to \eqref{TT} as the {\em Taylor-Taylor series expansion of $\mathsf f_n$ around $\underline{Y}_d$}. An essential point for us about this series 
development is that its terms $\Delta^l\mathsf f_{n,\dots,n}(\underline{Y}_d,\allowbreak\dots,\underline{Y}_d)(\underline{X}_d-\underline{Y}_d,\dots,\underline{X}_d-
\underline{Y}_d)$, $n,l\in\mathbb N$, determine uniquely the function $\mathsf f$. As for the free difference quotient, the difference-differential calculus obeys the
same rules (linearity, Leibniz rule, and chain rule) that usual derivatives obey - see \cite[2.3.2, 2.3.4, 2.3.6]{KVV}. Similar power series expansions are available for higher order
noncommutative functions, as shown in \cite[Section 7.4]{KVV}.

In our case, $\mathcal B$ will be a finite von Neumann algebra, possibly finite dimensional (possibly just $\mathbb C$).

\subsection{Relations between the difference-differential operator and the free difference quotient}\label{Sec:DD-FDQ}
We have seen in Sections \ref{pol} and \ref{sec:ncf} two perspectives on noncommutative functions and  derivatives. In this section, we intend to unify them to some extent. For
clarity, let us begin by considering noncommutative functions on open subsets of $\mathbb C_{\rm nc}^r=\coprod_nM_n(\mathbb C)^r$. Given a monomial $M\in\mathbb C\langle
\underline{\mathbf x}_r\rangle$, one may view it as a noncommutative function simply by performing evaluations on $r$-tuples of $n\times n$ complex matrices for all $n$. Thus, for 
an $M=\mathbf x_{i_1}\mathbf x_{i_2}\cdots\mathbf x_{i_m},i_1,\dots,i_m\in\{1,\dots,r\}$, one has $\partial_jM=\sum_{k:i_k=j}\mathbf x_{i_1}\cdots\mathbf x_{i_{k-1}}
\otimes\mathbf x_{i_{k+1}} \cdots\mathbf x_{i_m}$, with the convention that the empty word is the unit 1 (that is, if, say, $i_m=j$, then the last summand is
$\mathbf x_{i_1}\cdots\mathbf x_{i_{m-1}}\otimes1$). Evaluation on an $r$-tuple of matrices $(X_1,\dots,X_r)$ yields
$(\partial_jM)(X_1,\dots,X_r)=\sum_{k:i_k=j} X_{i_1}\cdots X_{i_{k-1}}\otimes X_{i_{k+1}} \cdots X_{i_m}\allowbreak\in M_n(\mathbb C)\otimes M_n(\mathbb C).$
However, $M_n(\mathbb C)\otimes M_n(\mathbb C)$ identifies naturally with the space of linear maps from $M_n(\mathbb C)$ to itself via
$$
M_n(\mathbb C)\otimes M_n(\mathbb C)\ni\sum_jA_j\otimes B_j\mapsto\left[C\mapsto\sum_jA_jCB_j\right]\in\mathscr L(M_n(\mathbb C), M_n(\mathbb C)).
$$
(One recognizes above $\mathsf{ev}_C\left(\sum_jA_j\otimes B_j\right)$.) The identification is bijective, and if one considers the opposite algebra structure on the second 
tensor, it is also an algebra isomorphism $M_n(\mathbb C)\otimes M_n(\mathbb C)^{\rm op}\simeq\mathscr L(M_n(\mathbb C), M_n(\mathbb C))$.
This allows one to easily identify the free difference quotient and the difference-differential operator: simply observe that
\begin{eqnarray*}
\lefteqn{M\left(\begin{bmatrix}X_1 & 0 \\ 0 & X_1 \end{bmatrix},\dots,\begin{bmatrix}X_j & C \\ 0 & X_j \end{bmatrix},\dots,\begin{bmatrix}X_r & 0 \\ 0 & X_r \end{bmatrix}\right)}\\
&=&\begin{bmatrix}X_{i_1}&\delta_{i_1,j}C\\0&X_{i_1}\end{bmatrix}\begin{bmatrix}X_{i_2}&\delta_{i_2,j}C\\0&X_{i_2}\end{bmatrix}\cdots
\begin{bmatrix}X_{i_m}&\delta_{i_m,j}C\\0&X_{i_m}\end{bmatrix}\\
&=&\begin{bmatrix}X_{i_1} X_{i_2}\cdots X_{i_m}&\sum_{k:i_k=j} X_{i_1}\cdots X_{i_{k-1}}CX_{i_{k+1}} \cdots X_{i_m}\\0&X_{i_1} X_{i_2}\cdots X_{i_m}\end{bmatrix},
\end{eqnarray*}
which shows that on monomial functions defined on $\mathbb C_{\rm nc}^r$ the above isomorphism identifies the difference-differential operator 
and the free difference quotient when evaluated on $r$-tuples of complex matrices of all sizes. The extension to convergent power series is performed by applying this
procedure term-by-term. Given that both resolvents of polynomials and exponentials of polynomials are noncommutative functions in the sense of \cite{KVV}
whose Taylor-Taylor series converge as described in \eqref{TT} and \eqref{speed}, it follows that the identification of free difference quotients and 
difference-differential operators described just above extends to the algebra generated by polynomials, exponentials of polynomials, 
resolvents of polynomials, and compositions of such (on the open noncommutative sets on which they are defined).

The only point in the above reasoning where the finite dimensionality of the domain of definition for the noncommutative functions we consider
came up is in the identification $M_n(\mathbb C)\otimes M_n(\mathbb C)^{\rm op}\simeq\mathscr L(M_n(\mathbb C), M_n(\mathbb C))$. However, for a given 
von Neumann algebra $\mathcal A$, one still has an inclusion $M_n(\mathcal A)\otimes M_n(\mathcal A)^{\rm op}\hookrightarrow\mathscr L(M_n(\mathcal A), M_n(\mathcal A))$
of algebras, in the sense that the correspondence $C\mapsto\sum_jA_jCB_j$ (finite sum) is continuous from $M_n(\mathcal A)$ to itself and
$$
M_n(\mathcal A)\otimes M_n(\mathcal A)^{\rm op}\ni\sum_jA_j\otimes B_j\mapsto\left[C\mapsto\sum_jA_jCB_j\right]\in\mathscr L(M_n(\mathcal A), M_n(\mathcal A))
$$
is an endomorphism of algebras. (In the above we look at the {\em algebraic} tensor product over $\mathbb C$ - no issues of continuity or closure occur - and the von Neumann 
algebra $\mathcal A$ is arbitrary.) This allows us to extend the identification of Voiculescu's free difference quotient and the difference-differential operator 
to a large class of noncommutative functions, but this time defined on noncommutative subsets of $\coprod_{n}M_n(\mathcal A)$ for an 
arbitrary von Neumann algebra $\mathcal A$. As the reader will notice in the following, the fact that we only need to consider the algebraic tensor product in 
the above simplifies greatly our work. The formal description of the above observation is the following: for any polynomial $P\in\mathbb C\langle\underline{\bf x}_r\rangle$,
any $n\in\mathbb N$, any algebraically free selfadjoints $(s_1,\dots,s_r)\in M_n(\mathcal A)^r$, and any $c\in M_n(\mathcal A)$, one has
$$
\mathsf{ev}_c\circ\partial_jP(s_1,\dots,s_r)=\Delta_jP(s_1,\dots,s_r;s_1,\dots,s_r)(c).
$$
Observe that, generally, if $\mathsf{f,g}$ are locally norm-bounded noncommutative functions that satisfy this hypothesis (namely that $\mathsf{ev}_c\circ\partial_j\mathsf f(s_1,
\dots,s_r)=\Delta_j\mathsf f(s_1,\dots,s_r;s_1,\dots,s_r)(c))$, then so does $\mathsf{fg}$, because of the Leibniz rule and the fact that $\partial_j$ is a derivation with respect to 
the $M_n(\mathcal A)$-bimodule structure of $M_n(\mathcal A)\otimes M_n(\mathcal A)$ - see \cite[Section 1.3]{Coalg}\footnote{We use Voiculescu's 
statement as phrased in \cite{Coalg} only for $p=1$, but with $\mathcal A$ replaced by $M_n(\mathcal A),n\in\mathbb N$.}. (That the statement holds for sums is trivial.)
It then follows that the above identification of $\partial_j$ and $\Delta_j$ holds in particular for rational functions and, more generally, for noncommutative functions that are
locally norm-limits of polynomials.

However, for the purpose of using Parraud's work \cite{P}, we need a slightly more general statement. Recall that $\partial_j\colon\mathbb C\langle\underline{\bf x}_r\rangle
\to\mathbb C\langle\underline{\bf x}_r\rangle\otimes\mathbb C\langle\underline{\bf x}_r\rangle$ takes values in a tensor product of two identical 
noncommutative polynomials algebras. This means that the evaluation of $\partial_jP$ on, say, $\mathcal A\otimes\mathcal A$ (may be any two algebras),
can be performed in different operators corresponding to the two tensor coordinates: given $r$-tuples $\underline{s}_r=(s_1,\dots,s_r)$ and $\underline{t}_r=(t_1,\dots,t_r)$ 
in $\mathcal A^r$, one can evaluate $(\partial_jP)(\underline{s}_r;\underline{t}_r)$ as well as $(\partial_jP)(\underline{s}_r;\underline{s}_r)$
(the reader may feel more comfortable writing $\partial_j\colon\mathbb C\langle\underline{\bf x}_r\rangle
\to\mathbb C\langle\underline{\bf x}_r\rangle\otimes\mathbb C\langle\underline{\bf y}_r\rangle$ instead).  Recalling the operation
$\#\colon\mathbb C\langle\underline{\bf x}_r\rangle\otimes
\mathbb C\langle\underline{\bf y}_r\rangle\to\mathbb C\langle\underline{\bf x}_r;\underline{\bf y}_r\rangle$, $\#(A(\underline{\bf x}_r)\otimes B(\underline{\bf y}_r))=
A(\underline{\bf x}_r)B(\underline{\bf y}_r)$ from Section \ref{pol} (with equal number of indeterminates $r=t$), this leads to 
\begin{eqnarray*}
\lefteqn{M\left(\begin{bmatrix}s_1 & 0 \\ 0 & t_1 \end{bmatrix},\dots,\begin{bmatrix}s_j & c \\ 0 & t_j \end{bmatrix},\dots,\begin{bmatrix}s_r & 0 \\ 0 & t_r \end{bmatrix}\right)}\\
&=&\begin{bmatrix}s_{i_1}&\delta_{i_1,j}c\\0&t_{i_1}\end{bmatrix}\begin{bmatrix}s_{i_2}&\delta_{i_2,j}c\\0&t_{i_2}\end{bmatrix}\cdots
\begin{bmatrix}s_{i_m}&\delta_{i_m,j}c\\0&t_{i_m}\end{bmatrix}\\
&=&\begin{bmatrix}s_{i_1} s_{i_2}\cdots s_{i_m}&\sum_{k:i_k=j} s_{i_1}\cdots s_{i_{k-1}}ct_{i_{k+1}} \cdots t_{i_m}\\0&t_{i_1} t_{i_2}\cdots t_{i_m}\end{bmatrix}
\end{eqnarray*}
for $M=\mathbf x_{i_1}\mathbf x_{i_2}\cdots\mathbf x_{i_m}$, and, as above, allows us to conclude that
$$
(\#\circ\partial_j\mathsf f)(\underline{s}_r;\underline{t}_r)=\Delta_j\mathsf f(\underline{s}_r;\underline{t}_r)(1),
$$
for functions $\mathsf f$ in the class previously mentioned.

The important (to us) consequence of the above (mostly tautological) argument is that the operations $\partial_j,\#,\mathsf{ev}_{\cdot}$ are translatable in terms of 
difference-differential operators, so that the nature of the noncommutative function we work with does not change. In particular,
if both the variables $s$ and the variables $t$ are in the domain of our function, then these operations become simply operations
between levels of {\em the same} noncommutative function.

\subsection{Domains for $(z-\mathcal S)^{-1}$}\label{domains}

As the reader remembers from Section \ref{HTSapproach}, we are concerned with resolvents of operators of the type 
$\mathcal{S}=\xi\otimes1\otimes 1+\sum_{i=1}^{r_1}(\gamma_i\otimes u_i\otimes 1+\gamma_i^*\otimes
u_i^*\otimes1)+\sum_{j=1}^{r_2}(\beta_j\otimes1\otimes {v_j}+\beta_j^*\otimes 1\otimes {v_j^*}),$
where $u_j,v_j$ are Cayley transforms of various free algebraically free selfadjoint variables. 
By choosing 
$$
r=\max\{r_1,r_2\},
$$ 
we may assume $r_1=r_2=r$ by simply setting the missing coefficients (be they $\gamma$ or $\beta$) to zero, so from
now on we make this assumption.
Since Cayley transforms of selfadjoints are unitaries, one has $u_i^*=u_i^{-1},v_i^*=v_i^{-1}$. We intend to
view $\mathcal S$ and its various transforms as noncommutative functions (hence analytic). Because of that, we cannot 
keep adjoints in the formula of $\mathcal S$, so that from now on, we agree that
\begin{eqnarray}
\lefteqn{\mathcal{S}=\xi\otimes1\otimes 1}\\\label{TheRealS}
&&\mbox{}+\sum_{i=1}^{r}(\gamma_i\otimes u_i\otimes 1+\gamma_i^*\otimes
u_i^{-1}\otimes1+\beta_i\otimes1\otimes {v_i}+\beta_i^*\otimes 1\otimes {v_i^{-1}}),\nonumber
\end{eqnarray}
where we remind the reader that $\xi=\xi^*,\gamma_i,\beta_i\in M_m(\mathbb C)$ are arbitrary, but fixed. Similarly,
$$
\mathcal S_u=\sum_{i=1}^{r}(\gamma_i\otimes u_i\otimes 1+\gamma_i^*\otimes
u_i^{-1}\otimes1),
$$
$$
\mathcal S_v=\sum_{i=1}^{r}(\beta_i\otimes1\otimes {v_i}+\beta_i^*\otimes 1\otimes {v_i^{-1}}),
$$
where $u_i$ and $v_i$ are simply placeholders for variables to be specified at a later time (in particular,
$\mathcal S=\xi\otimes1\otimes1+\mathcal S_u+\mathcal S_v$). Then $\mathcal S_u,\mathcal S_v,\mathcal S$ 
are selfadjoint whenever evaluated on unitary operators $u_i,v_i$, and $(u_i,v_i)$, respectively.
In the following, we restrict our variables $u_i,v_i$ to a very specific kind of operators, namely Cayley transforms of 
operators with imaginary part between $-1$ and $1$. Thus, for an arbitrary von Neumann algebra\footnote{$C^*$-algebra - or even less - would do.}
$\mathcal A$, we consider the set
$$
\mathcal I_r(\mathcal A)=\coprod_{n\ge1}\mathcal I_r(\mathcal A)_n=\coprod_{n\ge1}\left\{\underline{s}_r\!\in M_n(\mathcal A)^r\colon\!-1<\Im s_j<1,j\in\{1,\dots, r\} \right\}.
$$
The Cayley transform $\Psi(z)=\frac{z+i}{z-i}$ and its inverse $\Psi(z)^{-1}=\frac{z-i}{z+i}$ are both defined on operators whose
imaginary part is between $-1$ and 1. Indeed, if $-1<\Im s<1$, then $(s+i)^{-1}=(\Re s+i(1+\Im s))^{-1}
=(1+\Im s)^{-1/2}((1+\Im s)^{-1/2}\Re s(1+\Im s)^{-1/2}+i)^{-1}(1+\Im s)^{-1/2}$ and $(s-i)^{-1}=(\Re s-i(1-\Im s))^{-1}
=(1-\Im s)^{-1/2}((1-\Im s)^{-1/2}\Re s(1-\Im s)^{-1/2}-i)^{-1}(1-\Im s)^{-1/2}$. Since $\Im s>-1,$ one has $\Im s+1>0$, so that
$\Im s+1$ is invertible. As $(1+\Im s)^{-1/2}\Re s(1+\Im s)^{-1/2}$ is selfadjoint, $((1+\Im s)^{-1/2}\Re s(1+\Im s)^{-1/2}+i)^{-1}$
is bounded in norm by one. Thus, $\|(s+i)^{-1}\|\le\|(1+\Im s)^{-1}\|$. Similarly,  $1>\Im s\implies1-\Im s>0$ makes $(1-\Im s)^{-1}$ bounded as well,
and $\|(s-i)^{-1}\|\le\|(1-\Im s)^{-1}\|$. This implies that $\|\Psi(s)^\epsilon\|\le1+2\max\{\|(1-\Im s)^{-1}\|,\|(1+\Im s)^{-1}\|\},\epsilon\in\{\pm1\}.$
We also note that, regardless of $\mathcal A$, $\mathcal I_{r}(\mathcal A)$ is a noncommutative set. Indeed, if $\underline{s}_{r}
\in\mathcal I_{r}(\mathcal A)_n,\underline{t}_{r}\in\mathcal I_{r}(\mathcal A)_m$, then
$-1\otimes I_{n+m}=\begin{bmatrix} -1\otimes I_n & 0 \\ 0 & -1\otimes I_m \end{bmatrix}<\begin{bmatrix} \Im s_j & 0 \\ 0 & \Im t_j \end{bmatrix}<
\begin{bmatrix} 1\otimes I_n & 0 \\ 0 & 1\otimes I_m \end{bmatrix}=1\otimes I_{n+m}$, and if $U\in M_n(\mathbb C)$ is unitary, then
$-1\otimes I_n=-(1\otimes U)(1\otimes I_n)(1\otimes U^*)<(1\otimes U)\Im s_j(1\otimes U^*)<(1\otimes U)(1\otimes I_n)(1\otimes U^*)=1\otimes I_n$ for all
$1\le j\le r$, so that $\underline{s}_{r}\oplus\underline{t}_{r}\in\mathcal I_{r}(\mathcal A)_{n+m}$
and $(1\otimes U)\underline{s}_{r}(1\otimes U^*)\in\mathcal I_{r}(\mathcal A)_n$.

Next we observe that the formula for $\mathcal S$ makes full sense regardless of whether the $u_i$'s and $v_i$'s are from the same algebra
or from different algebras. This is just stating that tensor product (over $\mathbb C$) of arbitrary $C^*$-algebras is well-defined (the question of completion 
simply does not appear when we consider $\mathcal S,\mathcal S_u,\mathcal S_v$;
it barely appears when considering the resolvent $(z-\mathcal S)^{-1}$, in the sense that $\mathcal S$ is viewed as an element in the spatial tensor 
product, and the inverse which is the resolvent is viewed in the same space). 

We argue that we may view $\mathcal S$ - and some transforms of $\mathcal S$ - separately as a noncommutative function of 
$r$ variables, first in one tensor factor (with the coefficients and the $v$'s being fixed parameters), then in the second
(with the coefficients and the $u$'s being fixed parameters). We consider the composition of $\mathcal S$ with the  
map $\underline{s}_r\mapsto\underline{\Psi}_r(\underline{s}_r):=(\Psi(s_1),\dots\Psi(s_r))$. This 
map has just been shown to be a bijective noncommutative map from $\mathcal I_r(\mathcal A)$ onto its range $\underline{\Psi}_r(\mathcal I_r(\mathcal A))$
(it should be noted for future reference that subsets $\mathcal I_r^\varepsilon(\mathcal A):=
\coprod_{n\ge1}\{\underline{s}_r\in\mathcal I_r(\mathcal A)_n\colon -\varepsilon\le\Im s_j\le\varepsilon,1\le j\le r\}$ are mapped into uniformly 
norm-bounded sets, bounded by a constant which only depends on $\varepsilon\in[0,1)$).

While straightforward (like the proof for $\mathcal I_r(\mathcal A)$), for the sake of the reader's comfort we provide here the proof of the fact that, 
given (fixed) $z\in\mathbb C^+$, von Neumann algebra $\tilde{\mathcal A},p\in\mathbb N,$ $v_1,\dots,v_r\in\underline{\Psi}_r(\mathcal I_r(\tilde{\mathcal A})_p)$
such that $z-\xi\otimes1\otimes 1-\sum_{i=1}^{r}(\gamma_i\otimes \Psi(s_i)\otimes 1+\gamma_i^*\otimes
\Psi(s_i)^{-1}\otimes1+\beta_i\otimes1\otimes {v_i}+\beta_i^*\otimes 1\otimes {v_i^{-1}})$ is invertible when $s_i=s_i^*$,
the correspondence $\underline{s}_r\mapsto\Big(z-
\xi\otimes1\otimes 1-\sum_{i=1}^{r}(\gamma_i\otimes \Psi(s_i)\otimes 1+\gamma_i^*\otimes
\Psi(s_i)^{-1}\otimes1+\beta_i\otimes1\otimes {v_i}+\beta_i^*\otimes 1\otimes {v_i^{-1}})\Big)^{-1}$ is locally a noncommutative map.

First, let us observe that the correspondence  $\underline{s}_r\mapsto\xi\otimes1\otimes 1+\sum_{i=1}^{r}(\gamma_i\otimes \Psi(s_i)\otimes 1+\gamma_i^*\otimes
\Psi(s_i)^{-1}\otimes1+\beta_i\otimes1\otimes {v_i}+\beta_i^*\otimes 1\otimes {v_i^{-1}})$
is indeed a noncommutative map. This is quite obvious, however: the correspondences $s_j\mapsto\Psi(s_j)^{\pm1}$ have been seen just above to be noncommutative maps
on all of $\mathcal I_r(\mathcal A)$ by analytic functional calculus. Thus, it is enough to show that 
$\underline{u}_r\mapsto \xi\otimes1\otimes 1+\sum_{i=1}^{r}(\gamma_i\otimes u_i\otimes 1+\gamma_i^*\otimes
u_i^{-1}\otimes1+\beta_i\otimes1\otimes {v_i}+\beta_i^*\otimes 1\otimes {v_i^{-1}})$ satisfies the axioms of noncommutative functions 
on $\underline{\Psi}_r(\mathcal I_r(\mathcal A))$. This is tautological; one only needs to specify the manner in which the extension is performed.
Our level-one map is from (a subset of) $\mathcal A$ into $M_m(\mathbb C)\otimes\mathcal A\otimes\tilde{\mathcal A}$.
Thus, the amplifications must send elements from $M_n(\mathcal A)$ into $M_n(M_m(\mathbb C)\otimes\mathcal A\otimes\tilde{\mathcal A})$.
To be very rigorous\footnote{We know from matrix analysis that $(A\oplus B)\otimes C=(A\otimes C)\oplus(B\otimes C)$ but often
$A\otimes(B\oplus C)$ is only permutation equivalent to $(A\otimes B)\oplus(A\otimes C)$ - see \cite[Corollary 4.3.16]{H}.}, we consider first
the different formula corresponding to noncommutative functions taking values in $\mathcal A\otimes M_m(\mathbb C)\otimes\tilde{\mathcal A}$;
this corresponds to applying the isometric isomorphism $\mathsf{flip}$ in the first two coordinates, i.e. ${}^0\mathsf{flip}^1$ in our notation. Then we can safely write
\begin{align*}
\begin{bmatrix} u_i & 0 \\ 0 & \tilde{u}_i \end{bmatrix}\otimes\gamma_i&\otimes 1+
\begin{bmatrix} u_i^{-1} & 0 \\ 0 & \tilde{u}_i^{-1} \end{bmatrix}\otimes\gamma_i^*\otimes 1\\
=&
\begin{bmatrix} u_i\otimes\gamma_i\otimes 1 & 0 \\ 0 & \tilde{u}_i\otimes\gamma_i\otimes 1 \end{bmatrix}+
\begin{bmatrix} u_i^{-1}\otimes\gamma_i^*\otimes 1 & 0 \\ 0 & \tilde{u}_i^{-1}\otimes\gamma_i^*\otimes 1 \end{bmatrix},
\end{align*}
showing that (since sums of noncommutative functions are noncommutative functions) 
$\underline{s}_r\mapsto1\otimes\xi\otimes 1+\sum_{i=1}^{r}( \Psi(s_i)\otimes \gamma_i\otimes1+\Psi(s_i)^{-1}\otimes\gamma_i^*\otimes
1+1\otimes \beta_i\otimes{v_i}+1\otimes \beta_i^*\otimes {v_i^{-1}})$, and hence $\underline{s}_r\mapsto\Big(z-
1\otimes\xi\otimes 1-\sum_{i=1}^{r}(\Psi(s_i)\otimes \gamma_i\otimes 1+\Psi(s_i)^{-1}\otimes\gamma_i^*\otimes
1+1\otimes\beta_i\otimes {v_i}+1\otimes \beta_i^*\otimes {v_i^{-1}})\Big)^{-1}$ is - locally - a noncommutative map.
(Conjugation with scalar invertible matrices clearly passes through, and we know that sets of invertibility are level-wise open.) 
Since ${}^0{\sf flip}^1$ is isometric, it extends to the spatial tensor product completion,
so it can be applied to our resolvent as well (${}^0{\sf flip}^1$  is also its own inverse, i.e. ${}^0\mathsf{flip}^1\circ{}^0\mathsf{flip}^1$ is the identity map). 
As a von Neumann algebra isomorphism, it can be extended to matrices of all sizes, by an entry-wise application. Of course, two consecutive applications 
bring us to the identity, as in the case of level one. Thus, when analyzing the behavior of this correspondence as a noncommutative function in 
variables from $\mathcal A$, we first apply ${}^0{\sf flip}^1$, then we perform the evaluations at the desired levels, and then apply again ${}^0{\sf flip}^1$.
For $\underline{t}_r\mapsto\Big(z-
\xi\otimes1\otimes 1-\sum_{i=1}^{r}(\gamma_i\otimes u_i\otimes 1+\gamma_i^*\otimes
u_i^{-1}\otimes1+\beta_i\otimes1\otimes\Psi( {t_i})+\beta_i^*\otimes 1\otimes\Psi(t_i)^{-1})\Big)^{-1}$, one replaces ${}^0{\sf flip}^1$ by ${}^0{\sf flip}^1\circ{}^1{\sf flip}^0 
\colon M_m(\mathbb C)\otimes\mathcal A\otimes\tilde{\mathcal A}\to\tilde{\mathcal A}\otimes M_m(\mathbb C)\otimes\mathcal A$ in the above argument. 

We have thus proved locally the desired properties for our map. However, it is important to have large enough noncommutative domains on which the map in question
is defined. For that, we show next that for any  levels $p,q\in\mathbb N,p,q>0$, and any $z_0\in\mathbb C^+$,
there exists $\varepsilon>0$ such that the map $\{z\in\mathbb C^+\colon\Im z_0-\varepsilon<\Im z\}\times
\mathcal I_r^\varepsilon(\mathcal A)_p\times\mathcal I_r^\varepsilon(\tilde{\mathcal A})_q\ni(z,\underline{s}_r,\underline{t}_r)\mapsto
\Big(z-
\xi\otimes1\otimes 1-\sum_{i=1}^{r}(\gamma_i\otimes \Psi(s_i)\otimes 1+\gamma_i^*\otimes
\Psi(s_i)^{-1}\otimes1+\beta_i\otimes1\otimes\Psi(t_i)+\beta_i^*\otimes 1\otimes\Psi(t_i)^{-1})\Big)^{-1}$
is a well-defined analytic map.
This will be done by a rough, far from optimal, majorization.
Pick an arbitrary $j\in\{1,\dots,r\}$ and write
\begin{eqnarray*}
\lefteqn{\Im\left(\gamma_i\otimes \Psi(s_i)\otimes 1+\gamma_i^*\otimes
\Psi(s_i)^{-1}\otimes1\right)}\\
& = & \frac{1}{2i}\left(\gamma_i\otimes \Psi(s_i)+\gamma_i^*\otimes
\Psi(s_i)^{-1}-\gamma_i^*\otimes \Psi(s_i)^*-\gamma_i\otimes
[\Psi(s_i)^{-1}]^*\right)\otimes1\\
& = & \frac{1}{2i}\left(\gamma_i\otimes\left[\frac{s_i+i}{s_i-i}-\frac{s_i^*+i}{s_i^*-i}\right]+\gamma_i^*\otimes\left[\frac{s_i-i}{s_i+i}-\frac{s_i^*-i}{s_i^*+i}\right]
\right)\otimes1\\
& = & \frac{1}{2i}\left(\gamma_i\otimes\left[\frac{2i}{s_i-i}-\frac{2i}{s_i^*-i}\right]+\gamma_i^*\otimes\left[\frac{-2i}{s_i+i}-\frac{-2i}{s_i^*+i}\right]
\right)\otimes1\\
& = & \left(\!\gamma_i\otimes(s_i-i)^{-1}\!(s_i^*\!-s_i)(s_i^*\!-i)^{-1}\!+\gamma_i^*\otimes(s_i^*\!+i)^{-1}\!(s_i-s_i^*)(s_i+i)^{\!-1}
\right)\!\otimes\!1.
\end{eqnarray*}
We apply the obvious majorization:
\begin{eqnarray*}
\lefteqn{\|\Im\left(\gamma_i\otimes \Psi(s_i)\otimes 1+\gamma_i^*\otimes
\Psi(s_i)^{-1}\otimes1\right)\|}\\
& = & \|\!\left(\!\gamma_i\otimes(s_i-i)^{-1}\!(s_i^*\!-s_i)(s_i^*\!-i)^{-1}\!+\gamma_i^*\otimes(s_i^*\!+i)^{-1}\!(s_i-s_i^*)(s_i+i)^{\!-1}
\right)\!\|\\
& \le & 2\|\gamma_i\|\|s_i^*-s_i\|\|(s_i-i)^{-1}\|\|(s_i+i)^{-1}\|\\
& \le & 2\|\gamma_i\|\|\Im s_i\|\|(1-\Im s_i)^{-1}\|\|(1+\Im s_i)^{-1}\|\\
& < & 4\|\gamma_i\|\frac{\varepsilon}{(1-\varepsilon)^2}.
\end{eqnarray*}
This guarantees that 
\begin{eqnarray*}
\lefteqn{\left\|\Im\Big(\right.\xi\otimes1\otimes 1+\sum_{i=1}^{r}(\gamma_i\otimes \Psi(s_i)\otimes 1+\gamma_i^*\otimes
\Psi(s_i)^{-1}\otimes1}\\
& &\mbox{}+\beta_i\otimes1\otimes\Psi(t_i)+\beta_i^*\otimes1\otimes\Psi(t_i)^{-1})\left.\Big)\right\|<8\frac{\varepsilon\sum_{i=1}^r(\|\gamma_i\|+\|\beta_i\|)}{(1-\varepsilon)^2}
\end{eqnarray*}
whenever all of the variables $\underline{s}_r,\underline{t}_r$ satisfy $\|\Im s_j\|<\varepsilon,\|\Im t_j\|<\varepsilon$.
Thus, if we choose $\varepsilon\in\left(0,\min\left\{1-\frac{1}{\sqrt{2}}, \frac{\Im z_0}{1+16\sum_{i=1}^r(\|\gamma_i\|+\|\beta_i\|)}\right\}\right)$, we are guaranteed that
the correspondence $\{z\in\mathbb C^+\colon\Im z_0-\varepsilon<\Im z\}\times
\mathcal I_r^\varepsilon(\mathcal A)_p\times\mathcal I_r^\varepsilon(\tilde{\mathcal A})_q\ni(z,\underline{s}_r,\underline{t}_r)\mapsto
\Big(z-
\xi\otimes1\otimes 1-\sum_{i=1}^{r}(\gamma_i\otimes \Psi(s_i)\otimes 1+\gamma_i^*\otimes
\Psi(s_i)^{-1}\otimes1+\beta_i\otimes1\otimes\Psi(t_i)+\beta_i^*\otimes 1\otimes\Psi(t_i)^{-1})\Big)^{-1}$
is well-defined. Since it is rational in all variables, it is automatically analytic on this domain. We also note that $\varepsilon$ 
can be taken to depend only on $\Im z_0>0$ and on the initial data of gammas and betas, and not on $\xi$, not on the levels
$p,q\in\mathbb N$, not on the specific von Neumann (in fact just $C^*$-) algebras $\mathcal A,\tilde{\mathcal A}$, and not on the norms of the operators 
$s_i,t_j$ (which shows that even unbounded evaluations are possible, even though 
in that situation questions of continuity/analyticity become more involved - fortunately that does not concern us here).
In particular, we conclude from the above that, given a fixed $z_0\in\mathbb C^+$, there exists an $\varepsilon>0$ such that, if $\|\Im t_1\|,\dots,\|\Im t_r\|<\varepsilon$,
then $\underline{s}_r\mapsto z-\xi\otimes1\otimes1-\sum_{i=1}^{r}( \gamma_i\otimes\Psi(s_i)\otimes1+\gamma_i^*\otimes\Psi(s_i)^{-1}\otimes
1+\beta_i\otimes1\otimes\Psi(t_i)+\beta_i^*\otimes 1\otimes\Psi(t_i)^{-1})$ is invertible on all of $\mathcal I_r^\varepsilon(\mathcal A)$, and hence $\underline{s}_r\mapsto\Big(z-
\xi\otimes1\otimes 1-\sum_{i=1}^{r}(\gamma_i\otimes \Psi(s_i)\otimes 1+\gamma_i^*\otimes
\Psi(s_i)^{-1}\otimes1+\beta_i\otimes1\otimes\Psi(t_i)+\beta_i^*\otimes 1\otimes\Psi(t_i)^{-1})\Big)^{-1}$ is a noncommutative map defined on
$\mathcal I_r^\varepsilon(\mathcal A)$ which is bounded (uniformly) on $\mathcal I_r^{\varepsilon_1}(\mathcal A)$ for any fixed $\varepsilon_1\in[0,\varepsilon)$.

As we have seen already that $\|\Psi(s_i)^{\pm 1}\|,\|\Psi(t_j)^{\pm1}\|$ are bounded uniformly in terms of $\|\Im s_i\|,\|\Im t_j\|\in[0,1)$ exclusively, the above
argument also guarantees that, for a given majorization $\|\Im s_i\|,\|\Im t_j\|<\varepsilon\in(0,1),$ there exists an $M>0$ which depends only on $\varepsilon
\in[0,1)$ and on the initial data $\xi,\gamma_1,\dots,\gamma_r,\beta_1,\dots,\beta_r\in M_m(\mathbb C)$, such that
$z-\xi\otimes1\otimes 1-\sum_{i=1}^{r}(\gamma_i\otimes \Psi(s_i)\otimes 1+\gamma_i^*\otimes
\Psi(s_i)^{-1}\otimes1+\beta_i\otimes1\otimes\Psi(t_i)+\beta_i^*\otimes 1\otimes\Psi(t_i)^{-1})$ is invertible for all $|z|>M$, and has a power series
expansion in $\frac1z$ which converges uniformly on $|z|\ge M'$ if $M'>M$ is given. It is important to note that this constant $M$ does {\em not} depend on
the von Neumann algebras in which $s_i,t_i$ live, on the levels $p,q$ at which we evaluate, or on the norm of the operators $s_i,t_i$.

We summarize the facts established above in the following
\begin{proposition}\label{analyticity}
With the previously introduced notations, 

\begin{itemize}

\item given $z_0\in\mathbb C^+$,  there exists an $\varepsilon\in(0,1)$ such that the correspondence 
$\{z\in\mathbb C^+\colon\Im z_0-\varepsilon<\Im z\}\times
\mathcal I_r^\varepsilon(\mathcal A)_p\times\mathcal I_r^\varepsilon(\tilde{\mathcal A})_q\ni(z,\underline{s}_r,\underline{t}_r)\mapsto
\Big(z-\xi\otimes1\otimes 1-\sum_{i=1}^{r}(\gamma_i\otimes \Psi(s_i)\otimes 1+\gamma_i^*\otimes
\Psi(s_i)^{-1}\otimes1+\beta_i\otimes1\otimes\Psi(t_i)+\beta_i^*\otimes 1\otimes\Psi(t_i)^{-1})\Big)^{-1}$ 
is analytic for all $p,q\in\mathbb N,p,q>0$. This $\varepsilon$ does not depend on $p,q\in\mathbb N$.
If $0<\varepsilon_1<\varepsilon$ is given, then the set $\{z\in\mathbb C^+\colon\Im z_0-\varepsilon_1<\Im z\}\times
\mathcal I_r^{\varepsilon_1}(\mathcal A)_p\times\mathcal I_r^{\varepsilon_1}(\tilde{\mathcal A})_q$ is mapped inside a norm-bounded set,
with a bound depending on $\varepsilon_1$ and not on $p,q$. In particular, if any two of the three variables $(z,\underline{s}_r,\underline{t}_r)$
are fixed, then the correspondence in the third is a noncommutative function on the given domain;

\item given $\varepsilon\in(0,1)$, there exists an $M>0$ such that $\{z\in\mathbb C^+\colon|z|>M\}\times
\mathcal I_r^\varepsilon(\mathcal A)_p\times\mathcal I_r^\varepsilon(\tilde{\mathcal A})_q\ni(z,\underline{s}_r,\underline{t}_r)\mapsto
\Big(z-\xi\otimes1\otimes 1-\sum_{i=1}^{r}(\gamma_i\otimes \Psi(s_i)\otimes 1+\gamma_i^*\otimes
\Psi(s_i)^{-1}\otimes1+\beta_i\otimes1\otimes\Psi(t_i)+\beta_i^*\otimes 1\otimes\Psi(t_i)^{-1})\Big)^{-1}$ is analytic for all $p,q\in\mathbb N,p,q>0$. 
This $M$ does not depend on $p,q\in\mathbb N$. 
If $M_1>M$, then this correspondence sends $\{z\in\mathbb C^+\colon|z|>M_1\}\times
\mathcal I_r^\varepsilon(\mathcal A)_p\times\mathcal I_r^\varepsilon(\tilde{\mathcal A})_q$ in a norm-bounded set, with a bound that does not depend on $p,q$.
If any two of the three variables $(z,\underline{s}_r,\underline{t}_r)$
are fixed, then the correspondence in the third is a noncommutative function on the given domain.

\end{itemize}

Via a generalized Schwarz reflection principle, the first item above is rephrased in the obvious manner for $z_0\in\mathbb C^-$.

\end{proposition}

In order to use Parraud's formulas, which are shown to hold on polynomials and exponentials, we need to express the Cayley transform as a Fourier-like transform. 
We next use the formulas on the power series expansions at infinity of the resolvent of $\mathcal S$, whose terms are polynomials 
in Cayley transforms of our variables, and extend by analytic continuation. 

The Fourier-like formula we use is the following:
\begin{equation}\label{ix}
\Psi(x)^\epsilon=\frac{x+\epsilon i}{x-\epsilon i} =1-2\int_{-\infty}^0 e^{(i\epsilon x+1)y}\,{\rm d}y,\quad \epsilon\in\{\pm1\}.
\end{equation}
The first equality holds for all $x\in\overline{\mathbb C}$, and the second for all $x\in\{z\in\mathbb C\colon|\Im z|<1\}$, as analytic functions
(to be precise, when $\epsilon=1$, the second equality holds on $\{z\in\mathbb C\colon\Im z<1\}$, and when $\epsilon=-1$, on $\{z\in\mathbb C\colon\Im z>-1\}$). 
This extends via analytic functional calculus to bounded linear operators with imaginary part between $-1$ and $1$ 
(of course, \eqref{ix} makes sense for operators whose spectrum is included in the corresponding domain),
and, for {\em selfadjoint} unbounded operators affiliated to any type II von Neumann algebra, via continuous functional calculus; we use the obvious notation:
\begin{equation}\label{expo}
\Psi(Z)^\epsilon=\frac{Z+\epsilon i}{Z-\epsilon i} =1-2\int_{-\infty}^0 e^{(i\epsilon Z+1)y}\,{\rm d}y,\quad \epsilon\in\{\pm1\}.
\end{equation}
(Since there is no risk of confusion, we continue to write $\frac{Z+\epsilon i}{Z-\epsilon i}$ for $(Z+\epsilon i)(Z-\epsilon i)^{-1}=(Z-\epsilon i)^{-1}(Z+\epsilon i)$, as we have before.) 

Analyticity (see Example \ref{ExAFC}) guarantees that the Taylor-Taylor power series expansions around zero of the left and right hand sides in \eqref{expo} coincide.
However, a direct verification is possible: 
$$
\frac{Z+\epsilon i}{Z-\epsilon i} =1-2\sum_{n=0}^\infty(-i\epsilon Z)^n,\quad \|Z\|<1,
$$
\begin{eqnarray*}
1-2\int_{-\infty}^0e^{(i\epsilon Z+1)y}\,{\rm d}y&=&1-2\int_{-\infty}^0e^ye^{i\epsilon Zy}\,{\rm d}y=1-2\int_{-\infty}^0\!\!e^y\sum_{n=0}^\infty\frac{(i\epsilon Zy)^n}{n!}\,{\rm d}y\\
&=&1-2\sum_{n=0}^\infty\left(\int_{-\infty}^0\frac{y^ne^y}{n!}\,{\rm d}y\right)(i\epsilon Z)^n\\
& = & 1-2\sum_{n=0}^\infty\left(-1\right)^n(i\epsilon Z)^n,\quad \|Z\|<1.
\end{eqnarray*}
(We have used the commutativity of the identity with $Z$ to write $ e^{(i\epsilon Z+1)y}=e^ye^{i\epsilon Zy}$, and the elementary equality $\mathsf I_n=-n\mathsf I_{n-1}$
for $\mathsf I_n=\int_{-\infty}^0y^ne^y\,{\rm d}y$.) This gives an alternative, elementary proof of the equality \eqref{expo} 
on operators of norm strictly less than one. It follows from Example \ref{ExAFC} that the equality holds for all operators $Z$ with $-1<\Im Z<1$. 

Since $\Psi(Z)^\epsilon$ are noncommutative functions on an open set, they are differentiable, in the (much stronger) sense that they accept the application of
difference-differential operators and the resulting functions have domains which are no smaller than the domain of $\Psi(Z)^\epsilon$. For operators
of norm strictly less than one, a direct verification using the power series found above allows one to conclude that 
$\partial\Psi(Z)^\epsilon=\frac{i\epsilon}{2}(\Psi(Z)^\epsilon-1)\otimes(\Psi(Z)^\epsilon-1)$, regardless of which expression for $\Psi$ we consider. 
(Of course, thanks to Voiculescu's result \cite[Section 3]{Coalg}, this needs neither verification, nor proof, when the rational formula for $\Psi$ is employed.)
Thus, if $s_1,\dots,s_r$ are algebraically free selfadjoint variables, then
\begin{equation}\label{Cay}
\partial_j \Psi(s_k)^\epsilon=\delta_{j,k}\epsilon\frac{i}{2}(\Psi(s_k)^\epsilon-1)\otimes(\Psi(s_k)^\epsilon-1),\quad j,k=1,\dots,r,\ \epsilon\in\{\pm1\}.
\end{equation}
The formula derived from the expression involving exponentials can be seen to be the same, for example by employing (classical) analytic continuation.
A most direct verification shows the equality
$$
\Delta\Psi(Z_1;Z_2)(W)=(\mathsf{ev}_W\circ\#\circ\partial)(\Psi)(Z_1;Z_2)(W),\ \ -1<\Im Z_j<1,W\text{ arbitrary},
$$
with a similar formula for $\Psi(\cdot)^{-1}$.
In particular,
\begin{equation}\label{DeltaCay}
\Delta(\Psi^\epsilon)(s;t)(c)=\epsilon\frac{i}{2}(\Psi(s)^\epsilon-1)c(\Psi(t)^\epsilon-1),\quad-1<\Im s,\Im t<1, \ \epsilon\in\{\pm1\}.
\end{equation}

Up to this point, we have established that $\underline{s}_r\mapsto(z-\mathcal S)^{-1},\underline{t}_r\mapsto(z-\mathcal S)^{-1}$ are noncommutative
functions on convenient domains that contain the selfadjoints, that the Cayley transform $\Psi$ can be expressed in terms of a Fourier-like integral on $\mathbb R$,
and the formula for $\partial\Psi,\Delta\Psi$. As it will be seen in Section \ref{Sec:P}, Parraud's formulae for linear functionals that are used in the definition of
$E$ and $\Delta_N$ from Section \ref{HTSapproach} involve applications of $\partial_i,\#,\mathsf{ev}_{.},\mathsf{flip}$ to the tensor coordinates of $(z-\mathcal S)^{-1}$.
The fact that these operations make sense on $(z-\mathcal S)^{-1}$ follows very easily by employing the identification of $\mathsf{ev}\circ\partial_j$ and $\Delta_j$, 
coupled with Proposition \ref{analyticity}. However, that is not enough: we need to show that these applications indeed represent an extension of 
Parraud's functionals to $(z-\mathcal S)^{-1}$, as well as prove certain properties of the functions obtained this way, properties related to the behavior 
in $z$ when $z$ is close to the real line.

\subsection{Differentiation and evaluation}\label{diffev}

Recall the fact that $\mathsf{flip}$ is a von Neumann algebra isomorphism, so that $\mathsf{flip}(z-W)^{-1}=(z-\mathsf{flip}W)^{-1},$ for instance.
This extends to any rational expression. 

We record next the expression one obtains 
when applying the partial difference-differential operator to $(z-\mathcal S)^{-1}$. The computations are not new (see \cite[2.3.5]{KVV}), and we only provide them for the sake of 
completeness; the reader may view them as an exercise. We recall from above the manner in which we understand $\mathcal S$ as a noncommutative function, and (since we 
work in our computation  on the second tensor coordinate) we remind the reader that  - for computational purposes only - we first perform the flipping of the first two coordinates:
${}^0\mathsf{flip}^1(z-\mathcal S)=1\otimes b\otimes1-\sum_{i=1}^{r}(\Psi(s_i)\otimes \gamma_i\otimes 1+
\Psi(s_i)^{-1}\otimes\gamma_i^*\otimes1+1\otimes\beta_i\otimes\Psi(t_i)+1\otimes\beta_i^*\otimes\Psi(t_i)^{-1})$. Then we pick some other $c,u_1,\dots,u_r\in\mathcal A$
(or any matrix amplification of it, if needed) and evaluate the multiplicative inverse of $\begin{bmatrix} 1 & 0 \\ 0 & 1\end{bmatrix}\otimes b\otimes1-\sum_{i=1}^{r}\Big\{
\Psi\left(\begin{bmatrix} s_i & \delta_{i,j}c \\ 0 & u_i\end{bmatrix}\right)\otimes\gamma_i\otimes 1+\Psi\left(\begin{bmatrix} s_i & \delta_{i,j}c \\ 0 & u_i\end{bmatrix}\right)^{-1}
\otimes\gamma_i^*\otimes1+\begin{bmatrix} 1 & 0 \\ 0 & 1\end{bmatrix}\otimes\beta_i\otimes\Psi(t_i)+\begin{bmatrix} 1 & 0 \\ 0 & 1\end{bmatrix}\otimes\beta_i^*\otimes
\Psi(t_i)^{-1}\Big\}
=\begin{bmatrix} 1\otimes b\otimes1 & 0 \\ 0 & 1\otimes b\otimes1 \end{bmatrix}-\sum_{i=1}^{r}
(\begin{bmatrix} \Psi(s_i)\otimes\gamma_i\otimes 1 & \delta_{i,j}\Delta\Psi(s_i,u_i)(c)\otimes\gamma_i\otimes 1 \\ 0 & \Psi(u_i)\otimes\gamma_i\otimes 1\end{bmatrix}+
\begin{bmatrix} 1\otimes\beta_i\otimes\Psi(t_i) & 0 \\ 0 & 1\otimes\beta_i\otimes\Psi(t_i) \end{bmatrix}\ +\ 
\begin{bmatrix} 1\otimes\beta_i^*\otimes\Psi(t_i)^{-1} & 0 \\ 0 & 1\otimes\beta_i^*\otimes\Psi(t_i)^{-1} \end{bmatrix}+
\begin{bmatrix} \Psi(s_i)^{-1}\otimes\gamma_i\otimes 1 & -\delta_{i,j}\Psi(s_i)^{-1}\Delta\Psi(s_i,u_i)(c)\Psi(u_i)^{-1}\otimes\gamma_i\otimes 1 \\ 0 & \Psi(u_i)^{-1}\otimes\gamma_i
\otimes 1\end{bmatrix})$. (In order to make the argument more self-contained,
\begin{eqnarray}
\Psi\left(\begin{bmatrix} s_j & c \\ 0 & u_j\end{bmatrix}\right)& = & \begin{bmatrix} 1 & 0 \\ 0 & 1\end{bmatrix}+2i
\begin{bmatrix} s_j-i & c \\ 0 & u_j-i\end{bmatrix}^{-1}\nonumber\\
& = & 
\begin{bmatrix} 1+2i(s_j-i)^{-1} & -2i(s_j-i)^{-1}c(u_j-i)^{-1} \\ 0 & 1+2i(u_j-i)^{-1}\end{bmatrix}\nonumber\\
& = & 
\begin{bmatrix} \Psi(s_j) & \frac{i}{2}\left(\Psi(s_j)-1\right)c\left(\Psi(u_j)-1\right) \\ 0 & \Psi(u_j)\end{bmatrix};\label{p}
\end{eqnarray}
recall \eqref{DeltaCay}. If $c=1$, we recognize in the $(1,2)$ entry $((\#\circ\partial_j)\Psi)(s_j;u_j)$, and if $s_j=u_j$, we recognize $(\mathsf{ev}_c\circ\partial_j)(\Psi)(s_j)$.)
The forms of $\Psi(\cdot)$ and $\Psi(\cdot)^{-1}$ are so similar that we will not hesitate to keep in our formulae the expression $\Delta(\Psi^{-1})(s_i,u_i)(c)$
instead of $-\Psi(s_i)^{-1}\Delta\Psi(s_i,u_i)(c)\Psi(u_i)^{-1}$, or alternate between them. 

We view $(z-\mathcal S)^{-1}$ - after the application of ${}^0\mathsf{flip}^1$ - as a noncommutative function in the variables $\underline{s}_r$ and apply the usual 
calculation rules (for saving space, we use again $b=z-\xi$ in our notation):
{\tiny{\begin{eqnarray}
\lefteqn{\left(\begin{bmatrix} 1 & 0 \\ 0 & 1\end{bmatrix}\otimes b\otimes1-\sum_{i=1}^{r}
\left(\Psi\left(\begin{bmatrix} s_i & \delta_{i,j}c \\ 0 & u_i\end{bmatrix}\right)\otimes\gamma_i\otimes 1+\Psi\left(\begin{bmatrix} s_i & \delta_{i,j}c \\ 0 & u_i\end{bmatrix}\right)^{-1}
\otimes\gamma_i^*\otimes1\right.\right.}\nonumber\\
& & \mbox{}+\left.\left.\begin{bmatrix} 1 & 0 \\ 0 & 1\end{bmatrix}\otimes\beta_i\otimes\Psi(t_i)+\begin{bmatrix} 1 & 0 \\ 0 & 1\end{bmatrix}\otimes\beta_i^*\otimes
\Psi(t_i)^{-1}\right)\right)^{-1}\nonumber\\
& = & \left(\begin{bmatrix}(1\otimes b-\sum\Psi(s_i)\otimes\gamma_i+\Psi(s_i)^{-1}\otimes\gamma_i^*)\otimes1 & -\sum\delta_{i,j}(\Delta\Psi(s_i;u_i)(c)\otimes
\gamma_i+\Delta(\Psi^{-1})(s_i,;u_i)(c)\otimes\gamma_i^*)\otimes1 \\ 0 & (1\otimes b-\sum\Psi(u_i)\otimes\gamma_i
+\Psi(u_i)^{-1}\otimes\gamma_i^*)\otimes1\end{bmatrix}\right.\nonumber\\
&  &\mbox{}-\left. \begin{bmatrix} \sum1\otimes(\beta_i\otimes\Psi(t_i)+\beta_i^*\otimes\Psi(t_i)^{-1}) & 0 \\ 0 & \sum1\otimes(\beta_i\otimes\Psi(t_i)+\beta_i^*\otimes\Psi(t_i)^{-1})
\end{bmatrix}\right)^{-1}\nonumber\\
& = &  \begin{bmatrix} {}^0\mathsf{flip}^1(z-\mathcal S) & -{}^0\mathsf{flip}^1\sum\delta_{i,j}(\gamma_i\otimes\Delta\Psi(s_i;u_i)(c)
+\gamma_i^*\otimes\Delta(\Psi^{-1})(s_i,;u_i)(c))\otimes1 \\ 0 & {}^0\mathsf{flip}^1(z-\mathcal S)\end{bmatrix}\nonumber\\
& = & M_2({}^0\mathsf{flip}^1)\!\begin{bmatrix}(z-\mathcal S)^{-1} & \!\!(z-\mathcal S)^{-1}\!\left[\sum\delta_{i,j}(\gamma_i\otimes\Delta\Psi(s_i;u_i)(c)
+\gamma_i^*\otimes\Delta(\Psi^{-1})(s_i,;u_i)(c))\otimes1\right]\!(z-\mathcal S)^{-1} \\ 0 & (z-\mathcal S)^{-1}\end{bmatrix}\label{drezzy}
\end{eqnarray}}}\noindent
Here the upper left $\mathcal S$ is evaluated in $\underline{s}_r$, the lower right one in $\underline{u}_r$, the one on the left side of the $(1,2)$ entry in $\underline{s}_r$,
and the one on the right side of the $(1,2)$ entry in $\underline{u}_r$. Replacing $\Delta\Psi(s_i;u_i)(c)$ with the result of \eqref{p} and 
$\Delta(\Psi^{-1})(s_i;u_i)(c)$ with the result obtained by replacing $\Psi$ with $\Psi^{-1}$ in \eqref{p} followed by a second application of $\mathsf{flip}$
yields 
\begin{eqnarray}
\lefteqn{\Delta_j(z-\mathcal S)^{-1}}\nonumber\\
&\!\!\! = & \frac{i}{2}\left((z-\xi)\otimes1\otimes1-\sum_{i=1}^{r}(\gamma_i\otimes\Psi(s_i)\otimes1+\gamma_i^*\otimes\Psi(s_i)^{-1}\otimes1)\right.\nonumber\\
& & \mbox{}-\left.\sum_{i=1}^r(\beta_i\otimes1\otimes\Psi(t_i)+\beta_i^*\otimes1\otimes \Psi(t_i)^{-1})\right)^{-1}\nonumber\\
& & \mbox{}\times\left(\gamma_j\!\otimes\![(\Psi(s_j)\!-\!1)c(\Psi(u_j)\!-\!1)]-\gamma_j^*\otimes[(\Psi(s_j)^{-1}\!-\!1)c(\Psi(u_j)^{-1}\!-\!1)]\right)\!\otimes\!1\nonumber\\
&&\mbox{}\times\left((z-\xi)\otimes1\otimes1-\sum_{i=1}^{r}(\gamma_i\otimes\Psi(u_i)\otimes1+\gamma_i^*\otimes\Psi(u_i)^{-1}\otimes1)\right.\nonumber\\
& & \mbox{}-\left.\sum_{i=1}^r(\beta_i\otimes1\otimes\Psi(t_i)+\beta_i^*\otimes1\otimes \Psi(t_i)^{-1})\right)^{-1}.\label{frezzy}
\end{eqnarray}
Employing power series expansions around infinity yields the same result, as it can be directly verified.
The terms of the power series expansion are polynomials in $\Psi(\cdot)$ and $\Psi(\cdot)^{-1}$, functions for which the conclusions of Section \ref{Sec:DD-FDQ} hold.
Thus, the same holds for the (convergent) power series for $|z|$ large, and, by analytic continuation, for all $z$ in the connected component of the domain of definition which 
contains a neighborhood of infinity.
We wish to emphasize that in \eqref{frezzy} the power series expansions around infinity (in $1/z$) are convergent on the neighborhood of infinity which is the set
of points $z\in\mathbb C$ with the property that $w-\xi\otimes1\otimes1-\sum_{i=1}^{r}(\gamma_i\otimes\Psi(s_i)\otimes1+\gamma_i^*\otimes\Psi(s_i)^{-1}\otimes1)
-\sum_{i=1}^r(\beta_i\otimes1\otimes\Psi(t_i)+\beta_i^*\otimes1\otimes \Psi(t_i)^{-1})$ is invertible for all $w\in\mathbb C$ satisfying $|w|\ge|z|$ (if $u_i=s_i$ for all
$i\in\{1,\dots,r\}$, otherwise one needs to require the same condition for the expression in which the vector $\underline{s}_r$ is replaced by $\underline{u}_r$).
It is obvious that one may apply $\Delta_k$ in either of the variables $u$ or $s$ in the right-hand side of \eqref{frezzy}, and the usual rules 
for the computation of difference-differentials apply. Thus, we conclude that 
\begin{itemize}
\item The domain of analyticity in the variable $z$ does not decrease through the application of  $\Delta_j$.
\end{itemize}

\section{Asymptotic expansion of polynomials in  Cayley transforms of GUE's}\label{S:AsyEx}

\subsection{A fundamental result of Parraud}\label{Sec:P}

In \cite{P}, Parraud established an asymptotic expansion in the dimension of smooth functions in polynomials 
in deterministic matrices and iid GUE matrices. The constants in this expansion are built explicitly with the help 
of free probability. In this section, we are going to present the second order  expansion of the expectation 
of the normalized trace of a product of polynomials  in independent 
GUE matrices, contained in Parraud's  work. We just rewrite the constants in this expansion in terms of the operations 
$\#$, $\mathsf{ev}_1$, ${\sf flip}$ and $\partial$ introduced in Section \ref{pol}.

Let $(\mathcal A, \tau)$ be a $C^*$-probability space where $\tau$ is a faithful trace.
Let  $(\mathcal A_N , \tau_N )$ be the free product of $(M_N (C),\tr_N)$ and  $(\mathcal A, \tau)$.

\noindent Let $s$,  $w$, $z^1$, $z^2$, $\tilde w$, $\tilde z^1$, $\tilde z^2$ be free $r$-semi-circular systems in $(\mathcal A, \tau)$.
Let $\underline{X_N}_r$ be an $r$-tuple of independent GUE. Define  for $0\leq t_1\leq t_2$, the following noncommutative variables in $(\mathcal A_N, \tau_N)$
 $$z^1_{t_1}=(1-e^{-{t_1}})^{1/2}z^1 +(e^{-{t_1}} -e^{-t_2})^{1/2} w+ e^{-t_2/2} s,$$
$$z^2_{t_1} =(1-e^{-{t_1}})^{1/2}z^2 +(e^{-{t_1}} -e^{-t_2})^{1/2} w+ e^{-t_2/2} s,$$
$$\tilde z^1_{t_1} =(1-e^{-{t_1}})^{1/2}\tilde z^1 +(e^{-{t_1}} -e^{-t_2})^{1/2}\tilde w+ e^{-t_2/2} s,$$
$$\tilde z^2_{t_1} =(1-e^{-{t_1}})^{1/2}\tilde z^2 +(e^{-{t_1}} -e^{-t_2})^{1/2}\tilde  w+ e^{-t_2/2} s$$
and 
$$z^1_{t_1}(N)=(1-e^{-{t_1}})^{1/2}z^1 +(e^{-{t_1}} -e^{-t_2})^{1/2} w+ e^{-t_2/2} \underline{X_N}_r,$$
$$z^2_{t_1}(N) =(1-e^{-{t_1}})^{1/2}z^2 +(e^{-{t_1}} -e^{-t_2})^{1/2} w+ e^{-t_2/2} \underline{X_N}_r,$$
$$\tilde z^1_{t_1}(N) =(1-e^{-{t_1}})^{1/2}\tilde z^1 +(e^{-{t_1}} -e^{-t_2})^{1/2}\tilde w+ e^{-t_2/2} \underline{X_N}_r,$$
$$\tilde z^2_{t_1} (N)=(1-e^{-{t_1}})^{1/2}\tilde z^2 +(e^{-{t_1}} -e^{-t_2})^{1/2}\tilde  w+ e^{-t_2/2} \underline{X_N}_r.$$
Define now the linear forms $\nu_1$ and $\nu_1^{(N)}$ on $ \mathbb{C} \langle\underline{\mathbf x}_r\rangle$, by setting for  any $P \in\mathbb{C} \langle\underline{\mathbf x}_r\rangle$, 
\begin{eqnarray*}\nu_1(P)&=&\frac{1}{2}\int_0^{+\infty}\int_0^{t_2}   e^{-t_2-t_1} \tau \left(
R_1(P)( z^1_{t_1}, \tilde z^1_{t_1}, \tilde z^2_{t_1}, z_{t_1}^2)\right)\,{\rm  d}t_1{\rm d}t_2, \end{eqnarray*}
\begin{eqnarray*}\nu_1^{(N)}(P)&=&\frac{1}{2}\int_0^{+\infty}\int_0^{t_2}   e^{-t_2-t_1} \tau_N \left(
R_1(P)( z^1_{t_1}(N), \tilde z^1_{t_1}(N), \tilde z^2_{t_1}(N), z_{t_1}^2(N))\right)\,{\rm d}t_1{\rm d}t_2 \end{eqnarray*}
where $R_1(P)$ is the   function in  four indeterminates $x^1, \tilde x^1, \tilde x^2, x^2$ of $r$-tuples defined by 
\begin{equation}\label{defRP}
R_1 (P)(x^1\!,\tilde x^1\!,\tilde x^2\!,x^2)=\!\!\sum_{1\leq j,k\leq r}\!\!\#^3\!\left\{\left[({\sf flip}\circ\partial_{k})\!\otimes\!({\sf flip}\circ 
\partial_{k})\right](\partial_{j}(D_{j}P))\right\}(x^1\!,\tilde x^1\!,\tilde x^2\!,x^2),
\end{equation}
$\#$ being the embedding of the tensor product in the algebraic free product defined in Section \ref{pol}. It is clear that $R_1(P)$ belongs to 
$\mathbb{C}\langle\underline{\mathbf x}_{4r}\rangle$. For the reader's understanding, here is  the explicit formula for $R_1(P)$
when $P(x_1,\ldots,x_r)= x_{i_1}\cdots x_{i_n}$, $(i_1,\ldots,i_n)\in \{1,\ldots,r\}^n$:
\begin{eqnarray*}
\lefteqn{R_1(P)(x^1, \tilde x^1, \tilde x^2, x^2)}\\&=&\sum_{1\leq i,j\leq r} \hspace*{-0.2cm}\sum_{\tiny \begin{array}{cc}1\leq k \leq n\\ i_k=i\end{array}}\hspace*{-0.2cm} \sum_{\tiny \begin{array}{cc}k+1 \leq l\leq n \\ i_l=i\end{array}}\hspace*{-0.2cm} \sum_{\tiny \begin{array}{cc}k+1 \leq p\leq l-1\\ i_p=j\end{array}}\\&&\left\{
\sum_{\tiny \begin{array}{cc}l+1 \leq q\leq n \\ i_q=j\end{array}}x^1_{i_{p+1}}\cdots x^1_{i_{l-1}}
\tilde x^1_{i_{k+1}}\cdots \tilde x^1_{i_{p-1}}
\tilde x^2_{i_{q+1}}\cdots \tilde x^2_{i_{n}}
\tilde x^2_{i_{1}}\cdots \tilde x^2_{i_{k-1}}
 x^2_{i_{l+1}}\cdots  x^2_{i_{q-1}}\right.\\&&
+\left.\sum_{\tiny \begin{array}{cc}1 \leq q\leq  k-1 \\ i_q=j\end{array}}x^1_{i_{p+1}}\cdots x^1_{i_{l-1}}
\tilde x^1_{i_{k+1}}\cdots \tilde x^1_{i_{p-1}}
\tilde x^2_{i_{q+1}}\cdots \tilde x^2_{i_{k-1}}
x^2_{i_{l+1}}\cdots  x^2_{i_{n}}
 x^2_{i_{1}}\cdots  x^2_{i_{q-1}} \right\}
 \\&+&\sum_{1\leq i,j\leq r} \hspace*{-0.2cm}\sum_{\tiny \begin{array}{cc}1\leq k \leq n\\ i_k=i\end{array}}\hspace*{-0.2cm} \sum_{\tiny \begin{array}{cc}1 \leq l\leq k-1 \\ i_l=i\end{array}}\hspace*{-0.2cm} \sum_{\tiny \begin{array}{cc}l+1 \leq q\leq k-1\\ i_q=j\end{array}}\\&&\left\{
\sum_{\tiny \begin{array}{cc}k+1 \leq p\leq n \\ i_p=j\end{array}}x^1_{i_{p+1}}\cdots x^1_{i_{n}}
x^1_{i_{1}}\cdots x^1_{i_{l-1}}
\tilde x^1_{i_{k+1}}\cdots \tilde x^1_{i_{p-1}}
\tilde x^2_{i_{q+1}}\cdots \tilde x^2_{i_{k-1}}
 x^2_{i_{l+1}}\cdots  x^2_{i_{q-1}}\right.\\&&
+\left.\sum_{\tiny \begin{array}{cc}1 \leq p\leq  l-1 \\ i_p=j\end{array}}x^1_{i_{p+1}}\cdots x^1_{i_{l-1}}
\tilde x^1_{i_{k+1}}\cdots \tilde x^1_{i_{n}}
\tilde x^1_{i_{1}}\cdots \tilde x^1_{i_{p-1}}
\tilde x^2_{i_{q+1}}\cdots \tilde x^2_{i_{k-1}}
x^2_{i_{l+1}}\cdots  x^2_{i_{q-1}}
  \right\}.
\end{eqnarray*}
Now, define $$J_1=\left\{ \{2,1\}, \{3,1\}, \{5,4\}, \{6,4\}\right\}=\left\{ E_1, E_2,\tilde E_1, \tilde E_2\right\}$$
and set for any $i=1,\ldots,r$, $$x_{i,E_1}:= x^1_i,\; \; x_{i,\tilde E_1}:= \tilde x^1_i,\; \; x_{i,\tilde E_2}:= \tilde x^2_i,\; \; x_{i,E_2}:= x^2_i.$$ Thus { $R_1 ( P)$ is a function in  four indeterminates,
indexed by the sets of $J_1$, of r-tuples }$(x_{i, E_1})_{1\leq i\leq r}$, $(x_{i,\tilde E_1})_{1\leq i\leq r}$, $(x_{i,\tilde E_2})_{1\leq i\leq r}$ and $(x_{i, E_2})_{1\leq i\leq r}.$

Now let us introduce the following sets of sets.
$$J_2^{1,1}=\left\{ \{8,2,1,19\}, \{9,3,1,19\}, \{11,5,4,19\}, \{12,6,4,19\}\right\},$$
$$J_2^{2,1}=\left\{ \{8,7,1,19\}, \{9,7,1,19\}, \{11,10,4,19\}, \{12,10,4,19\}\right\},$$
$$J_2^{1,2}=\left\{ \{14,2,1,19\}, \{15,3,1,19\}, \{17,5,4,19\}, \{18,6,4,19\}\right\},$$
$$J_2^{2,2}=\left\{ \{14,13,1,19\}, \{15,13,1,19\}, \{17,16,4,19\}, \{18,16,4,19\}\right\},$$
$$J_2^{3,1}=\left\{ \{8,7,20,19\}, \{9,7,20,19\}, \{11,10,20,19\}, \{12,10, 20,19\}\right\},$$ 
$$J_2^{3,2}=\left\{\{14,13,21,19\}, \{15,13,21,19\},\{17,16,21,19\}, \{18,16,21,19\}\right\}$$ 
and 
$$\tilde J_2^{1,1}=\left\{ \{29,23,22,40\}, \{30,24,22,40\}, \{32,26,25,40\}, \{33,27,25,40\}\right\},$$
$$\tilde J_2^{2,1}=\left\{ \{29,28,22,40\}, \{30,28,22,40\}, \{32,31,25,40\}, \{33,31,25,40\}\right\},$$
$$\tilde J_2^{1,2}=\left\{ \{35,23,22,40\}, \{36,24,22,40\}, \{38,26,25,40\}, \{39,27,25,40\}\right\},$$
$$\tilde J_2^{2,2}=\left\{ \{35,34,22,40\}, \{36,34,22,40\}, \{38,37,25,40\}, \{39,37,25,40\}\right\},$$
$$\tilde J_2^{3,1}=\left\{ \{29,28,41,40\}, \{30,28,41,40\}, \{32,31,41,40\}, \{33,31,41,40\}\right\},$$ 
$$\tilde J_2^{3,2}=\left\{\{35,34,42,40\}, \{36,34,42,40\},\{38,37,42,40\}, \{39,37,42,40\}\right\}.$$

$J_2$ is by definition the union of all these sets, that is $J_2$ contains 48 sets of integer numbers lower than 42.
Let $s^1,\ldots, s^{42}$ be free  r-semi-circular systems in $(\mathcal A, \tau)$ and $X_N$ be a r-tuple of independent GUE matrices. Set 
$$A_2=\{T_2=(t_1,t_2,t_3,t_4), 0\leq t_1\leq t_2 \leq t_4, 0\leq t_3\leq t_4\}.$$
Define now, for $T_2=(t_1,t_2,t_3,t_4)\in A_2$, for any $1\leq i\leq r$, for any $I\in J_2$, $I=\{I_1,I_2,I_3,I_4\}$ where the $I_p$'s are integer numbers smaller or equal to 42, the following $\mathcal A_N$-valued random variables:
$$X^{N,T_2}_{i,I} =\sum_{l=1}^4 (e^{-\tilde t_{l-1}}-e^{-\tilde t_l})^{1/2} s_i^{I_l} +e^{-t_4/2}X_N^{(i)},$$
where $\tilde t_1\leq \cdots\leq \tilde t_4$ are the $t_i$'s in increasing order. Set
$$X^{N,T_2}_{1,1}= \left( X^{N,T_2}_{i,I}\right)_{1\leq i\leq r, I\in J_2^{1,1}},\;\;
X^{N,T_2}_{2,1}= \left( X^{N,T_2}_{i,I}\right)_{1\leq i\leq r, I\in J_2^{2,1}},$$
$$X^{N,T_2}_{1,2}= \left( X^{N,T_2}_{i,I}\right)_{1\leq i\leq r, I\in J_2^{1,2}},\; \;
X^{N,T_2}_{2,2}= \left( X^{N,T_2}_{i,I}\right)_{1\leq i\leq r, I\in J_2^{2,2}},$$
$$X^{N,T_2}_{3,1}= \left( X^{N,T_2}_{i,I}\right)_{1\leq i\leq r, I\in J_2^{3,1}},\; \;
X^{N,T_2}_{3,2}= \left( X^{N,T_2}_{i,I}\right)_{1\leq i\leq r, I\in J_2^{3,2}},$$
$$\tilde X^{N,T_2}_{1,1}= \left( X^{N,T_2}_{i,I}\right)_{1\leq i\leq r, I\in \tilde J_2^{1,1}},\; \;
\tilde X^{N,T_2}_{2,1}= \left( X^{N,T_2}_{i,I}\right)_{1\leq i\leq r, I\in \tilde  J_2^{2,1}},$$
$$\tilde X^{N,T_2}_{1,2}= \left( X^{N,T_2}_{i,I}\right)_{1\leq i\leq r, I\in \tilde J_2^{1,2}},\; \;
\tilde X^{N,T_2}_{2,2}= \left( X^{N,T_2}_{i,I}\right)_{1\leq i\leq r, I\in \tilde J_2^{2,2}},$$
$$\tilde X^{N,T_2}_{3,1}= \left( X^{N,T_2}_{i,I}\right)_{1\leq i\leq r, I\in \tilde J_2^{3,1}},\; \;
\tilde X^{N,T_2}_{3,2}= \left( X^{N,T_2}_{i,I}\right)_{1\leq i\leq r, I\in \tilde J_2^{3,2}}.$$
Since there are four elements in each $J_2^{l,p}$, $\tilde J_2^{l,p}$, { each $X^{N,T_2}_{l,p}$ and $\tilde X^{N,T_2}_{l,p}$ is a 4-tuple of $r$-tuples}. There is a natural bijection between $J_1$ and each of the sets $J_2^{1,1},J_2^{2,1} , J_2^{1,2}, J_2^{2,2},J_2^{3,1},J_2^{3,2}$, 
$\tilde J_2^{1,1},\tilde J_2^{2,1} ,\tilde J_2^{1,2}, \tilde J_2^{2,2}, \tilde J_2^{3,1},\tilde J_2^{3,2}$. Therefore, we can evaluate in $(X_{i,I})_{1\leq i \leq r, I\in J_2^{l,p}}$ or $(X_{i,I})_{1\leq i \leq r, I\in \tilde J_2^{l,p}}$, any polynomial with  indeterminates $(\mathbf x_{i,I})_{1\leq i \leq r, I\in J_1}$.\\

Similarly,  for $I$ in $J_1\cup J_2$, and $1\leq i\leq r$,  define   on monomials $M$ 
$$\partial_{i,I} M= \sum_{M=AX_{i,I} B} A\otimes B,$$
$$D_{i,I} M= \sum_{M=AX_{i,I} B}BA $$ and set for any $1\leq i\leq r$,
$$\partial_{i} =\sum_{I \in J_1}\partial_{ i,I},\; D_{ i}=\sum_{I \in J_1}D_{ i,I}.$$
We define, for any polynomial $Q$ in $((\mathbf{x}_{i, E_1})_{1\leq i\leq r }$, $((\mathbf{x}_{i, \tilde E_1})_{1\leq i\leq r }$, $((\mathbf{x}_{i, \tilde E_2})_{1\leq i\leq r }$, 
$((\mathbf{x}_{i,  E_2})_{1\leq i\leq r } )$,  for any 4-tuple $(y^1, \tilde y^1, \tilde y^2, y^2)$ of $4r$-tuples
$$
y^1= ((y^{1}_{i, E_1})_{1\leq i\leq r },((y^{1}_{i, \tilde E_1})_{1\leq i\leq r }, ((y^{1}_{i, \tilde E_2})_{1\leq i\leq r }, ((y^{1}_{i,  E_2})_{1\leq i\leq r } ),
$$ 
$$
\tilde y^1=((\tilde  y^{1}_{i, E_1})_{1\leq i\leq r },((\tilde  y^{1}_{i, \tilde E_1})_{1\leq i\leq r },((\tilde  y^{1}_{i, \tilde E_2})_{1\leq i\leq r }, ((\tilde  y^{1}_{i,  E_2})_{1\leq i\leq r }),
$$
$$
\tilde  y^2=((\tilde  y^{2}_{i, E_1})_{1\leq i\leq r},((\tilde  y^{2}_{i, \tilde E_1})_{1\leq i\leq r},((\tilde  y^{2}_{i, \tilde E_2})_{1\leq i\leq r}, ((\tilde  y^{2}_{i,  E_2})_{1\leq i\leq r} ),
$$
$$
y^2= ((  y^{2}_{i, E_1})_{1\leq i\leq r },((  y^{2}_{i, \tilde E_1})_{1\leq i\leq r }, (( y^{2}_{i, \tilde E_2})_{1\leq i\leq r }, (( y^{2}_{i,  E_2})_{1\leq i\leq r } ),
$$
\begin{equation}\label{rrr}
R_2^{(1)}\!(Q)(y^1\!, \tilde y^1\!, \tilde y^2\!, y^2)=\!\! \sum_{1\leq i,j\leq r}\!\!
\#^3\left[ ({\sf flip} \circ \partial_{j}) \otimes ({\sf flip} \circ \partial_{j})( \partial_{i}D_{i}(Q)) \right](y^1, \tilde y^1, \tilde y^2, y^2),
\end{equation}
\begin{equation}\label{rrr2}
R_2^{(2)}\!(Q)(y^1, \tilde y^1, \tilde y^2, y^2)=\!\!\sum_{1\leq i,j\leq r}\! \sum_{\tiny{ \begin{array}{cc}I,K\in J_1\\ I_1=K_1\end{array}}}
\hspace*{-0.4cm}\#^3\left[ ({\sf flip} \circ\partial_{j,I}) \otimes ({\sf flip} \circ\partial_{j,K})( \partial_{i}D_{i}(Q)) \right](y^1, \tilde y^1, \tilde y^2, y^2),
\end{equation}
\begin{equation}\label{rrr3}
R_2^{(3)}\!(Q)(y^1, \tilde y^1, \tilde y^2, y^2)=\!\!\sum_{1\leq i,j\leq r}\! \sum_{\tiny{ \begin{array}{cc}I,K\in J_2\\ I_2=K_2 \end{array}}}
\hspace*{-0.4cm}\#^3\left[ ({\sf flip} \circ\partial_{j,I}) \otimes ({\sf flip} \circ\partial_{j,K})( \partial_{i}D_{i}(Q)) \right](y^1, \tilde y^1, \tilde y^2, y^2).
\end{equation}
Now, define the linear form $\nu_2^{(N)}$ on $\mathbb{C} \langle\underline{\mathbf x}_{r}\rangle$ by setting for any $P\in \mathbb{C} \langle\underline{\mathbf x}_{r}\rangle$,
\begin{eqnarray*}\nu_2^{(N)}(P)&=& \frac{1}{4}\int_{A_2}  e^{-t_4-t_3-t_2-t_1}\\
&&\left\{\1_{[t_2,t_4]}(t_3) \mathbb{E} \left( \tau_N(
 R_2^{(1)}(R_1(P))
({X^{N,T_2}_{3,1}, \tilde X^{N,T_2}_{3,1}, \tilde X^{N,T_2}_{3,2},X^{N,T_2}_{3,2}})\right)
\right.\\
&&+\1_{[0,t_1]}(t_3)
\mathbb{E} \left( \tau_N(
R_2^{(2)}((R_1(P)))(X^{N,T_2}_{1,1}, \tilde X^{N,T_2}_{1,1}, \tilde X^{N,T_2}_{1,2},X^{N,T_2}_{1,2})\right)\\
&& \left.
+\1_{[t_1,t_2]}(t_3)
\mathbb{E} \left( \tau_N(
R_2^{(3)}((R_1(P)))(X^{N,T_2}_{2,1}, \tilde X^{N,T_2}_{2,1}, \tilde X^{N,T_2}_{2,2},X^{N,T_2}_{2,2})\right)
\right\}\\
& & {\rm d}t_1{\rm d}t_2{\rm d}t_3{\rm d}t_4.
\end{eqnarray*}

The following proposition is  a corollary of \cite[Proposition 3.7]{P}.
\begin{proposition}\label{niveau2}\cite[Proposition 3.7]{P}
Let $X_1, \ldots, X_r$ be independent GUE's.
For any $P\in \mathbb{C}\langle \underline{\mathbf x}_{r}\rangle$, one has 
\begin{eqnarray*}\mathbb{E}\left[{\rm tr}_N (P(X_1,\ldots,X_r))\right] 
& =  &\tau(P({\bf s}_1,\dots,{\bf s}_r))+\frac{\nu_1^{(N)}(P)}{N^2}
\\&=&\tau(P({\bf s}_1,\dots,{\bf s}_r))+\frac{\nu_1(P)}{N^2}+\frac{\nu_2^{(N)}(P)}{N^4}.\end{eqnarray*}
 \end{proposition}
\begin{corollary}\label{niveau2exp}
Let $X_1, \ldots, X_r$ be independent GUE's.
For any $n \in \N^*$, any $I=(\iota_1, \ldots, \iota_n)\in \{1,\ldots,r\}^n$, any ${\bf y}=(y_1,\ldots, y_n)\in \R^n$, one has 
\begin{eqnarray*}\lefteqn{\mathbb{E}\left[{\rm tr}_N (e^{iy_1X_{\iota_1}} e^{iy_2X_{\iota_2}}\cdots e^{iy_nX_{\iota_n}})\right] }\\
&\! = &\!\!\tau(e^{iy_1{\bf s}_{\iota_1}} e^{iy_2{\bf s}_{\iota_2}}\cdots e^{iy_n{\bf s}_{\iota_n}})\\
& &\!\!\!\mbox{}+\frac{1}{2N^2}\!\int_0^{+\infty}\!\int_0^{t_2}\!e^{-t_2-t_1}\tau_N\left(F_{1,(I,{\bf y})}(z^1_{t_1}(N),\tilde{z}^1_{t_1}(N),\tilde z^2_{t_1}(N), z_{t_1}^2(N))
\right){\rm d}t_1{\rm d}t_2 \\
&\! =  &\!\!\tau(e^{iy_1{\bf s}_{\iota_1}} e^{iy_2{\bf s}_{\iota_2}}\cdots e^{iy_n{\bf s}_{\iota_n}})\\
& &\!\!\!\mbox{}+\frac{1}{2N^2}\!\int_0^{+\infty}\!\int_0^{t_2}\!e^{-t_2-t_1}\tau\left(F_{1,(I,{\bf y})}( z^1_{t_1}, \tilde z^1_{t_1}, \tilde z^2_{t_1}, z_{t_1}^2)\right){\rm d}t_1 
{\rm d}t_2\\
& &\!\!\!\mbox{}+\frac{1}{4N^4}\!\int_{A_2}\!\! e^{-t_4-t_3-t_2-t_1}\!\left\{\!\1_{[t_2,t_4]}(t_3) \mathbb{E} \left(\! \tau_N( F_{2,(I,{\bf y})}^{(1)}({X^{N,T_2}_{3,1}\!, \tilde 
X^{N,T_2}_{3,1}\!, \tilde X^{N,T_2}_{3,2}\!,X^{N,T_2}_{3,2}})\!\right)\right.\\
& &\!\!\!\mbox{}+\1_{[0,t_1]}(t_3)\mathbb{E} \left( \tau_N(
F_{2,(I,{\bf y})}^{(2)}(X^{N,T_2}_{1,1}, \tilde X^{N,T_2}_{1,1}, \tilde X^{N,T_2}_{1,2},X^{N,T_2}_{1,2})\right)\\
& & \!\!\!\left.\mbox{}+\1_{[t_1,t_2]}(t_3)\mathbb{E} \left( \tau_N(
F_{2,(I,{\bf y})}^{(3)}(X^{N,T_2}_{2,1}, \tilde X^{N,T_2}_{2,1}, \tilde X^{N,T_2}_{2,2},X^{N,T_2}_{2,2})\right)
\right\}\\
& & {\rm d}t_1{\rm d}t_2{\rm d}t_3{\rm d}t_4,
\end{eqnarray*}
where $F_{1,(I,{\bf y})}$, $F_{2,(I,{\bf y})}^{(1)}, F_{2,(I,{\bf y})}^{(2)}, F_{2,(I,{\bf y})}^{(3)}$ are the noncommutative functions 
$$F_{1, (I,{\bf y})}=\sum_{m_1,m_2,\dots,m_n=0}^\infty i^{m_1+\cdots+m_n}\frac{y_1^{m_1}\cdots y_n^{m_n}}{m_1!m_2!\cdots m_n!}
R_1\left(M_{I, m_1,\ldots,m_n}\right),$$
$$F_{2, (I,{\bf y})}^{(1)}=\sum_{m_1,m_2,\dots,m_n=0}^\infty i^{m_1+\cdots+m_n}\frac{y_1^{m_1}\cdots y_n^{m_n}}{m_1!m_2!\cdots m_n!}
R_2^{(1)}\left(M_{I, m_1,\ldots,m_n}\right),$$
$$F_{2, (I,{\bf y})}^{(2)}=\sum_{m_1,m_2,\dots,m_n=0}^\infty i^{m_1+\cdots+m_n}\frac{y_1^{m_1}\cdots y_n^{m_n}}{m_1!m_2!\cdots m_n!}
R_2^{(2)}\left(M_{I, m_1,\ldots,m_n}\right),$$
$$F_{2, (I,{\bf y})}^{(3)}=\sum_{m_1,m_2,\dots,m_n=0}^\infty i^{m_1+\cdots+m_n}\frac{y_1^{m_1}\cdots y_n^{m_n}}{m_1!m_2!\cdots m_n!}
R_2^{(3)}\left(M_{I, m_1,\ldots,m_n}\right),$$
and  $M_{I, m_1,\ldots,m_n}$ is the monomial $(\mathbf{x}_1,\ldots,\mathbf{x}_r)\mapsto
\mathbf{x}_{\iota_1}^{m_1}\mathbf{x}_{\iota_2}^{m_2}\cdots \mathbf{x}_{\iota_n}^{m_n},$
and the series converge in operator norm, uniformly on norm balls of any given, fixed radius $($as uniform limits of noncommutative polynomials - hence noncommutative functions - 
on any given norm-bounded set, the functions $F_{1,(I,{\bf y})}$, $F_{2,(I,{\bf y})}^{(1)}, F_{2,(I,{\bf y})}^{(2)}, F_{2,(I,{\bf y})}^{(3)}$ are automatically noncommutative 
in the sense described in Section \ref{sec:ncf}$)$.
\end{corollary}

\begin{proof}
\begin{eqnarray*}
\lefteqn{e^{iy_1X_{\iota_1}} e^{iy_2X_{\iota_2}}\cdots e^{iy_nX_{\iota_n}}}\\
&=& \sum_{m_1,m_2,\dots,m_n=0}^\infty\frac{1}{m_1!m_2!\cdots m_n!}
(iy_1X_{\iota_1})^{m_1}(iy_2X_{\iota_2})^{m_2}\cdots (iy_nX_{\iota_n})^{m_n},
\end{eqnarray*}
where the series is absolutely convergent. Thus,
\begin{eqnarray*}
\lefteqn{\tr_N \left[e^{iy_1X_{\iota_1}} e^{iy_2X_{\iota_2}}\cdots e^{iy_nX_{\iota_n}}\right]}\\
&=& \sum_{m_1,m_2,\dots,m_n=0}^\infty\frac{1}{m_1!m_2!\cdots m_n!}\tr_N \left[
(iy_1X_{\iota_1})^{m_1}(iy_2X_{\iota_2})^{m_2}\cdots (iy_nX_{\iota_n})^{m_n}\right].
\end{eqnarray*}
Now, we have $$\left|\tr_N \left[
(iy_1X_{\iota_1})^{m_1}(iy_2X_{\iota_2})^{m_2}\cdots (iy_nX_{\iota_n})^{m_n}\right]\right|\leq 
(|y_1|\|X_{\iota_1}\|)^{m_1}\cdots (|y_n|\|X_{\iota_n}\|)^{m_n}.$$
Define for $j=1,\ldots,r$,  $d_j=\sum_{l,  i_l=j}|y_l|$. Applying Fubini's theorem for positive functions, we have that 
$$\sum_{m_1,m_2,\dots,m_n=0}^\infty\frac{1}{m_1!m_2!\cdots m_n!}\mathbb{E}\left[(|y_1|\|X_{\iota_1}\|)^{m_1}\cdots (|y_n|\|X_{\iota_n}\|)^{m_n}\right]$$\begin{eqnarray*} &=&\mathbb{E}\left( e^{|y_1|\|X_{\iota_1}\|} \cdots e^{|y_n|\|X_{\iota_n}\|}\right)\\
&=&\mathbb{E}( e^{d_1\|X_{1}\|}) \cdots \mathbb{E}( e^{d_r\|X_{r}\|})\\
&\leq &2^r \prod_{i=1}^r \mathbb{E}( \Tr e^{d_i X_{i}})\\&\leq & 2^r N^r \prod_{i=1}^r e^{2d_i +\frac{d_i^2}{2N}},
\end{eqnarray*}
where, in the two last lines,  we use the inequalities (5.4) and (5.3) in \cite{HT}.
Therefore we can deduce that $$\sum_{m_1,m_2,\dots,m_n=0}^\infty\frac{1}{m_1!m_2!\cdots m_n!}\tr_N \left[
(iy_1X_{\iota_1})^{m_1}(iy_2X_{\iota_2})^{m_2}\cdots (iy_nX_{\iota_n})^{m_n}\right]$$
is absolutely convergent and then, by Fubini theorem, that  
\begin{eqnarray*} 
\lefteqn{\mathbb{E}\tr_N \left[e^{iy_1X_{\iota_1}} e^{iy_2X_{\iota_2}}\cdots e^{iy_nX_{\iota_n}}\right]}\\
&=& \sum_{m_1,m_2,\dots,m_n=0}^\infty\frac{1}{m_1!m_2!\cdots m_n!}\mathbb{E}\tr_N \left[
(iy_1X_{\iota_1})^{m_1}(iy_2X_{\iota_2})^{m_2}\cdots (iy_nX_{\iota_n})^{m_n}\right].
\end{eqnarray*}
The result readily follows from Proposition \ref{niveau2}.

\end{proof}
\subsection{Parraud's formulae for polynomials in Cayley transforms}

Consider first $\nu_1$ and $\nu_1^{(N)}$. The non-obvious part of Parraud's formula consists of 
$$
\#^3\left\{ \left[({\sf flip}\circ \partial_{j})\otimes ({\sf flip}\circ \partial_j)\right](\partial_{i}(D_{i}P))\right\},\quad 1\le i,j\le r,
$$
where $D_i=\mathsf{ev}_1\circ\mathsf{flip}\circ\partial_j$. We intend to write this formula in terms of the difference-differential operators,
applied to noncommutative functions in $r$ noncommuting variables. Instead of the polynomial $P$, we consider an arbitrary polynomial in Cayley transforms and their inverses
(noncommutative Laurent polynomials in the Cayley transform of each of the $r$ selfadjoint indeterminates), defined on an open noncommutative subset of $\mathcal A^r$. 
Such a function $f$ is given by a (finite) formula in $\mathbf x_1,\dots,\mathbf x_r$. Thus, let us apply $\#^3\left\{ \left[({\sf flip}\circ \partial_{j})\otimes ({\sf flip}\circ \partial_j)
\right](\partial_{i}(D_{i}\ \cdot\ ))\right\}$ to our $f$, but via the identification with the difference-differential operator. Thus, $(D_if)(s_1,\dots,s_r)$ is expressed as the formula for 
$\Delta_if(s_1,\dots,s_r;s_1,\dots,s_r)(1)$ when the expression of $f$ is viewed as being in the opposite algebra $\mathcal A^{\rm op}$. This is a new noncommutative function in 
$r$ variables, which we denote by $g(s_1,\dots,s_r)$. The object obtained through the application of $\partial_i$ to $g$ would be expected to be identified with 
$\Delta_ig(s_1,\dots,s_r;s_1,\dots,s_r)(\cdot)$. However, since one may act on the two tensor factors independently with various operations (and Parraud's formula 
requires us to do so), the correct interpretation is for $\partial_ig$ to be identified with $((s_1,\dots,s_r);(\tilde{s}_1,\dots,\tilde{s}_r))\mapsto\Delta_ig(s_1,\dots,s_r;\tilde{s}_1,
\dots,\tilde{s}_r)(\cdot)$, a noncommutative function in $(s_1,\dots,s_r)$ and $(\tilde{s}_1,\dots,\tilde{s}_r)$ with values in a space of linear maps - different variables, 
although the evaluation can be performed in the same $r$-tuple as well. It is helpful to recall how this formula comes from the evaluations on upper triangular matrices 
as seen in Section \ref{sec:ncf}:
\begin{eqnarray*}
\lefteqn{g\left(\begin{bmatrix} {s}_1 & \delta_{i,1}\, \cdot \\ 0 & \tilde{s}_1\end{bmatrix},\dots,\begin{bmatrix} {s}_r & \delta_{i,r}\, \cdot \\ 0 & \tilde{s}_r\end{bmatrix}\right)=}\\
& & \begin{bmatrix} g({s}_1,\dots,s_r) & \Delta_{i}g(s_1,\dots,s_r;\tilde{s}_1,\dots,\tilde{s}_r)(\cdot) \\ 0 & g(\tilde{s}_1,\dots,\tilde{s}_1)\end{bmatrix}.
\end{eqnarray*}
As mentioned in Section \ref{sec:ncf}, this is a sample of a {\em higher order} noncommutative function (see \cite[Chapter 3]{KVV}). It is important that one may now view the
variables $\underline{s}_r$ and $\underline{\tilde{s}}_r$ as independent variables: Parraud's formula requires us now to apply $\Delta_j$ with respect to the variables $s_1,\dots,s_r$
and the same for $\tilde{s}_1,\dots,\tilde{s}_r$, but each of the two correspondences $(s_1,\dots,s_r)\mapsto\Delta_{i}g(s_1,\dots,s_r;\tilde{s}_1,\dots,\tilde{s}_r)(\cdot)$
and $(\tilde{s}_1,\dots,\tilde{s}_r)\mapsto\Delta_{i}g(s_1,\dots,s_r;\tilde{s}_1,\dots,\tilde{s}_r)(\cdot)$ viewed as for functions on noncommutative
subsets of $\mathcal A^{\rm op}$. Specifically, with the usual notation $\underline{s}_r=(s_1,\dots,s_r),\underline{\tilde{s}}_r=(\tilde{s}_1,\dots,\tilde{s}_r)$, 
the noncommutative function $h(\underline{s}_r;\underline{\tilde{s}}_r)=\Delta_{i}g(s_1,\dots,s_r;\tilde{s}_1,\dots,\tilde{s}_r)(\cdot)$ acts on $\coprod_{n,m}
M_{n\times m}(\mathcal A)$, from the left with the variables $\underline{s}_r$ and from the right with the variables $\underline{\tilde{s}}_r$.
Denote by $\Delta_j$ the partial difference-differential operator with respect to the variable $s_j$ and by $\tilde{\Delta}_j$ the one with respect to 
$\tilde{s}_j$ (one could a priori think of $\tilde{\Delta}_j$ as $\Delta_{j+r}$, but especially in this context it would be the wrong view, since it
would obscure the fact that the $s$ variables act on the left and the $\tilde{s}$ ones on the right). We view now $h$ as being defined on 
noncommutative subsets that live in $\mathcal A^{\rm op}$ - that is, in variables $\underline{s}_r^{\rm op}$ and $\underline{\tilde{s}}_r^{\rm op}$ - and apply
$\Delta_j\tilde{\Delta}_jh(\underline{s^1}_r^{\rm op},\underline{s^2}_r^{\rm op};\underline{\tilde{s^2}}_r^{\rm op},\underline{\tilde{s^1}}_r^{\rm op})=\tilde{\Delta}_j\Delta_jh(
\underline{s^1}^{\rm op}_r,\underline{s^2}_r^{\rm op};\underline{\tilde{s^2}}_r^{\rm op},\underline{\tilde{s^1}}_r^{\rm op})$ to obtain a 
third order (trilinear-valued) noncommutative map in four $r$-tuples of variables. Finally, we drop the op, that is, we let 
$k(\underline{s^1}_r,\underline{s^2}_r;\underline{\tilde{s^2}}_r,\underline{\tilde{s^1}}_r)=
\Delta_j\tilde{\Delta}_jh(\underline{s^1}_r,\underline{s^2}_r;\underline{\tilde{s^2}}_r,\underline{\tilde{s^1}}_r)$, still a third order noncommutative map. 
Then $R_1(f)$ (see \eqref{defRP}) is the noncommutative map resulting after  evaluating the tri-linear map thus obtained in $(1,1,1).$
$R_2^{(1)}(R_1(f))$, $R_2^{(2)}(R_1(f))$ and $ R_2^{(3)}(R_1(f))$ (see \eqref{rrr}, \eqref{rrr2}, \eqref{rrr3}) are similarly the noncommutative maps  obtained by iterating 
the process.  Since sums, products, inverses, and compositions (when well-defined) of noncommutative functions are noncommutative functions \cite[Section 2.3]{KVV},
one can easily see  that $R_1(\Psi(\mathbf x_{i_1})^{\epsilon_1}\cdots \Psi(\mathbf x_{i_n})^{\epsilon_n})$ 
is a  noncommutative function on $\mathcal I_{4r}(\mathcal A)$, more precisely  a polynomial in the Cayley transforms and their inverses 
of $4r$ indeterminates; similarly   for $R$ in $ \{R_2^{(1)}, R_2^{(2)}, R_2^{(3)}\}$, one can easily see  that $R(R_1(\Psi(\mathbf x_{i_1})^{\epsilon_1}\cdots \Psi(\mathbf 
x_{i_n})^{\epsilon_n}))$, is a  noncommutative function on $\mathcal I_{16r}(\mathcal A)$, more precisely a Laurent polynomial in the Cayley transforms  
of $16r$ indeterminates.

Now, we rewrite the power series expansion 
$\frac{Z+\epsilon i}{Z-\epsilon i}=1-2\sum_{n=0}^\infty(-i\epsilon Z)^n=-1+\sum_{n=1}^\infty2(-1)^{n+1}(i\epsilon Z)^n=\sum_{n=0}^\infty c_{n}(i\epsilon Z)^n$,
where $c_n=(-1)^{n+1}\min\{n+1,2\}$ ($c_0=-1,c_1=2,c_2=-2,c_3=2,\dots$). 
Then, for any element $\mathbf x$ in $I_{r}(\mathcal A)$, such that $\|{\mathbf x}_i\|<1$, $i=1,\ldots,r$,
\begin{eqnarray*}
\lefteqn{\Psi({\mathbf x}_{\iota_1})^{\epsilon_1}\cdots\Psi({\mathbf x}_{\iota_n})^{\epsilon_n}}\\
&=& \sum_{m_1,m_2,\dots,m_n=0}^\infty c_{m_1}c_{m_2}\cdots c_{m_n}
(i\epsilon_1 {\mathbf x}_{\iota_1})^{m_1}\cdots (i\epsilon_n {\mathbf x}_{\iota_n})^{m_n}.
\end{eqnarray*}
As one may verify by using the explicit formulas provided in Sections \ref{pol}--\ref{domains}, any of $R\in \{ R_1, R_2^{(1)}, R_2^{(2)}, R_2^{(3)}\}$ is well-defined on 
rational functions, and (in the obvious sense made clear in Section \ref{sec:ncf}) the domain does not change - see Section \ref{domains}.  In particular, for a product of rational 
functions $\Psi(\cdot)^{\pm1}$, one has
\begin{eqnarray}
\lefteqn{\!\!R\left[\Psi({\mathbf x}_{\iota_1})^{\epsilon_1}\cdots\Psi({\mathbf x}_{\iota_n})^{\epsilon_n}\right]}\nonumber \\
&\!\!=& \!\sum_{m_1,m_2,\dots,m_n=0}^\infty c_{m_1}c_{m_2}\cdots c_{m_n}i^{m_1+\ldots+m_n}\epsilon_1^{m_1}\cdots \epsilon_n^{m_n}
R\left[{\mathbf x}_{\iota_1}^{m_1}\cdots {\mathbf x}_{\iota_n}^{m_n}\right] \label{passageR}.
\end{eqnarray}

\begin{lemma}\label{c}
$$
c_{m_1}c_{m_2}\cdots c_{m_n}=\int_{(-\infty,0]^n}\!\!\frac{e^{y_1}y_1^{m_1}e^{y_2}y_2^{m_2}\cdots e^{y_n}y_n^{m_n}}{m_1!m_2!\cdots m_n!}\,\prod_{j=1}^n
({\rm d}\delta_0(y_j)-2{\rm d}y_j).
$$
\end{lemma}
\begin{proof}
Clearly if all of $m_1,m_2,\dots,m_n\neq0$, then the integral with respect to the Dirac measure is zero, so that the equality to be proved becomes
$$
c_{m_1}c_{m_2}\cdots c_{m_n}=(-2)^n\int_{(-\infty,0]^n}\!\!\frac{e^{y_1}y_1^{m_1}e^{y_2}y_2^{m_2}\cdots e^{y_n}y_n^{m_n}}{m_1!m_2!\cdots m_n!}\,
{\rm d}y_1{\rm d}y_2\cdots {\rm d}y_n,
$$
which is obviously true. If some of the $m_j$'s are equal to zero, then we separate them form the rest. Since the scalar coefficients do commute, we may assume without
loss of generality that $m_1=m_2=\cdots=m_k=0,m_{k+1},\dots,m_n>0$. Recalling that $c_0=-1$, the equality to be proved becomes
\begin{eqnarray*}
\lefteqn{(-1)^kc_{m_{k+1}}\cdots c_{m_n}=}\\
& &\!\!\!\!\!\!\!(-2)^{n-k}\!\!\int_{(-\infty,0]^n}\!\!\!e^{y_1}\cdots e^{y_k}\frac{e^{y_{k+1}}y_{k+1}^{m_{k+1}}\cdots e^{y_n}y_n^{m_n}}{m_{k+1}!\cdots m_n!}\!
\prod_{j=1}^k\!({\rm d}\delta_0(y_j)\!-\!2{\rm d}y_j){\rm d}y_{k+1}\!\cdots{\rm d}y_n.
\end{eqnarray*}
This is again obvious: since the integral factors, one needs only notice that $\int_{(-\infty,0]^k}e^{y_1}\cdots e^{y_k}\prod_{j=1}^k
({\rm d}\delta_0(y_j)-2{\rm d}y_j)=(1-2)^k=(-1)^k$.
\end{proof}

\begin{proposition}\label{comparaisonR}
We have the following identity between the two noncommutative functions on $\mathcal I_{r}({\mathcal A})$:
$$
R_1\left[\Psi({\mathbf x}_{\iota_1})^{\epsilon_1}\cdots\Psi({\mathbf x}_{\iota_n})^{\epsilon_n}\right]=
\int_{(-\infty,0]^n} F_{1,(I, \bf \epsilon y)} e^{y_1+\cdots+ y_n} \prod_{i=1}^n ({\rm d}\delta_0(y_i) -2{\rm d}y_i).
$$
For $R\in \{  R_2^{(1)}, R_2^{(2)}, R_2^{(3)}\}$, we have the following identity between the two noncommutative functions on $\mathcal I_{r}({\mathcal A})$:
$$
R\left[\Psi({\mathbf x}_{\iota_1})^{\epsilon_1}\cdots\Psi({\mathbf x}_{\iota_n})^{\epsilon_n}\right]=
\int_{(-\infty,0]^n} F_{I, \bf \epsilon y}\, e^{y_1+\cdots +y_n} \prod_{i=1}^n ({\rm d}\delta_0(y_i) -2{\rm d}y_i),
$$
with the corresponding $F_{I,\epsilon\bf y}\in\{F_{2,(I,{\bf\epsilon y})}^{(1)},F_{2,(I,{\bf\epsilon y})}^{(2)},F_{2,(I,{\bf\epsilon y})}^{(3)}\}$ defined in Corollary \ref{niveau2exp}.
\end{proposition}

The functions in the first equality depend on $4r$ variables, while the ones in the second depend on $16r$ variables. When we refer to them as defined on 
$\mathcal I_r(\mathcal A)$ - i.e. as functions of $r$ variables - we mean that they are `produced' via various operations, including applications of the difference-differential operator,
out of functions of $r$ variables; indeed, if $f(\cdot)$ is a noncommutative function of $r$ variables, then $\Delta f(\,\cdot\,;\,\cdot\,)$ (or, better said, a restriction of it) is a function 
of $2r$ variables: however, it comes from a noncommutative function of $r$ variables, function which does encode all the information necessary in order to know $\Delta f(\,\cdot\,;\,
\cdot\,)$. The noncommutative function structure is essential in order for this to be possible.
\begin{proof}[Proof of Proposition \ref{comparaisonR}]
For any $\mathbf x\in I_{r}(\mathcal A)$ such that $\|{\mathbf x}_i\|<1$, $i=1,\ldots,r$, the  identity follows from \eqref{passageR} and Lemma \ref{c}. 
The left-hand side is a noncommutative rational function simply by definition and by the fact that all of the $R$s preserve rationality. Each of the functions $F$ is an entire
noncommutative function, whose Taylor-Taylor series converges uniformly in norm on bounded subsets. The integrals in the right-hand side of the two formulas in 
Proposition \ref{comparaisonR} converge when evaluated in elements from $\mathcal I_{r}({\mathcal A})$ (in the sense previously specified). 
Since both sides of the identity are locally bounded noncommutative functions on $\mathcal I_r(\mathcal A)$ and thus automatically analytic, the classical identity principle for Banach 
space-valued analytic functions allows us to conclude.
\end{proof}

Now, apply Corollary \ref{niveau2exp} with ${\bf \epsilon y}$ instead of ${\bf y}$, multiply the identity by 
$e^{y_1+\cdots+ y_n}$, integrate on $(-\infty, 0]^n$ against  the measure $\prod_{i=1}^n ({\rm d}\delta_0(y_i)-{\rm d}y_i)$.  Since 
 for any bounded selfadjoint operator $x$, we have
$$
\Psi(x)^\epsilon= \int_{-\infty}^0 e^{(i\epsilon x+1)y}\,({\rm d}\delta_0(y)-{\rm d}y),
$$
 and using Proposition \ref{comparaisonR}, we obtain the following
\begin{proposition}\label{dvptpourCayley}
Let $X_1, \ldots, X_r$ be independent GUE's.
For any $P\in \mathbb{C}\langle \underline{\mathbf x}_{r},\underline{\mathbf x}_{r}^{-1} \rangle$, one has 
\begin{eqnarray*}
\mathbb{E}\left[{\rm tr}_N (P(\Psi(X_1),\ldots,\Psi(X_r))\right] &=&  \tau(P(\Psi({\bf s}_1),\dots,\Psi({\bf s}_r))+\frac{\hat \nu_1^{(N)}(P)}{N^2}\\ 
&= & \tau(P(\Psi({\bf s}_1),\dots,\Psi({\bf s}_r))+\frac{\hat \nu_1(P)}{N^2}+\frac{\hat \nu_2^{(N)}(P)}{N^4}.
\end{eqnarray*}
$$
\hat\nu_1(P)=\frac{1}{2}\int_0^{+\infty}\!\!\!\int_0^{t_2}\!e^{-t_2-t_1}\tau\left(R_1[P(\Psi({\bf x}_1),\dots,\Psi({\bf x}_r))](z^1_{t_1},\tilde z^1_{t_1},\tilde z^2_{t_1}, 
z_{t_1}^2)\right){\rm d}t_1{\rm d}t_2,
$$
\begin{align*}
&\hat\nu_1^{(N)}(P)=\\
&\frac12\int_0^{+\infty}\!\!\!\int_0^{t_2}\!e^{-t_2-t_1}\tau_N\!\left(\!R_1[P(\Psi({\bf x}_1),\dots,\Psi({\bf x}_r))]( z^1_{t_1}(N),\tilde z^1_{t_1}(N),\tilde z^2_{t_1}(N),
z_{t_1}^2(N))\right){\rm d}t_1{\rm d}t_2,
\end{align*}
\begin{eqnarray*}
\lefteqn{\hat \nu_2^{(N)}(P)=\frac{1}{4}\int_{A_2}  e^{-t_4-t_3-t_2-t_1}}\\
&&\mbox{}\times\left\{\!\1_{[t_2,t_4]}(t_3) \mathbb{E}\!\left(\tau_N( R_{2}^{(1)}\!\left\{R_1[P(\Psi({\bf x}_1),\dots,\Psi({\bf x}_r)]\right\}\!
({X^{N,T_2}_{3,1}, \tilde X^{N,T_2}_{3,1}, \tilde X^{N,T_2}_{3,2},X^{N,T_2}_{3,2}})\right)\right.\\
&&\mbox{}+\1_{[0,t_1]}(t_3)
\mathbb{E}\!\left(\tau_N(R_{2}^{(2)}\!\left\{R_1[P(\Psi({\bf x}_1),\dots,\Psi({\bf x}_r)]\right\}\!(X^{N,T_2}_{1,1}, \tilde X^{N,T_2}_{1,1}, \tilde X^{N,T_2}_{1,2},
X^{N,T_2}_{1,2})\right)\\
&& \left.
\mbox{}+\1_{[t_1,t_2]}(t_3)\mathbb{E}\!\left(\tau_N(R_{2}^{(3)}\!\left\{R_1[P(\Psi({\bf x}_1),\dots,\Psi({\bf x}_r)]\right\}\!(X^{N,T_2}_{2,1},\tilde X^{N,T_2}_{2,1},
\tilde X^{N,T_2}_{2,2},X^{N,T_2}_{2,2})\right)\!\right\}\\
&&{\rm d}t_1{\rm d}t_2{\rm d}t_3{\rm d}t_4.
\end{eqnarray*}
 \end{proposition}

\subsection{Parraud's formulae on rational functions}

In this section we concern ourselves with the extension of Parraud's functionals $\nu_j,\nu_j^{(N)},j=1,2,$ to rational expressions. 
In itself, the extension is trivial, and would not merit a separate section, except that in our case the rational functions are defined on
(subsets of) $M_m(\mathbb C)\otimes\mathcal A\otimes\mathcal A$ and we apply the functionals to one of the last two tensor coordinates at a time.
Since we find it convenient (although not necessary) to have explicit formulae for the resulting object, we are forced to use at some points
the opposite algebra structure on $\mathcal A$. Thus, we recall from Section \ref{pol} that $\partial_j$ acts on
$\mathbb C\langle\underline{\bf x}_r\rangle^{\rm op}$ as $\mathsf{flip}\circ\partial_j$ acts on $\mathbb C\langle\underline{\bf x}_r\rangle$;
by that we mean that $\mathbb C\langle\underline{\bf x}_r\rangle^{\rm op}=\mathbb C\langle{\bf x}_1^{\rm op},\cdots,{\bf x}_r^{\rm op}\rangle$
is a vector space whose basis is given by monomials which are denoted thus: $(\mathbf{x}_{i_1}\mathbf{x}_{i_2}\cdots\mathbf{x}_{i_n})^{\rm op}=
\mathbf{x}_{i_n}^{\rm op}\cdots\mathbf{x}_{i_2}^{\rm op}\mathbf{x}_{i_1}^{\rm op}$. Then $\partial_j\colon
\mathbb C\langle\underline{\bf x}_r\rangle^{\rm op}\to\mathbb C\langle\underline{\bf x}_r\rangle^{\rm op}\otimes\mathbb C\langle\underline{\bf x}_r\rangle^{\rm op}$
is the free difference quotient (``differentiation'') with respect to $\mathbf x_j^{\rm op}$ and hence it acts by
$\partial_j(\mathbf{x}_{i_1}\mathbf{x}_{i_2}\cdots\mathbf{x}_{i_n})^{\rm op}=
\partial_j(\mathbf{x}_{i_n}^{\rm op}\cdots\mathbf{x}_{i_2}^{\rm op}\mathbf{x}_{i_1}^{\rm op})=\sum_{k\colon i_k=j}
\mathbf{x}_{i_n}^{\rm op}\cdots\mathbf{x}_{i_{k+1}}^{\rm op}\otimes\mathbf{x}_{i_{k-1}}^{\rm op}\cdots\mathbf{x}_{i_1}^{\rm op}\allowbreak
=\sum_{k\colon i_k=j}(\mathbf{x}_{i_{k+1}}\cdots\mathbf{x}_{i_n})^{\rm op}\otimes(\mathbf{x}_{i_1}\cdots\mathbf{x}_{i_{k-1}})^{\rm op}$.
When viewed as an element in $\mathbb C\langle\underline{\bf x}_r\rangle\otimes\mathbb C\langle\underline{\bf x}_r\rangle$, as opposed to an element in
$\mathbb C\langle\underline{\bf x}_r\rangle^{\rm op}\otimes\mathbb C\langle\underline{\bf x}_r\rangle^{\rm op}$, we recognize here the $\mathsf{flip}$ 
operation applied to $\partial_j(\mathbf{x}_{i_1}\mathbf{x}_{i_2}\cdots\mathbf{x}_{i_n}).$
Of course, this extends to rational functions. The identification of $\mathsf{ev}_\cdot\circ\partial_j$ with $\Delta_j$ from 
Section \ref{Sec:DD-FDQ} extends to noncommutative rational functions in the sense of \cite{KVV} defined on 
noncommutative subsets in $\mathcal A^{\rm op}$ when $\partial_j$ is viewed as acting on the opposite algebra. 

Consider first $\nu_1$ and $\nu_1^{(N)}$. The non-obvious part of Parraud's formula consists of 
$$
\#^3\left\{ \left[({\sf flip}\circ \partial_{j})\otimes ({\sf flip}\circ \partial_j)\right](\partial_{i}(D_{i}P))\right\},\quad 1\le i,j\le r,
$$
where $D_i=\mathsf{ev}_1\circ\mathsf{flip}\circ\partial_j$. We intend to write this formula in terms of the difference-differential operators,
applied to noncommutative functions in $r$ noncommuting variables. Thus, instead of the polynomial $P$, we consider an arbitrary noncommutative 
rational function $f$ defined on an open noncommutative subset of $\mathcal A^r$. This function is given by a (finite) formula in $\mathbf x_1,\dots,\mathbf x_r$,
and, in order to cover the case of $(z-\mathcal S)^{-1}$, we assume that it might have non-scalar coefficients; an example would be
$(t_1,\dots,t_r)\mapsto
\Big((z-\xi)\otimes1\otimes1-\sum_{i=1}^{r}(\gamma_i\otimes\Psi(s_i)\otimes1+\gamma_i^*\otimes\Psi(s_i)^{-1}\otimes1)
-\sum_{i=1}^r(\beta_i\otimes1\otimes\Psi(t_i)+\beta_i^*\otimes1\otimes \Psi(t_i)^{-1})\Big)^{-1}$, where the coefficients are in $M_m(\mathbb C)\otimes\mathcal A$.
However, via a flip (specifically, via a conjugation with ${}^1\mathsf{flip}^0$), one views $(s_1,\dots,s_r)\mapsto
\Big((z-\xi)\otimes1\otimes1-\sum_{i=1}^{r}(\gamma_i\otimes\Psi(s_i)\otimes1+\gamma_i^*\otimes\Psi(s_i)^{-1}\otimes1)
-\sum_{i=1}^r(\beta_i\otimes1\otimes\Psi(t_i)+\beta_i^*\otimes1\otimes \Psi(t_i)^{-1})\Big)^{-1}$ as a rational function with coefficients in 
$M_m(\mathbb C)\otimes\mathcal A$ as well. In fact, these two examples and what is derived from them through the operations introduced in Section \ref{pol}
are the only ones of interest for our paper. Thus, let us apply $\#^3\left\{ \left[({\sf flip}\circ \partial_{j})\otimes ({\sf flip}\circ \partial_j)\right](\partial_{i}(D_{i}\ \cdot\ ))\right\}$
to our $f$, but via the identification with the difference-differential operator. Thus, $(D_if)(s_1,\dots,s_r)$ is expressed as the formula for $\Delta_if(s_1,\dots,s_r;s_1,\dots,s_r)(1)$
when the expression of $f$ is viewed as being in the opposite algebra $\mathcal A^{\rm op}$. This is a new noncommutative function in $r$ variables, which we denote by $g
(s_1,\dots,s_r)$. The object obtained through the application of $\partial_i$ to $g$ would be expected to be identified with $\Delta_ig(s_1,\dots,s_r;s_1,\dots,s_r)(\cdot)$.
However, since one may act on the two tensor factors independently with various operations (and Parraud's formula requires us to do so), the correct interpretation is for $\partial_ig$
to be identified with $((s_1,\dots,s_r);(\tilde{s}_1,\dots,\tilde{s}_r))\mapsto\Delta_ig(s_1,\dots,s_r;\tilde{s}_1,\dots,\tilde{s}_r)(\cdot)$, a noncommutative function
in $(s_1,\dots,s_r)$ and $(\tilde{s}_1,\dots,\tilde{s}_r)$ with values in a space of linear maps - different variables, although the evaluation can be performed in the same $r$-tuple 
as well. It is helpful to recall how this formula comes from the evaluations on upper triangular matrices as seen in Section \ref{sec:ncf}:
\begin{eqnarray*}
\lefteqn{g\left(\begin{bmatrix} {s}_1 & \delta_{i,1}\, \cdot \\ 0 & \tilde{s}_1\end{bmatrix},\dots,\begin{bmatrix} {s}_r & \delta_{i,r}\, \cdot \\ 0 & \tilde{s}_r\end{bmatrix}\right)=}\\
& & \begin{bmatrix} g({s}_1,\dots,s_r) & \Delta_{i}g(s_1,\dots,s_r;\tilde{s}_1,\dots,\tilde{s}_r)(\cdot) \\ 0 & g(\tilde{s}_1,\dots,\tilde{s}_1)\end{bmatrix}.
\end{eqnarray*}
As mentioned in Section \ref{sec:ncf}, this is a sample of a {\em higher order} noncommutative function (see \cite[Chapter 3]{KVV}). It is important that one may now view the
variables $\underline{s}_r$ and $\underline{\tilde{s}}_r$ as independent variables: Parraud's formula requires us now to apply $\Delta_j$ with respect to the variables $s_1,\dots,s_r$
and the same for $\tilde{s}_1,\dots,\tilde{s}_r$, but each of the two correspondences $(s_1,\dots,s_r)\mapsto\Delta_{i}g(s_1,\dots,s_r;\tilde{s}_1,\dots,\tilde{s}_r)(\cdot)$
and $(\tilde{s}_1,\dots,\tilde{s}_r)\mapsto\Delta_{i}g(s_1,\dots,s_r;\tilde{s}_1,\dots,\tilde{s}_r)(\cdot)$ viewed as for functions on noncommutative
subsets of $\mathcal A^{\rm op}$. Specifically, with the usual notation $\underline{s}_r=(s_1,\dots,s_r),\underline{\tilde{s}}_r=(\tilde{s}_1,\dots,\tilde{s}_r)$, 
the noncommutative function $h(\underline{s}_r;\underline{\tilde{s}}_r)=\Delta_{i}g(s_1,\dots,s_r;\tilde{s}_1,\dots,\tilde{s}_r)(\cdot)$ acts on $\coprod_{n,m}
M_{n\times m}(\mathcal A)$, from the left with the variables $\underline{s}_r$ and from the right with the variables $\underline{\tilde{s}}_r$.
Denote by $\Delta_j$ the partial difference-differential operator with respect to the variable $s_j$ and by $\tilde{\Delta}_j$ the one with respect to 
$\tilde{s}_j$ (one could a priori think of $\tilde{\Delta}_j$ as $\Delta_{j+r}$, but especially in this context it would be the wrong view, since it
would obscure the fact that the $s$ variables act on the left and the $\tilde{s}$ ones on the right). We view now $h$ as being defined on 
noncommutative subsets that live in $\mathcal A^{\rm op}$ - that is, in variables $\underline{s}_r^{\rm op}$ and $\underline{\tilde{s}}_r^{\rm op}$ - and apply
$\Delta_j\tilde{\Delta}_jh(\underline{s^1}_r^{\rm op},\underline{s^2}_r^{\rm op};\underline{\tilde{s^2}}_r^{\rm op},\underline{\tilde{s^1}}_r^{\rm op})=\tilde{\Delta}_j\Delta_jh(
\underline{s^1}^{\rm op}_r,\underline{s^2}_r^{\rm op};\underline{\tilde{s^2}}_r^{\rm op},\underline{\tilde{s^1}}_r^{\rm op})$ to obtain a 
third order (trilinear-valued) noncommutative map in four $r$-tuples of variables. Finally, we drop the op, that is, we let 
$k(\underline{s^1}_r,\underline{s^2}_r;\underline{\tilde{s^2}}_r,\underline{\tilde{s^1}}_r)=
\Delta_j\tilde{\Delta}_jh(\underline{s^1}_r,\underline{s^2}_r;\underline{\tilde{s^2}}_r,\underline{\tilde{s^1}}_r)$, still a third order noncommutative map, 
we evaluate it in the variables $( z^1_{t_1}, \tilde z^1_{t_1}, \tilde z^2_{t_1}, z_{t_1}^2)$ defined in the previous section,
and evaluates the tri-linear map thus obtained in $(1,1,1).$

In the following, for the comfort of the reader, we perform the above steps for the resolvent $(z-\mathcal S)^{-1}$, viewed as a noncommutative function of the variables 
$\underline{s}_r$ positioned on the second tensor coordinate. Recall that (after re-applying ${}^0\mathsf{flip}^1$), the application of $D_i$ to $(z-\mathcal S)^{-1}$ yields
{\footnotesize{\begin{eqnarray}
& & \frac{i}{2}\left(b\otimes1\otimes1-\sum_{i=1}^r(\gamma_i\otimes\Psi(s_i^{\rm op})\!\otimes\!1\!+\!\gamma_i^*\!\otimes\!\Psi(s_i^{\rm op})^{-1}\!\otimes\!1)\nonumber\right.\\
& & \mbox{}-\left.\sum_{i=1}^r(\beta_i\otimes1\otimes\Psi(t_i)+\beta_i^*\otimes1\otimes\Psi(t_i)^{-1})\right)^{-1}\nonumber\\
& & \mbox{}\times\left(\gamma_j\otimes\left[\left(\Psi(s_j)-1\right)\left(\Psi(s_j)-1\right)\right]^{\rm op}\!-\gamma_j^*\!\otimes\left[\left(\Psi(s_j)^{-1}-1\right)
\left(\Psi(s_j)^{-1}-1\right)\right]^{\rm op}\right)\!\otimes\!1\nonumber\\
&&\mbox{}\times\left(b\otimes1\otimes1-\sum_{i=1}^r(\gamma_i\otimes\Psi(s^{\rm op}_i)\otimes1+\gamma_i^*\otimes\Psi(s^{\rm op}_i)^{-1}\otimes1)\right.\nonumber\\
& & \mbox{}-\left.\sum_{i=1}^r(\beta_i\otimes1\otimes\Psi(t_i)+\beta_i^*\otimes1\otimes\Psi(t_i)^{-1})\right)^{-1}.\nonumber
\end{eqnarray}}}\noindent
The above formula has a remarkably favourable structure, which makes it coincide with 
{\footnotesize{\begin{eqnarray}
& & \frac{i}{2}\left(b\otimes1\otimes1-\sum_{i=1}^r(\gamma_i\otimes\Psi(s_i)\!\otimes\!1\!+\!\gamma_i^*\!\otimes\!\Psi(s_i)^{-1}\!\otimes\!1)\nonumber\right.\\
& & \mbox{}-\left.\sum_{i=1}^r(\beta_i\otimes1\otimes\Psi(t_i)+\beta_i^*\otimes1\otimes\Psi(t_i)^{-1})\right)^{-1}\nonumber\\
& & \mbox{}\times\left(\gamma_j\otimes\left(\Psi(s_j)-1\right)\left(\Psi(s_j)-1\right)-\gamma_j^*\!\otimes\left(\Psi(s_j)^{-1}-1\right)
\left(\Psi(s_j)^{-1}-1\right)\right)\!\otimes\!1\nonumber\\
&&\mbox{}\times\left(b\otimes1\otimes1-\sum_{i=1}^r(\gamma_i\otimes\Psi(s_i)\otimes1+\gamma_i^*\otimes\Psi(s_i)^{-1}\otimes1)\right.\nonumber\\
& & \mbox{}-\left.\sum_{i=1}^r(\beta_i\otimes1\otimes\Psi(t_i)+\beta_i^*\otimes1\otimes\Psi(t_i)^{-1})\right)^{-1},\nonumber
\end{eqnarray}}}\noindent
meaning that for this formula we obtain the same regardless of whether the second tensor coordinate contains elements from $\mathcal A$ or $\mathcal A^{\rm op}$.
(The reader uncomfortable with the opposite structure can reach this conclusion by noting that the op structure on selfadjoint elements
coincides with the usual structure followed by taking the adjoint.) The next step is to apply $\Delta_j$ again and evaluate it in two $r$-tuples to obtain 
a linear map-valued (i.e. first-order) noncommutative function. We have seen what effect the application of $\Delta_j$ on $(z-\mathcal S)^{-1}$ has in 
Equation \eqref{frezzy}, which we rewrite below with $\cdot$ (a dot) instead of $c$ and with $\tilde{s}_k,1\le k\le r$, instead of $u_k$ (convention $b=z-\xi$ is preserved):
\begin{eqnarray}
&\!\!\!\! & \frac{i}{2}\left(b\otimes1\otimes1-\sum_{i=1}^{r}(\gamma_i\otimes\Psi(s_i)\otimes1+\gamma_i^*\otimes\Psi(s_i)^{-1}\otimes1)\right.\nonumber\\
&\!\!\!\! & \mbox{}-\left.\sum_{i=1}^r(\beta_i\otimes1\otimes\Psi(t_i)+\beta_i^*\otimes1\otimes \Psi(t_i)^{-1})\right)^{-1}\nonumber\\
&\!\!\!\! & \mbox{}\times\left(\gamma_j\!\otimes\![(\Psi(s_j)\!-\!1)\, \cdot\, (\Psi(\tilde{s}_j)\!-\!1)]-\gamma_j^*\otimes[(\Psi(s_j)^{-1}\!-\!1)\,\cdot\,(\Psi(\tilde{s}_j)^{-1}\!-\!1)]
\right)\!\otimes\!1\nonumber\\
&\!\!\!\! &\mbox{}\times\left(b\otimes1\otimes1-\sum_{i=1}^{r}(\gamma_i\otimes\Psi(\tilde{s}_i))\otimes1+\gamma_i^*\otimes\Psi(\tilde{s}_i)^{-1})\otimes1)\right.\nonumber\\
&\!\!\!\! & \mbox{}-\left.\sum_{i=1}^r(\beta_i\otimes1\otimes\Psi(t_i)+\beta_i^*\otimes1\otimes \Psi(t_i)^{-1})\right)^{-1}.\nonumber
\end{eqnarray}
Together with the Leibniz rule, this yields the expression 
{\footnotesize{\begin{eqnarray}
& & \!\!\!\!\!\!\frac{-1}{4}\left(b\otimes1\otimes1-\sum_{i=1}^r(\gamma_i\otimes\Psi(s_i)\!\otimes\!1\!+\!\gamma_i^*\!\otimes\!\Psi(s_i)^{-1}\!\otimes\!1)\nonumber\right.\\
& & \mbox{}-\left.\sum_{i=1}^r(\beta_i\otimes1\otimes\Psi(t_i)+\beta_i^*\otimes1\otimes\Psi(t_i)^{-1})\right)^{-1}\nonumber\\
& & \mbox{}\times\left(\gamma_j\!\otimes\![(\Psi(s_j)\!-\!1)\, \cdot\, (\Psi(\tilde{s}_j)\!-\!1)]-\gamma_j^*\otimes[(\Psi(s_j)^{-1}\!-\!1)\,\cdot\,(\Psi(\tilde{s}_j)^{-1}\!-\!1)]\right)\!\otimes\!1\nonumber\\
& &\mbox{}\times\left(b\otimes1\otimes1-\sum_{i=1}^{r}(\gamma_i\otimes\Psi(\tilde{s}_i))\otimes1+\gamma_i^*\otimes\Psi(\tilde{s}_i)^{-1})\otimes1)\right.\nonumber\\
& & \mbox{}-\left.\sum_{i=1}^r(\beta_i\otimes1\otimes\Psi(t_i)+\beta_i^*\otimes1\otimes \Psi(t_i)^{-1})\right)^{-1}\nonumber\\
& & \mbox{}\times\left(\gamma_j\otimes\left(\Psi(\tilde{s}_j)-1\right)\left(\Psi(\tilde{s}_j)-1\right)-\gamma_j^*\!\otimes\left(\Psi(\tilde{s}_j)^{-1}-1\right)
\left(\Psi(\tilde{s}_j)^{-1}-1\right)\right)\!\otimes\!1\nonumber\\
&&\mbox{}\times\left(b\otimes1\otimes1-\sum_{i=1}^r(\gamma_i\otimes\Psi(\tilde{s}_i)\otimes1+\gamma_i^*\otimes\Psi(\tilde{s}_i)^{-1}\otimes1)\right.\nonumber\\
& & \mbox{}-\left.\sum_{i=1}^r(\beta_i\otimes1\otimes\Psi(t_i)+\beta_i^*\otimes1\otimes\Psi(t_i)^{-1})\right)^{-1}\nonumber\\
& & \!\!\!\!\!\!\mbox{}-\frac14\left(b\otimes1\otimes1-\sum_{i=1}^r(\gamma_i\otimes\Psi(s_i)\!\otimes\!1\!+\!\gamma_i^*\!\otimes\!\Psi(s_i)^{-1}\!\otimes\!1)\nonumber\right.\\
& & \mbox{}-\left.\sum_{i=1}^r(\beta_i\otimes1\otimes\Psi(t_i)+\beta_i^*\otimes1\otimes\Psi(t_i)^{-1})\right)^{-1}\nonumber\\
& & \mbox{}\times\left(\gamma_j\!\otimes\![(\Psi(s_j)\!-\!1)\, \cdot\, (\Psi(\tilde{s}_j)\!-\!1)^2+(\Psi(s_j)\!-\!1)^2\, \cdot\, (\Psi(\tilde{s}_j)\!-\!1)]\right.\nonumber\\
& & \mbox{}\left.+\gamma_j^*\otimes[(\Psi(s_j)^{-1}\!-\!1)^2\,\cdot\,(\Psi(\tilde{s}_j)^{-1}\!-\!1)+(\Psi(s_j)^{-1}\!-\!1)\,\cdot\,(\Psi(\tilde{s}_j)^{-1}\!-\!1)^2]\right)\!\otimes\!1\nonumber\\
&&\mbox{}\times\left(b\otimes1\otimes1-\sum_{i=1}^r(\gamma_i\otimes\Psi(\tilde{s}_i)\otimes1+\gamma_i^*\otimes\Psi(\tilde{s}_i)^{-1}\otimes1)\right.\nonumber\\
& & \mbox{}-\left.\sum_{i=1}^r(\beta_i\otimes1\otimes\Psi(t_i)+\beta_i^*\otimes1\otimes\Psi(t_i)^{-1})\right)^{-1}\nonumber\\
& & \!\!\!\!\!\!-\frac{1}{4}\left(b\otimes1\otimes1-\sum_{i=1}^r(\gamma_i\otimes\Psi(s_i)\!\otimes\!1\!+\!\gamma_i^*\!\otimes\!\Psi(s_i)^{-1}\!\otimes\!1)\nonumber\right.\\
& & \mbox{}-\left.\sum_{i=1}^r(\beta_i\otimes1\otimes\Psi(t_i)+\beta_i^*\otimes1\otimes\Psi(t_i)^{-1})\right)^{-1}\nonumber\\
& & \mbox{}\times\left(\gamma_j\!\otimes\!(\Psi(s_j)\!-\!1)(\Psi(s_j)\!-\!1)-\gamma_j^*\otimes(\Psi(s_j)^{-1}\!-\!1)(\Psi({s}_j)^{-1}\!-\!1)]\right)\!\otimes\!1\nonumber\\
& &\mbox{}\times\left(b\otimes1\otimes1-\sum_{i=1}^{r}(\gamma_i\otimes\Psi(s_i))\otimes1+\gamma_i^*\otimes\Psi({s}_i)^{-1})\otimes1)\right.\nonumber\\
& & \mbox{}-\left.\sum_{i=1}^r(\beta_i\otimes1\otimes\Psi(t_i)+\beta_i^*\otimes1\otimes \Psi(t_i)^{-1})\right)^{-1}\nonumber\\
& & \mbox{}\times\left(\gamma_j\otimes\left[\left(\Psi({s}_j)-1\right)\,\cdot\,\left(\Psi(\tilde{s}_j)-1\right)\right]-\gamma_j^*\!\otimes\left[\left(\Psi({s}_j)^{-1}-1\right)\,\cdot\,
\left(\Psi(\tilde{s}_j)^{-1}-1\right)\right]\right)\!\otimes\!1\nonumber\\
&&\mbox{}\times\left(b\otimes1\otimes1-\sum_{i=1}^r(\gamma_i\otimes\Psi(\tilde{s}_i)\otimes1+\gamma_i^*\otimes\Psi(\tilde{s}_i)^{-1}\otimes1)\right.\nonumber\\
& & \mbox{}-\left.\sum_{i=1}^r(\beta_i\otimes1\otimes\Psi(t_i)+\beta_i^*\otimes1\otimes\Psi(t_i)^{-1})\right)^{-1}\nonumber,
\end{eqnarray}}}\noindent
in which the reader is asked to mainly keep track of how the number of resolvents grows and where the dots (the placeholders for the argument of the linear map) are located.
Here that is still reasonably easy to do: there are three summands, in two of them there are three resolvents as factors, and in one there are two. The 
placeholder appears in the `middle' of the last-mentioned one, and alternatingly `left and right of the center one' in the other two. As an aside, one observes again
that, when evaluated in $\underline{s}_r=\underline{\tilde{s}}_r$, the above expression does not change regardless of whether 
$\underline{s}_r\in\mathcal A^r$ or $\underline{s}_r\in(\mathcal A^{\rm op})^r$.

We apply next $\Delta_k$ for $k$ possibly (but not necessarily) different from $j$, independently on the two sides of the placeholder $\cdot$ - that is, with respect
to $\underline{s}_r$ and $\underline{\tilde{s}}_r$ independently - but with $\underline{s}_r$ and $\underline{\tilde{s}}_r$ being viewed as belonging to $(\mathcal A^{\rm op})^r$,
meaning that the formulas one obtains left and right of the two newly occurring placeholders being understood to express elements in the opposite algebra.
To distinguish between the two sides, we use $\Delta_k$ for the partial difference-differential operator with respect to the $k^{\rm th}$ left-hand variable $s_k$
and $\tilde{\Delta}_k$ for the partial difference-differential operator with respect to the $k^{\rm th}$ right-hand variable $\tilde{s}_k$. Moreover, since 
it now becomes very difficult to read the formula with the required attention paid to the sides of the three placeholders, we mark the left side by writing it
in red, the middle in black, and the right in blue. The left and right placeholders will be marked by $\bullet$, while the middle remains $\cdot$.
We apply first $\tilde{\Delta}_k$, hence color blue only appears below. In addition, since there are three terms in the object to which we apply $\tilde{\Delta}_k$, we mark in 
the right-hand side where the result of the application of $\tilde{\Delta}_k$ ends by writing to its right ``end $x$th term,'' $x=1,2$, and 3:
{\footnotesize{\begin{eqnarray}
& & \!\!\!\!\!\!\tilde{\Delta}_k\left[\left(b\otimes1\otimes1-\sum_{i=1}^r(\gamma_i\otimes\Psi(s_i^{\rm op})\!\otimes\!1\!+\!\gamma_i^*\!\otimes\!\Psi(s_i^{\rm op})^{-1}\!\otimes\!1)\nonumber\right.\right.\\
& & \mbox{}-\left.\sum_{i=1}^r(\beta_i\otimes1\otimes\Psi(t_i)+\beta_i^*\otimes1\otimes\Psi(t_i)^{-1})\right)^{-1}\nonumber\\
& & \mbox{}\times\left(\gamma_j\!\otimes\![(\Psi(s_j^{\rm op})\!-\!1)\, \cdot\, (\Psi(\tilde{s}_j^{\rm op})\!-\!1)]-\gamma_j^*\otimes[(\Psi(s_j^{\rm op})^{-1}\!-\!1)\,\cdot\,(\Psi(\tilde{s}_j^{\rm op})^{-1}\!-\!1)]\right)\!\otimes\!1\nonumber\\
& &\mbox{}\times\left(b\otimes1\otimes1-\sum_{i=1}^{r}(\gamma_i\otimes\Psi(\tilde{s}^{\rm op}_i)\otimes1+\gamma_i^*\otimes\Psi(\tilde{s}^{\rm op}_i)^{-1}\otimes1)\right.\nonumber\\
& & \mbox{}-\left.\sum_{i=1}^r(\beta_i\otimes1\otimes\Psi(t_i)+\beta_i^*\otimes1\otimes \Psi(t_i)^{-1})\right)^{-1}\nonumber\\
& & \mbox{}\times\left(\gamma_j\otimes\left(\Psi(\tilde{s}_j^{\rm op})-1\right)^2-\gamma_j^*\!\otimes\left(\Psi(\tilde{s}^{\rm op}_j)^{-1}-1\right)^2
\right)\!\otimes\!1\nonumber\\
&&\mbox{}\times\left(b\otimes1\otimes1-\sum_{i=1}^r(\gamma_i\otimes\Psi(\tilde{s}^{\rm op}_i)\otimes1+\gamma_i^*\otimes\Psi(\tilde{s}^{\rm op}_i)^{-1}\otimes1)\right.\nonumber\\
& & \mbox{}-\left.\sum_{i=1}^r(\beta_i\otimes1\otimes\Psi(t_i)+\beta_i^*\otimes1\otimes\Psi(t_i)^{-1})\right)^{-1}\nonumber\\
& & \!\!\!\!\!\!\mbox{}+\left(b\otimes1\otimes1-\sum_{i=1}^r(\gamma_i\otimes\Psi(s_i^{\rm op})\otimes1+\gamma_i^*\otimes\Psi(s_i^{\rm op})^{-1}\otimes1)\nonumber\right.\\
& & \mbox{}-\left.\sum_{i=1}^r(\beta_i\otimes1\otimes\Psi(t_i)+\beta_i^*\otimes1\otimes\Psi(t_i)^{-1})\right)^{-1}\nonumber\\
& & \mbox{}\times\left(\gamma_j\!\otimes\![(\Psi(s^{\rm op}_j)\!-\!1)\, \cdot\, (\Psi(\tilde{s}^{\rm op}_j)\!-\!1)^2+(\Psi(s_j^{\rm op})\!-\!1)^2\, \cdot\, (\Psi(\tilde{s}^{\rm op}_j)\!-\!1)]\right.\nonumber\\
& & \mbox{}\left.+\gamma_j^*\otimes[(\Psi(s_j^{\rm op})^{-1}\!-\!1)^2\,\cdot\,(\Psi(\tilde{s}^{\rm op}_j)^{-1}\!-\!1)+(\Psi(s_j^{\rm op})^{-1}\!-\!1)\,\cdot\,(\Psi(\tilde{s}^{\rm op}_j)^{-1}\!-\!1)^2]\right)\!\otimes\!1\nonumber\\
&&\mbox{}\times\left(b\otimes1\otimes1-\sum_{i=1}^r(\gamma_i\otimes\Psi(\tilde{s}^{\rm op}_i)\otimes1+\gamma_i^*\otimes\Psi(\tilde{s}^{\rm op}_i)^{-1}\otimes1)\right.\nonumber\\
& & \mbox{}-\left.\sum_{i=1}^r(\beta_i\otimes1\otimes\Psi(t_i)+\beta_i^*\otimes1\otimes\Psi(t_i)^{-1})\right)^{-1}\nonumber\\
& & \!\!\!\!\!\!+\left(b\otimes1\otimes1-\sum_{i=1}^r(\gamma_i\otimes\Psi(s_i^{\rm op})\!\otimes\!1\!+\!\gamma_i^*\!\otimes\!\Psi(s_i^{\rm op})^{-1}\!\otimes\!1)\nonumber\right.\\
& & \mbox{}-\left.\sum_{i=1}^r(\beta_i\otimes1\otimes\Psi(t_i)+\beta_i^*\otimes1\otimes\Psi(t_i)^{-1})\right)^{-1}\nonumber\\
& & \mbox{}\times\left(\gamma_j\!\otimes\!(\Psi(s_j^{\rm op})\!-\!1)(\Psi(s_j^{\rm op})\!-\!1)-\gamma_j^*\otimes(\Psi(s_j^{\rm op})^{-1}\!-\!1)(\Psi(s^{\rm op}_j)^{-1}\!-\!1)]\right)\!\otimes\!1\nonumber\\
& &\mbox{}\times\left(b\otimes1\otimes1-\sum_{i=1}^{r}(\gamma_i\otimes\Psi(s^{\rm op}_i))\otimes1+\gamma_i^*\otimes\Psi(s^{\rm op}_i)^{-1})\otimes1)\right.\nonumber\\
& & \mbox{}-\left.\sum_{i=1}^r(\beta_i\otimes1\otimes\Psi(t_i)+\beta_i^*\otimes1\otimes \Psi(t_i)^{-1})\right)^{-1}\nonumber\\
& & \mbox{}\times\left(\gamma_j\otimes\left[\left(\Psi(s^{\rm op}_j)-1\right)\,\cdot\,\left(\Psi(\tilde{s}^{\rm op}_j)-1\right)\right]-\gamma_j^*\!\otimes\left[\left(\Psi(s^{\rm op}_j)^{-1}-1\right)\,\cdot\,\left(\Psi(\tilde{s}^{\rm op}_j)^{-1}-1\right)\right]\right)\!\otimes\!1\nonumber\\
&&\mbox{}\times\left(b\otimes1\otimes1-\sum_{i=1}^r(\gamma_i\otimes\Psi(\tilde{s}^{\rm op}_i)\otimes1+\gamma_i^*\otimes\Psi(\tilde{s}^{\rm op}_i)^{-1}\otimes1)\right.\nonumber\\
& & \mbox{}-\left.\left.\sum_{i=1}^r(\beta_i\otimes1\otimes\Psi(t_i)+\beta_i^*\otimes1\otimes\Psi(t_i)^{-1})\right)^{-1}\right]\nonumber\\
&=& \left(b\otimes1\otimes1-\sum_{i=1}^r(\gamma_i\otimes\Psi(s_i^{\rm op})\!\otimes\!1\!+\!\gamma_i^*\!\otimes\!\Psi(s_i^{\rm op})^{-1}\!\otimes\!1)\nonumber\right.\\
& & \mbox{}-\left.\sum_{i=1}^r(\beta_i\otimes1\otimes\Psi(t_i)+\beta_i^*\otimes1\otimes\Psi(t_i)^{-1})\right)^{-1}\nonumber\\
& & \mbox{}\times\delta_{j,k}\frac{i}{2}\left(\gamma_j\!\otimes\![(\Psi(s_j^{\rm op})\!-\!1)\, \cdot\, (\Psi(\tilde{s}_j^{\rm op})\!-\!1)\bullet{\color{blue}{(\Psi(\tilde{\mathsf s}_j^{\rm op})\!-\!1)}}]\right.\nonumber\\
& & \mbox{}+\left.\gamma_j^*\otimes[(\Psi(s_j^{\rm op})^{-1}\!-\!1)\,\cdot\,(\Psi(\tilde{s}_j^{\rm op})^{-1}\!-\!1)\bullet{\color{blue}{(\Psi(\tilde{\mathsf s}_j^{\rm op})^{-1}\!-\!1)}}]\right)\!\otimes\!1\nonumber\\
& &\mbox{}\times\left(b\otimes{\color{blue}{1}}\otimes1-\sum_{i=1}^{r}(\gamma_i\otimes{\color{blue}{\Psi(\tilde{\mathsf s}^{\rm op}_i)}}\otimes1+\gamma_i^*\otimes{\color{blue}{\Psi(\tilde{\mathsf s}^{\rm op}_i)^{-1}}}\otimes1)\right.\nonumber\\
& & \mbox{}-\left.\sum_{i=1}^r(\beta_i\otimes{\color{blue}{1}}\otimes\Psi(t_i)+\beta_i^*\otimes{\color{blue}{1}}\otimes \Psi(t_i)^{-1})\right)^{-1}\nonumber\\
& & \mbox{}\times\left(\gamma_j\otimes{\color{blue}{\left(\Psi(\tilde{\mathsf s}_j^{\rm op})-1\right)^2}}-\gamma_j^*\!\otimes{\color{blue}{\left(\Psi(\tilde{\mathsf s}^{\rm op}_j)^{-1}-1\right)^2}}
\right)\!\otimes\!1\nonumber\\
&&\mbox{}\times\left(b\otimes{\color{blue}{1}}\otimes1-\sum_{i=1}^r(\gamma_i\otimes{\color{blue}{\Psi(\tilde{\mathsf s}^{\rm op}_i)}}\otimes1+\gamma_i^*\otimes{\color{blue}{\Psi(\tilde{\mathsf s}^{\rm op}_i)^{-1}}}\otimes1)\right.\nonumber\\
& & \mbox{}-\left.\sum_{i=1}^r(\beta_i\otimes{\color{blue}{1}}\otimes\Psi(t_i)+\beta_i^*\otimes{\color{blue}{1}}\otimes\Psi(t_i)^{-1})\right)^{-1}\nonumber\\
& &\!\!\!\!\!\! \mbox{}+\left(b\otimes1\otimes1-\sum_{i=1}^r(\gamma_i\otimes\Psi(s_i^{\rm op})\!\otimes\!1\!+\!\gamma_i^*\!\otimes\!\Psi(s_i^{\rm op})^{-1}\!\otimes\!1)\nonumber\right.\\
& & \mbox{}-\left.\sum_{i=1}^r(\beta_i\otimes1\otimes\Psi(t_i)+\beta_i^*\otimes1\otimes\Psi(t_i)^{-1})\right)^{-1}\nonumber\\
& & \mbox{}\times\left(\gamma_j\!\otimes\![(\Psi(s_j^{\rm op})\!-\!1)\, \cdot\, (\Psi(\tilde{s}_j^{\rm op})\!-\!1)]-\gamma_j^*\otimes[(\Psi(s_j^{\rm op})^{-1}\!-\!1)\,\cdot\,(\Psi(\tilde{s}_j^{\rm op})^{-1}\!-\!1)]\right)\!\otimes\!1\nonumber\\
& &\mbox{}\times\left(b\otimes1\otimes1-\sum_{i=1}^{r}(\gamma_i\otimes\Psi(\tilde{s}^{\rm op}_i))\otimes1+\gamma_i^*\otimes\Psi(\tilde{s}^{\rm op}_i)^{-1})\otimes1)\right.\nonumber\\
& & \mbox{}-\left.\sum_{i=1}^r(\beta_i\otimes1\otimes\Psi(t_i)+\beta_i^*\otimes1\otimes \Psi(t_i)^{-1})\right)^{-1}\nonumber\\
& &\mbox{}\times\frac{i}{2}\left(\gamma_k\!\otimes\![(\Psi(\tilde{s}_k^{\rm op})\!-\!1)\bullet{\color{blue}{(\Psi(\tilde{\sf s}_k^{\rm op})\!-\!1)}}]-\gamma_k^*\otimes[(\Psi(\tilde{s}_k^{\rm op})^{-1}\!-\!1)\bullet{\color{blue}{(\Psi(\tilde{\sf s}_k^{\rm op})^{-1}\!-\!1)}}]\right)\!\otimes\!1\nonumber\\
& &\mbox{}\times\left(b\otimes{\color{blue}{1}}\otimes1-\sum_{i=1}^{r}(\gamma_i\otimes{\color{blue}{\Psi(\tilde{\sf s}^{\rm op}_i)}}\otimes1+\gamma_i^*\otimes{\color{blue}{\Psi(\tilde{\sf s}^{\rm op}_i)^{-1}}}\otimes1)\right.\nonumber\\
& & \mbox{}-\left.\sum_{i=1}^r(\beta_i\otimes{\color{blue}{1}}\otimes\Psi(t_i)+\beta_i^*\otimes{\color{blue}{1}}\otimes \Psi(t_i)^{-1})\right)^{-1}\nonumber\\
& & \mbox{}\times\left(\gamma_j\otimes{\color{blue}{\left(\Psi(\tilde{\sf s}_j^{\rm op})-1\right)^2}}-\gamma_j^*\!\otimes{\color{blue}{\left(\Psi(\tilde{\sf s}^{\rm op}_j)^{-1}-1\right)^2}}
\right)\!\otimes\!1\nonumber\\
&&\mbox{}\times\left(b\otimes{\color{blue}{1}}\otimes1-\sum_{i=1}^r(\gamma_i\otimes{\color{blue}{\Psi(\tilde{\sf s}^{\rm op}_i)}}\otimes1+\gamma_i^*\otimes{\color{blue}{\Psi(\tilde{\sf s}^{\rm op}_i)^{-1}}}\otimes1)\right.\nonumber\\
& & \mbox{}-\left.\sum_{i=1}^r(\beta_i\otimes{\color{blue}{1}}\otimes\Psi(t_i)+\beta_i^*\otimes{\color{blue}{1}}\otimes\Psi(t_i)^{-1})\right)^{-1}\nonumber\\
& &\!\!\!\!\!\! \mbox{}+\left(b\otimes1\otimes1-\sum_{i=1}^r(\gamma_i\otimes\Psi(s_i^{\rm op})\!\otimes\!1\!+\!\gamma_i^*\!\otimes\!\Psi(s_i^{\rm op})^{-1}\!\otimes\!1)\nonumber\right.\\
& & \mbox{}-\left.\sum_{i=1}^r(\beta_i\otimes1\otimes\Psi(t_i)+\beta_i^*\otimes1\otimes\Psi(t_i)^{-1})\right)^{-1}\nonumber\\
& & \mbox{}\times\left(\gamma_j\!\otimes\![(\Psi(s_j^{\rm op})\!-\!1)\, \cdot\, (\Psi(\tilde{s}_j^{\rm op})\!-\!1)]-\gamma_j^*\otimes[(\Psi(s_j^{\rm op})^{-1}\!-\!1)\,\cdot\,(\Psi(\tilde{s}_j^{\rm op})^{-1}\!-\!1)]\right)\!\otimes\!1\nonumber\\
& &\mbox{}\times\left(b\otimes1\otimes1-\sum_{i=1}^{r}(\gamma_i\otimes\Psi(\tilde{s}^{\rm op}_i)\otimes1+\gamma_i^*\otimes\Psi(\tilde{s}^{\rm op}_i)^{-1}\otimes1)\right.\nonumber\\
& & \mbox{}-\left.\sum_{i=1}^r(\beta_i\otimes1\otimes\Psi(t_i)+\beta_i^*\otimes1\otimes \Psi(t_i)^{-1})\right)^{-1}\nonumber\\
& & \mbox{}\times\frac{i}{2}\delta_{j,k}\left(\gamma_j\otimes\left\{\left(\Psi(\tilde{s}_j^{\rm op})-1\right)^2\bullet{\color{blue}{\left(\Psi(\tilde{\sf s}_j^{\rm op})-1\right)}}+
\left(\Psi(\tilde{s}_j^{\rm op})-1\right)\bullet{\color{blue}{\left(\Psi(\tilde{\sf s}_j^{\rm op})-1\right)^2}}\right\}\right.\nonumber\\
& & \mbox{}+\left.\gamma_j^*\!\otimes\left\{\left(\Psi(\tilde{s}^{\rm op}_j)^{-1}-1\right)\bullet{\color{blue}{\left(\Psi(\tilde{\sf s}^{\rm op}_j)^{-1}-1\right)^2}}+\left(\Psi(\tilde{s}^{\rm op}_j)^{-1}-1\right)^2\!\bullet{\color{blue}{\left(\Psi(\tilde{\sf s}^{\rm op}_j)^{-1}-1\right)}}\right\}
\right)\!\otimes\!1\nonumber\\
&&\mbox{}\times\left(b\otimes{\color{blue}{1}}\otimes1-\sum_{i=1}^r(\gamma_i\otimes{\color{blue}{\Psi(\tilde{\sf s}^{\rm op}_i)}}\otimes1+\gamma_i^*\otimes{\color{blue}{\Psi(\tilde{\sf s}^{\rm op}_i)^{-1}}}\otimes1)\right.\nonumber\\
& & \mbox{}-\left.\sum_{i=1}^r(\beta_i\otimes{\color{blue}{1}}\otimes\Psi(t_i)+\beta_i^*\otimes{\color{blue}{1}}\otimes\Psi(t_i)^{-1})\right)^{-1}\nonumber\\
& &\!\!\!\!\!\! \mbox{}+\left(b\otimes1\otimes1-\sum_{i=1}^r(\gamma_i\otimes\Psi(s_i^{\rm op})\!\otimes\!1\!+\!\gamma_i^*\!\otimes\!\Psi(s_i^{\rm op})^{-1}\!\otimes\!1)\nonumber\right.\\
& & \mbox{}-\left.\sum_{i=1}^r(\beta_i\otimes1\otimes\Psi(t_i)+\beta_i^*\otimes1\otimes\Psi(t_i)^{-1})\right)^{-1}\nonumber\\
& & \mbox{}\times\left(\gamma_j\!\otimes\![(\Psi(s_j^{\rm op})\!-\!1)\, \cdot\, (\Psi(\tilde{s}_j^{\rm op})\!-\!1)]-\gamma_j^*\otimes[(\Psi(s_j^{\rm op})^{-1}\!-\!1)\,\cdot\,(\Psi(\tilde{s}_j^{\rm op})^{-1}\!-\!1)]\right)\!\otimes\!1\nonumber\\
& &\mbox{}\times\left(b\otimes1\otimes1-\sum_{i=1}^{r}(\gamma_i\otimes\Psi(\tilde{s}^{\rm op}_i)\otimes1+\gamma_i^*\otimes\Psi(\tilde{s}^{\rm op}_i)^{-1}\otimes1)\right.\nonumber\\
& & \mbox{}-\left.\sum_{i=1}^r(\beta_i\otimes1\otimes\Psi(t_i)+\beta_i^*\otimes1\otimes \Psi(t_i)^{-1})\right)^{-1}\nonumber\\
& & \mbox{}\times\left(\gamma_j\otimes\left(\Psi(\tilde{s}_j^{\rm op})-1\right)^2-\gamma_j^*\!\otimes\left(\Psi(\tilde{s}^{\rm op}_j)^{-1}-1\right)^2
\right)\!\otimes\!1\nonumber\\
&&\mbox{}\times\left(b\otimes1\otimes1-\sum_{i=1}^r(\gamma_i\otimes\Psi(\tilde{s}^{\rm op}_i)\otimes1+\gamma_i^*\otimes\Psi(\tilde{s}^{\rm op}_i)^{-1}\otimes1)\right.\nonumber\\
& & \mbox{}-\left.\sum_{i=1}^r(\beta_i\otimes1\otimes\Psi(t_i)+\beta_i^*\otimes1\otimes\Psi(t_i)^{-1})\right)^{-1}\nonumber\\
& & \mbox{}\times\frac{i}{2}\left(\gamma_k\!\otimes\![(\Psi(\tilde{s}_k^{\rm op})\!-\!1)\bullet{\color{blue}{(\Psi(\tilde{\sf s}_k^{\rm op})\!-\!1)}}]-\gamma_k^*\otimes[(\Psi(\tilde{s}_k^{\rm op})^{-1}\!-\!1)\bullet{\color{blue}{(\Psi(\tilde{\sf s}_k^{\rm op})^{-1}\!-\!1)}}]\right)\!\otimes\!1\nonumber\\
&&\mbox{}\times\left(b\otimes{\color{blue}{1}}\otimes1-\sum_{i=1}^r(\gamma_i\otimes{\color{blue}{\Psi(\tilde{\sf s}^{\rm op}_i)}}\otimes1+\gamma_i^*\otimes{\color{blue}{\Psi(\tilde{\sf s}^{\rm op}_i)^{-1}}}\otimes1)\right.\nonumber\\
& & \mbox{}-\left.\sum_{i=1}^r(\beta_i\otimes{\color{blue}{1}}\otimes\Psi(t_i)+\beta_i^*\otimes{\color{blue}{1}}\otimes\Psi(t_i)^{-1})\right)^{-1}\quad\quad\text{end 1st term}\nonumber\\
& & \!\!\!\!\!\!\mbox{}+\left(b\otimes1\otimes1-\sum_{i=1}^r(\gamma_i\otimes\Psi(s_i^{\rm op})\otimes1+\gamma_i^*\otimes\Psi(s_i^{\rm op})^{-1}\otimes1)\nonumber\right.\\
& & \mbox{}-\left.\sum_{i=1}^r(\beta_i\otimes1\otimes\Psi(t_i)+\beta_i^*\otimes1\otimes\Psi(t_i)^{-1})\right)^{-1}\nonumber\\
& & \mbox{}\times\frac{i}{2}\delta_{j,k}\left(\gamma_j\!\otimes\!\left[(\Psi(s^{\rm op}_j)\!-\!1)\, \cdot\,\left\{ (\Psi(\tilde{s}^{\rm op}_j)\!-\!1)^2\bullet{\color{blue}{(\Psi(\tilde{\sf s}^{\rm op}_j)-1)}}+(\Psi(\tilde{s}^{\rm op}_j)\!-\!1)\bullet{\color{blue}{(\Psi(\tilde{\sf s}^{\rm op}_j)-1)^2}}\right\}\right.\right.\nonumber\\
& & +(\Psi(s_j^{\rm op})\!-\!1)^2\, \cdot\, (\Psi(\tilde{s}^{\rm op}_j)\!-\!1)\bullet{\color{blue}{(\Psi(\tilde{\sf s}^{\rm op}_j)\!-\!1)}}]\nonumber\\
& & \mbox{}-\gamma_j^*\otimes\left[(\Psi(s_j^{\rm op})^{-1}\!-\!1)^2\,\cdot\,(\Psi(\tilde{s}^{\rm op}_j)^{-1}\!-\!1)\bullet{\color{blue}{(\Psi(\tilde{\sf s}^{\rm op}_j)^{-1}\!-\!1)}}\right.\nonumber\\
& & \mbox{}\left.+\left.(\Psi(s_j^{\rm op})^{-1}\!-\!1)\cdot\left\{(\Psi(\tilde{s}^{\rm op}_j)^{-1}\!-\!1)^2\!\bullet{\color{blue}{(\Psi(\tilde{\sf s}^{\rm op}_j)^{-1}\!-\!1)}}\!+\!
(\Psi(\tilde{s}^{\rm op}_j)^{-1}\!-\!1)\!\bullet\!{\color{blue}{(\Psi(\tilde{\sf s}^{\rm op}_j)^{-1}\!-\!1)^2}}\right\}\right]\right)\!\otimes\!1\nonumber\\
&&\mbox{}\times\left(b\otimes{\color{blue}{1}}\otimes1-\sum_{i=1}^r(\gamma_i\otimes{\color{blue}{\Psi(\tilde{\sf s}^{\rm op}_i)}}\otimes1+\gamma_i^*\otimes{\color{blue}{\Psi(\tilde{\sf s}^{\rm op}_i)^{-1}}}\otimes1)\right.\nonumber\\
& & \mbox{}-\left.\sum_{i=1}^r(\beta_i\otimes{\color{blue}{1}}\otimes\Psi(t_i)+\beta_i^*\otimes{\color{blue}{1}}\otimes\Psi(t_i)^{-1})\right)^{-1}\nonumber\\
& & \!\!\!\!\!\!\mbox{}+\left(b\otimes1\otimes1-\sum_{i=1}^r(\gamma_i\otimes\Psi(s_i^{\rm op})\otimes1+\gamma_i^*\otimes\Psi(s_i^{\rm op})^{-1}\otimes1)\nonumber\right.\\
& & \mbox{}-\left.\sum_{i=1}^r(\beta_i\otimes1\otimes\Psi(t_i)+\beta_i^*\otimes1\otimes\Psi(t_i)^{-1})\right)^{-1}\nonumber\\
& & \mbox{}\times\left(\gamma_j\!\otimes\![(\Psi(s^{\rm op}_j)\!-\!1)\, \cdot\, (\Psi(\tilde{s}^{\rm op}_j)\!-\!1)^2+(\Psi(s_j^{\rm op})\!-\!1)^2\, \cdot\, (\Psi(\tilde{s}^{\rm op}_j)\!-\!1)]\right.\nonumber\\
& & \mbox{}\left.+\gamma_j^*\otimes[(\Psi(s_j^{\rm op})^{-1}\!-\!1)^2\,\cdot\,(\Psi(\tilde{s}^{\rm op}_j)^{-1}\!-\!1)+(\Psi(s_j^{\rm op})^{-1}\!-\!1)\,\cdot\,(\Psi(\tilde{s}^{\rm op}_j)^{-1}\!-\!1)^2]\right)\!\otimes\!1\nonumber\\
&&\mbox{}\times\left(b\otimes1\otimes1-\sum_{i=1}^r(\gamma_i\otimes\Psi(\tilde{s}^{\rm op}_i)\otimes1+\gamma_i^*\otimes\Psi(\tilde{s}^{\rm op}_i)^{-1}\otimes1)\right.\nonumber\\
& & \mbox{}-\left.\sum_{i=1}^r(\beta_i\otimes1\otimes\Psi(t_i)+\beta_i^*\otimes1\otimes\Psi(t_i)^{-1})\right)^{-1}\nonumber\\
& &\mbox{}\times\frac{i}{2}\left(\gamma_k\!\otimes\![(\Psi(\tilde{s}_k^{\rm op})\!-\!1)\bullet{\color{blue}{(\Psi(\tilde{\sf s}_k^{\rm op})\!-\!1)}}]-\gamma_k^*\otimes[(\Psi(\tilde{s}_k^{\rm op})^{-1}\!-\!1)\bullet{\color{blue}{(\Psi(\tilde{\sf s}_k^{\rm op})^{-1}\!-\!1)}}]\right)\!\otimes\!1\nonumber\\
&&\mbox{}\times\left(b\otimes{\color{blue}{1}}\otimes1-\sum_{i=1}^r(\gamma_i\otimes{\color{blue}{\Psi(\tilde{\sf s}^{\rm op}_i)}}\otimes1+\gamma_i^*\otimes{\color{blue}{\Psi(\tilde{\sf s}^{\rm op}_i)^{-1}}}\otimes1)\right.\nonumber\\
& & \mbox{}-\left.\sum_{i=1}^r(\beta_i\otimes{\color{blue}{1}}\otimes\Psi(t_i)+\beta_i^*\otimes{\color{blue}{1}}\otimes\Psi(t_i)^{-1})\right)^{-1}\quad\quad\text{end 2nd term}\nonumber\\
& & \!\!\!\!\!\!+\left(b\otimes1\otimes1-\sum_{i=1}^r(\gamma_i\otimes\Psi(s_i^{\rm op})\!\otimes\!1\!+\!\gamma_i^*\!\otimes\!\Psi(s_i^{\rm op})^{-1}\!\otimes\!1)\nonumber\right.\\
& & \mbox{}-\left.\sum_{i=1}^r(\beta_i\otimes1\otimes\Psi(t_i)+\beta_i^*\otimes1\otimes\Psi(t_i)^{-1})\right)^{-1}\nonumber\\
& & \mbox{}\times\left(\gamma_j\!\otimes\!(\Psi(s_j^{\rm op})\!-\!1)(\Psi(s_j^{\rm op})\!-\!1)-\gamma_j^*\otimes(\Psi(s_j^{\rm op})^{-1}\!-\!1)(\Psi(s^{\rm op}_j)^{-1}\!-\!1)]\right)\!\otimes\!1\nonumber\\
& &\mbox{}\times\left(b\otimes1\otimes1-\sum_{i=1}^{r}(\gamma_i\otimes\Psi(s^{\rm op}_i))\otimes1+\gamma_i^*\otimes\Psi(s^{\rm op}_i)^{-1})\otimes1)\right.\nonumber\\
& & \mbox{}-\left.\sum_{i=1}^r(\beta_i\otimes1\otimes\Psi(t_i)+\beta_i^*\otimes1\otimes \Psi(t_i)^{-1})\right)^{-1}\nonumber\\
& & \mbox{}\times\delta_{j,k}\frac{i}{2}\left(\gamma_j\otimes\left[\left(\Psi(s^{\rm op}_j)-1\right)\,\cdot\,\left(\Psi(\tilde{s}^{\rm op}_j)-1\right)\bullet{\color{blue}{\left(\Psi(\tilde{\sf s}^{\rm op}_j)-1\right)}}\right]\right.\nonumber\\
& & \mbox{}+\left.\gamma_j^*\!\otimes\left[\left(\Psi(s^{\rm op}_j)^{-1}-1\right)\,\cdot\,\left(\Psi(\tilde{s}^{\rm op}_j)^{-1}-1\right)\bullet{\color{blue}{\left(\Psi(\tilde{\sf s}^{\rm op}_j)^{-1}-1\right)}}\right]\right)\!\otimes\!1\nonumber\\
&&\mbox{}\times\left(b\otimes{\color{blue}{1}}\otimes1-\sum_{i=1}^r(\gamma_i\otimes{\color{blue}{\Psi(\tilde{\sf s}^{\rm op}_i)}}\otimes1+\gamma_i^*\otimes{\color{blue}{\Psi(\tilde{\sf s}^{\rm op}_i)^{-1}}}\otimes1)\right.\nonumber\\
& & \mbox{}-\left.\sum_{i=1}^r(\beta_i\otimes{\color{blue}{1}}\otimes\Psi(t_i)+\beta_i^*\otimes{\color{blue}{1}}\otimes\Psi(t_i)^{-1})\right)^{-1}\nonumber\\
& & \!\!\!\!\!\!+\left(b\otimes1\otimes1-\sum_{i=1}^r(\gamma_i\otimes\Psi(s_i^{\rm op})\!\otimes\!1\!+\!\gamma_i^*\!\otimes\!\Psi(s_i^{\rm op})^{-1}\!\otimes\!1)\nonumber\right.\\
& & \mbox{}-\left.\sum_{i=1}^r(\beta_i\otimes1\otimes\Psi(t_i)+\beta_i^*\otimes1\otimes\Psi(t_i)^{-1})\right)^{-1}\nonumber\\
& & \mbox{}\times\left(\gamma_j\!\otimes\!(\Psi(s_j^{\rm op})\!-\!1)(\Psi(s_j^{\rm op})\!-\!1)-\gamma_j^*\otimes(\Psi(s_j^{\rm op})^{-1}\!-\!1)(\Psi(s^{\rm op}_j)^{-1}\!-\!1)]\right)\!\otimes\!1\nonumber\\
& &\mbox{}\times\left(b\otimes1\otimes1-\sum_{i=1}^{r}(\gamma_i\otimes\Psi(s^{\rm op}_i))\otimes1+\gamma_i^*\otimes\Psi(s^{\rm op}_i)^{-1})\otimes1)\right.\nonumber\\
& & \mbox{}-\left.\sum_{i=1}^r(\beta_i\otimes1\otimes\Psi(t_i)+\beta_i^*\otimes1\otimes \Psi(t_i)^{-1})\right)^{-1}\nonumber\\
& & \mbox{}\times\left(\gamma_j\otimes\left[\left(\Psi(s^{\rm op}_j)-1\right)\,\cdot\,\left(\Psi(\tilde{s}^{\rm op}_j)-1\right)\right]-\gamma_j^*\!\otimes\left[\left(\Psi(s^{\rm op}_j)^{-1}-1\right)\,\cdot\,\left(\Psi(\tilde{s}^{\rm op}_j)^{-1}-1\right)\right]\right)\!\otimes\!1\nonumber\\
&&\mbox{}\times\left(b\otimes1\otimes1-\sum_{i=1}^r(\gamma_i\otimes\Psi(\tilde{s}^{\rm op}_i)\otimes1+\gamma_i^*\otimes\Psi(\tilde{s}^{\rm op}_i)^{-1}\otimes1)\right.\nonumber\\
& & \mbox{}-\left.\sum_{i=1}^r(\beta_i\otimes1\otimes\Psi(t_i)+\beta_i^*\otimes1\otimes\Psi(t_i)^{-1})\right)^{-1}\nonumber\\
& &\mbox{}\times\frac{i}{2}\left(\gamma_k\!\otimes\![(\Psi(\tilde{s}_k^{\rm op})\!-\!1)\bullet{\color{blue}{(\Psi(\tilde{\sf s}_k^{\rm op})\!-\!1)}}]-\gamma_k^*\otimes[(\Psi(\tilde{s}_k^{\rm op})^{-1}\!-\!1)\bullet{\color{blue}{(\Psi(\tilde{\sf s}_k^{\rm op})^{-1}\!-\!1)}}]\right)\!\otimes\!1\nonumber\\
&&\mbox{}\times\left(b\otimes{\color{blue}{1}}\otimes1-\sum_{i=1}^r(\gamma_i\otimes{\color{blue}{\Psi(\tilde{\sf s}^{\rm op}_i)}}\otimes1+\gamma_i^*\otimes{\color{blue}{\Psi(\tilde{\sf s}^{\rm op}_i)^{-1}}}\otimes1)\right.\nonumber\\
& & \mbox{}-\left.\sum_{i=1}^r(\beta_i\otimes{\color{blue}{1}}\otimes\Psi(t_i)+\beta_i^*\otimes{\color{blue}{1}}\otimes\Psi(t_i)^{-1})\right)^{-1}\quad\quad\text{end of 3rd term.}
\label{43}\end{eqnarray}}}\noindent
We remind the reader that the above is a second order noncommutative function, and the variables are denoted by $s,\tilde{s},$ and $\tilde{\sf s}$. Next, we ``differentiate in $s_k$'' 
- meaning, we apply the partial difference-differential operator $\Delta_k$ with respect to $s_k$ - in the right-hand side of \eqref{43}. Now a new placeholder $\bullet$ will appear,
and left of it the formulas on the second tensor coordinate will be red (the evaluation is in $(\underline{\sf s}_r,\underline{s}_r)$). We will obviously not copy all of the right-hand 
side of $\eqref{43}$ with square brackets around it and a $\Delta_k$ in front of it, but we just write below the final result. As in the case of $\tilde{\Delta}_k$, we mark where the 
result of differentiating each of the eight terms ends, this time in capital letters (``END $X$th TERM'', $X=1,\dots,8$), while at the same time preserving the markings from the 
application of $\tilde{\Delta}_k$ above:
{\footnotesize{\begin{eqnarray}
&& \!\!\!\!\!\!\left(b\otimes{\color{red}1}\otimes1-\sum_{i=1}^r(\gamma_i\otimes{\color{red}\Psi(\mathsf s_i^{\rm op})}\!\otimes\!1\!+\!\gamma_i^*\!\otimes\!{\color{red}\Psi(\mathsf s_i^{\rm op})^{-1}}\!\otimes\!1)\nonumber\right.\\
& & \mbox{}-\left.\sum_{i=1}^r(\beta_i\otimes{\color{red}1}\otimes\Psi(t_i)+\beta_i^*\otimes{\color{red}1}\otimes\Psi(t_i)^{-1})\right)^{-1}\nonumber\\
& & \mbox{}\times\frac{i}{2}(\gamma_k\otimes[{\color{red}(\Psi(\mathsf s_k^{\rm op})-1)}\bullet(\Psi( s_k^{\rm op})-1)]-\gamma_k^*\otimes[{\color{red}(\Psi(\mathsf s_k^{\rm op})^{-1}-1)}\bullet(\Psi(s_k^{\rm op})^{-1}-1)])\otimes1\nonumber\\
& & \mbox{}\times\left(b\otimes1\otimes1-\sum_{i=1}^r(\gamma_i\otimes\Psi(s_i^{\rm op})\!\otimes\!1\!+\!\gamma_i^*\!\otimes\!\Psi(s_i^{\rm op})^{-1}\!\otimes\!1)\nonumber\right.\\
& & \mbox{}-\left.\sum_{i=1}^r(\beta_i\otimes1\otimes\Psi(t_i)+\beta_i^*\otimes1\otimes\Psi(t_i)^{-1})\right)^{-1}\nonumber\\
& & \mbox{}\times\delta_{j,k}\frac{i}{2}\left(\gamma_j\!\otimes\![(\Psi(s_j^{\rm op})\!-\!1)\, \cdot\, (\Psi(\tilde{s}_j^{\rm op})\!-\!1)\bullet{\color{blue}{(\Psi(\tilde{\mathsf s}_j^{\rm op})\!-\!1)}}]\right.\nonumber\\
& & \mbox{}+\left.\gamma_j^*\otimes[(\Psi(s_j^{\rm op})^{-1}\!-\!1)\,\cdot\,(\Psi(\tilde{s}_j^{\rm op})^{-1}\!-\!1)\bullet{\color{blue}{(\Psi(\tilde{\mathsf s}_j^{\rm op})^{-1}\!-\!1)}}]\right)\!\otimes\!1\nonumber\\
& &\mbox{}\times\left(b\otimes{\color{blue}{1}}\otimes1-\sum_{i=1}^{r}(\gamma_i\otimes{\color{blue}{\Psi(\tilde{\mathsf s}^{\rm op}_i)}}\otimes1+\gamma_i^*\otimes{\color{blue}{\Psi(\tilde{\mathsf s}^{\rm op}_i)^{-1}}}\otimes1)\right.\nonumber\\
& & \mbox{}-\left.\sum_{i=1}^r(\beta_i\otimes{\color{blue}{1}}\otimes\Psi(t_i)+\beta_i^*\otimes{\color{blue}{1}}\otimes \Psi(t_i)^{-1})\right)^{-1}\nonumber\\
& & \mbox{}\times\left(\gamma_j\otimes{\color{blue}{\left(\Psi(\tilde{\mathsf s}_j^{\rm op})-1\right)^2}}-\gamma_j^*\!\otimes{\color{blue}{\left(\Psi(\tilde{\mathsf s}^{\rm op}_j)^{-1}-1\right)^2}}
\right)\!\otimes\!1\nonumber\\
&&\mbox{}\times\left(b\otimes{\color{blue}{1}}\otimes1-\sum_{i=1}^r(\gamma_i\otimes{\color{blue}{\Psi(\tilde{\mathsf s}^{\rm op}_i)}}\otimes1+\gamma_i^*\otimes{\color{blue}{\Psi(\tilde{\mathsf s}^{\rm op}_i)^{-1}}}\otimes1)\right.\nonumber\\
& & \mbox{}-\left.\sum_{i=1}^r(\beta_i\otimes{\color{blue}{1}}\otimes\Psi(t_i)+\beta_i^*\otimes{\color{blue}{1}}\otimes\Psi(t_i)^{-1})\right)^{-1}\nonumber%
\\
&&\!\!\!\!\!\! \mbox{}+ \left(b\otimes{\color{red}1}\otimes1-\sum_{i=1}^r(\gamma_i\otimes{\color{red}\Psi(\mathsf s_i^{\rm op})}\!\otimes\!1\!+\!\gamma_i^*\!\otimes\!{\color{red}\Psi(\mathsf s_i^{\rm op})^{-1}}\!\otimes\!1)\nonumber\right.\\
& & \mbox{}-\left.\sum_{i=1}^r(\beta_i\otimes{\color{red}1}\otimes\Psi(t_i)+\beta_i^*\otimes{\color{red}1}\otimes\Psi(t_i)^{-1})\right)^{-1}\nonumber\\
& & \mbox{}\times\frac{\delta_{j,k}}{4}\left(\gamma_j^*\!\otimes\![{\color{red}(\Psi(\mathsf s_j^{\rm op})^{-1}\!-\!1)}\bullet(\Psi(s_j^{\rm op})^{-1}\!-\!1)\, \cdot\, (\Psi(\tilde{s}_j^{\rm op})^{-1}\!-\!1)\bullet{\color{blue}{(\Psi(\tilde{\mathsf s}_j^{\rm op})^{-1}\!-\!1)}}]\right.\nonumber\\
& & \mbox{}-\left.\gamma_j\otimes[{\color{red}(\Psi(\mathsf s_j^{\rm op})\!-\!1)}\bullet(\Psi(s_j^{\rm op})\!-\!1)\,\cdot\,(\Psi(\tilde{s}_j^{\rm op})\!-\!1)\bullet{\color{blue}{(\Psi(\tilde{\mathsf s}_j^{\rm op})\!-\!1)}}]\right)\!\otimes\!1\nonumber\\
& &\mbox{}\times\left(b\otimes{\color{blue}{1}}\otimes1-\sum_{i=1}^{r}(\gamma_i\otimes{\color{blue}{\Psi(\tilde{\mathsf s}^{\rm op}_i)}}\otimes1+\gamma_i^*\otimes{\color{blue}{\Psi(\tilde{\mathsf s}^{\rm op}_i)^{-1}}}\otimes1)\right.\nonumber\\
& & \mbox{}-\left.\sum_{i=1}^r(\beta_i\otimes{\color{blue}{1}}\otimes\Psi(t_i)+\beta_i^*\otimes{\color{blue}{1}}\otimes \Psi(t_i)^{-1})\right)^{-1}\nonumber\\
& & \mbox{}\times\left(\gamma_j\otimes{\color{blue}{\left(\Psi(\tilde{\mathsf s}_j^{\rm op})-1\right)^2}}-\gamma_j^*\!\otimes{\color{blue}{\left(\Psi(\tilde{\mathsf s}^{\rm op}_j)^{-1}-1\right)^2}}
\right)\!\otimes\!1\nonumber\\
&&\mbox{}\times\left(b\otimes{\color{blue}{1}}\otimes1-\sum_{i=1}^r(\gamma_i\otimes{\color{blue}{\Psi(\tilde{\mathsf s}^{\rm op}_i)}}\otimes1+\gamma_i^*\otimes{\color{blue}{\Psi(\tilde{\mathsf s}^{\rm op}_i)^{-1}}}\otimes1)\right.\nonumber\\
& & \mbox{}-\left.\sum_{i=1}^r(\beta_i\otimes{\color{blue}{1}}\otimes\Psi(t_i)+\beta_i^*\otimes{\color{blue}{1}}\otimes\Psi(t_i)^{-1})\right)^{-1}\nonumber\text{END 1st TERM}\\
& &\!\!\!\!\!\! \mbox{}+\left(b\otimes{\color{red}1}\otimes1-\sum_{i=1}^r(\gamma_i\otimes{\color{red}\Psi(\mathsf s_i^{\rm op})}\!\otimes\!1\!+\!\gamma_i^*\!\otimes\!{\color{red}\Psi(\mathsf s_i^{\rm op})^{-1}}\!\otimes\!1)\nonumber\right.\\
& & \mbox{}-\left.\sum_{i=1}^r(\beta_i\otimes{\color{red}1}\otimes\Psi(t_i)+\beta_i^*\otimes{\color{red}1}\otimes\Psi(t_i)^{-1})\right)^{-1}\nonumber\\
& & \mbox{}\times\frac{i}{2}(\gamma_k\otimes[{\color{red}(\Psi(\mathsf s_k^{\rm op})-1)}\bullet(\Psi( s_k^{\rm op})-1)]-\gamma_k^*\otimes[{\color{red}(\Psi(\mathsf s_k^{\rm op})^{-1}-1)}\bullet(\Psi(s_k^{\rm op})^{-1}-1)])\otimes1\nonumber\\
& & \mbox{}\times\left(b\otimes1\otimes1-\sum_{i=1}^r(\gamma_i\otimes\Psi(s_i^{\rm op})\!\otimes\!1\!+\!\gamma_i^*\!\otimes\!\Psi(s_i^{\rm op})^{-1}\!\otimes\!1)\nonumber\right.\\
& & \mbox{}-\left.\sum_{i=1}^r(\beta_i\otimes1\otimes\Psi(t_i)+\beta_i^*\otimes1\otimes\Psi(t_i)^{-1})\right)^{-1}\nonumber\\
& & \mbox{}\times\left(\gamma_j\!\otimes\![(\Psi(s_j^{\rm op})\!-\!1)\, \cdot\, (\Psi(\tilde{s}_j^{\rm op})\!-\!1)]-\gamma_j^*\otimes[(\Psi(s_j^{\rm op})^{-1}\!-\!1)\,\cdot\,(\Psi(\tilde{s}_j^{\rm op})^{-1}\!-\!1)]\right)\!\otimes\!1\nonumber\\
& &\mbox{}\times\left(b\otimes1\otimes1-\sum_{i=1}^{r}(\gamma_i\otimes\Psi(\tilde{s}^{\rm op}_i))\otimes1+\gamma_i^*\otimes\Psi(\tilde{s}^{\rm op}_i)^{-1})\otimes1)\right.\nonumber\\
& & \mbox{}-\left.\sum_{i=1}^r(\beta_i\otimes1\otimes\Psi(t_i)+\beta_i^*\otimes1\otimes \Psi(t_i)^{-1})\right)^{-1}\nonumber\\
& &\mbox{}\times\frac{i}{2}\left(\gamma_k\!\otimes\![(\Psi(\tilde{s}_k^{\rm op})\!-\!1)\bullet{\color{blue}{(\Psi(\tilde{\sf s}_k^{\rm op})\!-\!1)}}]-\gamma_k^*\otimes[(\Psi(\tilde{s}_k^{\rm op})^{-1}\!-\!1)\bullet{\color{blue}{(\Psi(\tilde{\sf s}_k^{\rm op})^{-1}\!-\!1)}}]\right)\!\otimes\!1\nonumber\\
& &\mbox{}\times\left(b\otimes{\color{blue}{1}}\otimes1-\sum_{i=1}^{r}(\gamma_i\otimes{\color{blue}{\Psi(\tilde{\sf s}^{\rm op}_i)}}\otimes1+\gamma_i^*\otimes{\color{blue}{\Psi(\tilde{\sf s}^{\rm op}_i)^{-1}}}\otimes1)\right.\nonumber\\
& & \mbox{}-\left.\sum_{i=1}^r(\beta_i\otimes{\color{blue}{1}}\otimes\Psi(t_i)+\beta_i^*\otimes{\color{blue}{1}}\otimes \Psi(t_i)^{-1})\right)^{-1}\nonumber\\
& & \mbox{}\times\left(\gamma_j\otimes{\color{blue}{\left(\Psi(\tilde{\sf s}_j^{\rm op})-1\right)^2}}-\gamma_j^*\!\otimes{\color{blue}{\left(\Psi(\tilde{\sf s}^{\rm op}_j)^{-1}-1\right)^2}}
\right)\!\otimes\!1\nonumber\\
&&\mbox{}\times\left(b\otimes{\color{blue}{1}}\otimes1-\sum_{i=1}^r(\gamma_i\otimes{\color{blue}{\Psi(\tilde{\sf s}^{\rm op}_i)}}\otimes1+\gamma_i^*\otimes{\color{blue}{\Psi(\tilde{\sf s}^{\rm op}_i)^{-1}}}\otimes1)\right.\nonumber\\
& & \mbox{}-\left.\sum_{i=1}^r(\beta_i\otimes{\color{blue}{1}}\otimes\Psi(t_i)+\beta_i^*\otimes{\color{blue}{1}}\otimes\Psi(t_i)^{-1})\right)^{-1}\nonumber\\
& &\!\!\!\!\!\! \mbox{}+\left(b\otimes{\color{red}1}\otimes1-\sum_{i=1}^r(\gamma_i\otimes{\color{red}\Psi(\mathsf s_i^{\rm op})}\!\otimes\!1\!+\!\gamma_i^*\!\otimes\!{\color{red}\Psi(\mathsf s_i^{\rm op})^{-1}}\!\otimes\!1)\nonumber\right.\\
& & \mbox{}-\left.\sum_{i=1}^r(\beta_i\otimes{\color{red}1}\otimes\Psi(t_i)+\beta_i^*\otimes{\color{red}1}\otimes\Psi(t_i)^{-1})\right)^{-1}\nonumber\\
& & \mbox{}\times\frac{\delta_{j,k}}{4}\left(\gamma_j^*\!\otimes\![{\color{red}(\Psi(\mathsf s_j^{\rm op})^{-1}\!-\!1)}\bullet(\Psi(s_j^{\rm op})^{-1}\!-\!1)\, \cdot\, (\Psi(\tilde{s}_j^{\rm op})^{-1}\!-\!1)]\right.\nonumber\\
& & \left.\mbox{}-\gamma_j\otimes[{\color{red}(\Psi(\mathsf s_j^{\rm op})\!-\!1)}\bullet(\Psi(s_j^{\rm op})\!-\!1)\,\cdot\,(\Psi(\tilde{s}_j^{\rm op})\!-\!1)]\right)\!\otimes\!1\nonumber\\
& &\mbox{}\times\left(b\otimes1\otimes1-\sum_{i=1}^{r}(\gamma_i\otimes\Psi(\tilde{s}^{\rm op}_i))\otimes1+\gamma_i^*\otimes\Psi(\tilde{s}^{\rm op}_i)^{-1})\otimes1)\right.\nonumber\\
& & \mbox{}-\left.\sum_{i=1}^r(\beta_i\otimes1\otimes\Psi(t_i)+\beta_i^*\otimes1\otimes \Psi(t_i)^{-1})\right)^{-1}\nonumber\\
& &\mbox{}\times\left(\gamma_k\!\otimes\![(\Psi(\tilde{s}_k^{\rm op})\!-\!1)\bullet{\color{blue}{(\Psi(\tilde{\sf s}_k^{\rm op})\!-\!1)}}]-\gamma_k^*\otimes[(\Psi(\tilde{s}_k^{\rm op})^{-1}\!-\!1)\bullet{\color{blue}{(\Psi(\tilde{\sf s}_k^{\rm op})^{-1}\!-\!1)}}]\right)\!\otimes\!1\nonumber\\
& &\mbox{}\times\left(b\otimes{\color{blue}{1}}\otimes1-\sum_{i=1}^{r}(\gamma_i\otimes{\color{blue}{\Psi(\tilde{\sf s}^{\rm op}_i)}}\otimes1+\gamma_i^*\otimes{\color{blue}{\Psi(\tilde{\sf s}^{\rm op}_i)^{-1}}}\otimes1)\right.\nonumber\\
& & \mbox{}-\left.\sum_{i=1}^r(\beta_i\otimes{\color{blue}{1}}\otimes\Psi(t_i)+\beta_i^*\otimes{\color{blue}{1}}\otimes \Psi(t_i)^{-1})\right)^{-1}\nonumber\\
& & \mbox{}\times\left(\gamma_j\otimes{\color{blue}{\left(\Psi(\tilde{\sf s}_j^{\rm op})-1\right)^2}}-\gamma_j^*\!\otimes{\color{blue}{\left(\Psi(\tilde{\sf s}^{\rm op}_j)^{-1}-1\right)^2}}
\right)\!\otimes\!1\nonumber\\
&&\mbox{}\times\left(b\otimes{\color{blue}{1}}\otimes1-\sum_{i=1}^r(\gamma_i\otimes{\color{blue}{\Psi(\tilde{\sf s}^{\rm op}_i)}}\otimes1+\gamma_i^*\otimes{\color{blue}{\Psi(\tilde{\sf s}^{\rm op}_i)^{-1}}}\otimes1)\right.\nonumber\\
& & \mbox{}-\left.\sum_{i=1}^r(\beta_i\otimes{\color{blue}{1}}\otimes\Psi(t_i)+\beta_i^*\otimes{\color{blue}{1}}\otimes\Psi(t_i)^{-1})\right)^{-1}\text{END 2nd TERM}\nonumber\\
& &\!\!\!\!\!\! \mbox{}+\left(b\otimes{\color{red}1}\otimes1-\sum_{i=1}^r(\gamma_i\otimes{\color{red}\Psi(\mathsf s_i^{\rm op})}\!\otimes\!1\!+\!\gamma_i^*\!\otimes\!{\color{red}\Psi(\mathsf s_i^{\rm op})^{-1}}\!\otimes\!1)\nonumber\right.\\
& & \mbox{}-\left.\sum_{i=1}^r(\beta_i\otimes{\color{red}1}\otimes\Psi(t_i)+\beta_i^*\otimes{\color{red}1}\otimes\Psi(t_i)^{-1})\right)^{-1}\nonumber\\
& & \mbox{}\times\frac{i}{2}\left(\gamma_k\!\otimes\![{\color{red}(\Psi(\mathsf s_k^{\rm op})\!-\!1)}\bullet(\Psi(s_k^{\rm op})\!-\!1)]\!-\!\gamma_k^*\!\otimes\![{\color{red}(\Psi(\mathsf s_k^{\rm op})^{-1}\!-\!1)}\bullet(\Psi(s_k^{\rm op})^{-1}\!-\!1)]\right)\!\otimes\!1 \nonumber\\
& & \mbox{}\times\left(b\otimes1\otimes1-\sum_{i=1}^r(\gamma_i\otimes\Psi(s_i^{\rm op})\!\otimes\!1\!+\!\gamma_i^*\!\otimes\!\Psi(s_i^{\rm op})^{-1}\!\otimes\!1)\nonumber\right.\\
& & \mbox{}-\left.\sum_{i=1}^r(\beta_i\otimes1\otimes\Psi(t_i)+\beta_i^*\otimes1\otimes\Psi(t_i)^{-1})\right)^{-1}\nonumber\\
& & \mbox{}\times\left(\gamma_j\!\otimes\![(\Psi(s_j^{\rm op})\!-\!1)\, \cdot\, (\Psi(\tilde{s}_j^{\rm op})\!-\!1)]-\gamma_j^*\otimes[(\Psi(s_j^{\rm op})^{-1}\!-\!1)\,\cdot\,(\Psi(\tilde{s}_j^{\rm op})^{-1}\!-\!1)]\right)\!\otimes\!1\nonumber\\
& &\mbox{}\times\left(b\otimes1\otimes1-\sum_{i=1}^{r}(\gamma_i\otimes\Psi(\tilde{s}^{\rm op}_i)\otimes1+\gamma_i^*\otimes\Psi(\tilde{s}^{\rm op}_i)^{-1}\otimes1)\right.\nonumber\\
& & \mbox{}-\left.\sum_{i=1}^r(\beta_i\otimes1\otimes\Psi(t_i)+\beta_i^*\otimes1\otimes \Psi(t_i)^{-1})\right)^{-1}\nonumber\\
& & \mbox{}\times\frac{i}{2}\delta_{j,k}\left(\gamma_j\otimes\left\{\left(\Psi(\tilde{s}_j^{\rm op})-1\right)^2\bullet{\color{blue}{\left(\Psi(\tilde{\sf s}_j^{\rm op})-1\right)}}+
\left(\Psi(\tilde{s}_j^{\rm op})-1\right)\bullet{\color{blue}{\left(\Psi(\tilde{\sf s}_j^{\rm op})-1\right)^2}}\right\}\right.\nonumber\\
& & \mbox{}+\left.\gamma_j^*\!\otimes\left\{\left(\Psi(\tilde{s}^{\rm op}_j)^{-1}-1\right)\bullet{\color{blue}{\left(\Psi(\tilde{\sf s}^{\rm op}_j)^{-1}-1\right)^2}}+\left(\Psi(\tilde{s}^{\rm op}_j)^{-1}-1\right)^2\!\bullet{\color{blue}{\left(\Psi(\tilde{\sf s}^{\rm op}_j)^{-1}-1\right)}}\right\}
\right)\!\otimes\!1\nonumber\\
&&\mbox{}\times\left(b\otimes{\color{blue}{1}}\otimes1-\sum_{i=1}^r(\gamma_i\otimes{\color{blue}{\Psi(\tilde{\sf s}^{\rm op}_i)}}\otimes1+\gamma_i^*\otimes{\color{blue}{\Psi(\tilde{\sf s}^{\rm op}_i)^{-1}}}\otimes1)\right.\nonumber\\
& & \mbox{}-\left.\sum_{i=1}^r(\beta_i\otimes{\color{blue}{1}}\otimes\Psi(t_i)+\beta_i^*\otimes{\color{blue}{1}}\otimes\Psi(t_i)^{-1})\right)^{-1}\nonumber\\
& &\!\!\!\!\!\! \mbox{}+\left(b\otimes{\color{red}1}\otimes1-\sum_{i=1}^r(\gamma_i\otimes{\color{red}\Psi(\mathsf s_i^{\rm op})}\!\otimes\!1\!+\!\gamma_i^*\!\otimes\!{\color{red}\Psi(\mathsf s_i^{\rm op})^{-1}}\!\otimes\!1)\nonumber\right.\\
& & \mbox{}-\left.\sum_{i=1}^r(\beta_i\otimes{\color{red}1}\otimes\Psi(t_i)+\beta_i^*\otimes{\color{red}1}\otimes\Psi(t_i)^{-1})\right)^{-1}\nonumber\\
& & \mbox{}\times\frac{i}{2}\delta_{j,k}\left(\gamma_j\!\otimes\![{\color{red}(\Psi(\mathsf s_j^{\rm op})\!-\!1)}\bullet(\Psi(s_j^{\rm op})\!-\!1)\,\cdot\,
(\Psi(\tilde{s}_j^{\rm op})-1)]\right.\nonumber\\
& & \mbox{}+\left.\gamma_j^*\!\otimes\![{\color{red}(\Psi(\mathsf s_j^{\rm op})^{-1}\!-\!1)}\bullet(\Psi(s_j^{\rm op})^{-1}\!-\!1)\,\cdot\,
(\Psi(\tilde{s}_j^{\rm op})^{-1}-1)]\right)\otimes1\nonumber\\
& &\mbox{}\times\left(b\otimes1\otimes1-\sum_{i=1}^{r}(\gamma_i\otimes\Psi(\tilde{s}^{\rm op}_i)\otimes1+\gamma_i^*\otimes\Psi(\tilde{s}^{\rm op}_i)^{-1}\otimes1)\right.\nonumber\\
& & \mbox{}-\left.\sum_{i=1}^r(\beta_i\otimes1\otimes\Psi(t_i)+\beta_i^*\otimes1\otimes \Psi(t_i)^{-1})\right)^{-1}\nonumber\\
& & \mbox{}\times\frac{i}{2}\delta_{j,k}\left(\gamma_j\otimes\left\{\left(\Psi(\tilde{s}_j^{\rm op})-1\right)^2\bullet{\color{blue}{\left(\Psi(\tilde{\sf s}_j^{\rm op})-1\right)}}+
\left(\Psi(\tilde{s}_j^{\rm op})-1\right)\bullet{\color{blue}{\left(\Psi(\tilde{\sf s}_j^{\rm op})-1\right)^2}}\right\}\right.\nonumber\\
& & \mbox{}+\left.\gamma_j^*\!\otimes\left\{\left(\Psi(\tilde{s}^{\rm op}_j)^{-1}-1\right)\bullet{\color{blue}{\left(\Psi(\tilde{\sf s}^{\rm op}_j)^{-1}-1\right)^2}}+\left(\Psi(\tilde{s}^{\rm op}_j)^{-1}-1\right)^2\!\bullet{\color{blue}{\left(\Psi(\tilde{\sf s}^{\rm op}_j)^{-1}-1\right)}}\right\}
\right)\!\otimes\!1\nonumber\\
&&\mbox{}\times\left(b\otimes{\color{blue}{1}}\otimes1-\sum_{i=1}^r(\gamma_i\otimes{\color{blue}{\Psi(\tilde{\sf s}^{\rm op}_i)}}\otimes1+\gamma_i^*\otimes{\color{blue}{\Psi(\tilde{\sf s}^{\rm op}_i)^{-1}}}\otimes1)\right.\nonumber\\
& & \mbox{}-\left.\sum_{i=1}^r(\beta_i\otimes{\color{blue}{1}}\otimes\Psi(t_i)+\beta_i^*\otimes{\color{blue}{1}}\otimes\Psi(t_i)^{-1})\right)^{-1}\nonumber\text{END 3rd TERM}\\
& &\!\!\!\!\!\! \mbox{}+\left(b\otimes{\color{red}1}\otimes1-\sum_{i=1}^r(\gamma_i\otimes{\color{red}\Psi(\mathsf s_i^{\rm op})}\!\otimes\!1\!+\!\gamma_i^*\!\otimes\!{\color{red}\Psi(\mathsf s_i^{\rm op})^{-1}}\!\otimes\!1)\nonumber\right.\\
& & \mbox{}-\left.\sum_{i=1}^r(\beta_i\otimes{\color{red}1}\otimes\Psi(t_i)+\beta_i^*\otimes{\color{red}1}\otimes\Psi(t_i)^{-1})\right)^{-1}\nonumber\\
& & \mbox{}\times\frac{i}{2}\left(\gamma_k\!\otimes\![{\color{red}(\Psi(\mathsf s_k^{\rm op})\!-\!1)}\bullet(\Psi(s_k^{\rm op})\!-\!1)]\!-\!\gamma_k^*\!\otimes\![{\color{red}(\Psi(\mathsf s_k^{\rm op})^{-1}\!-\!1)}\bullet(\Psi(s_k^{\rm op})^{-1}\!-\!1)]\right)\!\otimes\!1 \nonumber\\
& & \mbox{}\times\left(b\otimes1\otimes1-\sum_{i=1}^r(\gamma_i\otimes\Psi(s_i^{\rm op})\!\otimes\!1\!+\!\gamma_i^*\!\otimes\!\Psi(s_i^{\rm op})^{-1}\!\otimes\!1)\nonumber\right.\\
& & \mbox{}-\left.\sum_{i=1}^r(\beta_i\otimes1\otimes\Psi(t_i)+\beta_i^*\otimes1\otimes\Psi(t_i)^{-1})\right)^{-1}\nonumber\\%
& & \mbox{}\times\left(\gamma_j\!\otimes\![(\Psi(s_j^{\rm op})\!-\!1)\, \cdot\, (\Psi(\tilde{s}_j^{\rm op})\!-\!1)]-\gamma_j^*\otimes[(\Psi(s_j^{\rm op})^{-1}\!-\!1)\,\cdot\,(\Psi(\tilde{s}_j^{\rm op})^{-1}\!-\!1)]\right)\!\otimes\!1\nonumber\\
& &\mbox{}\times\left(b\otimes1\otimes1-\sum_{i=1}^{r}(\gamma_i\otimes\Psi(\tilde{s}^{\rm op}_i)\otimes1+\gamma_i^*\otimes\Psi(\tilde{s}^{\rm op}_i)^{-1}\otimes1)\right.\nonumber\\
& & \mbox{}-\left.\sum_{i=1}^r(\beta_i\otimes1\otimes\Psi(t_i)+\beta_i^*\otimes1\otimes \Psi(t_i)^{-1})\right)^{-1}\nonumber\\
& & \mbox{}\times\left(\gamma_j\otimes\left(\Psi(\tilde{s}_j^{\rm op})-1\right)^2-\gamma_j^*\!\otimes\left(\Psi(\tilde{s}^{\rm op}_j)^{-1}-1\right)^2
\right)\!\otimes\!1\nonumber\\
&&\mbox{}\times\left(b\otimes1\otimes1-\sum_{i=1}^r(\gamma_i\otimes\Psi(\tilde{s}^{\rm op}_i)\otimes1+\gamma_i^*\otimes\Psi(\tilde{s}^{\rm op}_i)^{-1}\otimes1)\right.\nonumber\\
& & \mbox{}-\left.\sum_{i=1}^r(\beta_i\otimes1\otimes\Psi(t_i)+\beta_i^*\otimes1\otimes\Psi(t_i)^{-1})\right)^{-1}\nonumber\\
& & \mbox{}\times\frac{i}{2}\left(\gamma_k\!\otimes\![(\Psi(\tilde{s}_k^{\rm op})\!-\!1)\bullet{\color{blue}{(\Psi(\tilde{\sf s}_k^{\rm op})\!-\!1)}}]-\gamma_k^*\otimes[(\Psi(\tilde{s}_k^{\rm op})^{-1}\!-\!1)\bullet{\color{blue}{(\Psi(\tilde{\sf s}_k^{\rm op})^{-1}\!-\!1)}}]\right)\!\otimes\!1\nonumber\\
&&\mbox{}\times\left(b\otimes{\color{blue}{1}}\otimes1-\sum_{i=1}^r(\gamma_i\otimes{\color{blue}{\Psi(\tilde{\sf s}^{\rm op}_i)}}\otimes1+\gamma_i^*\otimes{\color{blue}{\Psi(\tilde{\sf s}^{\rm op}_i)^{-1}}}\otimes1)\right.\nonumber\\
& & \mbox{}-\left.\sum_{i=1}^r(\beta_i\otimes{\color{blue}{1}}\otimes\Psi(t_i)+\beta_i^*\otimes{\color{blue}{1}}\otimes\Psi(t_i)^{-1})\right)^{-1}\nonumber\\%
& &\!\!\!\!\!\! \mbox{}+\left(b\otimes{\color{red}1}\otimes1-\sum_{i=1}^r(\gamma_i\otimes{\color{red}\Psi(\mathsf s_i^{\rm op})}\!\otimes\!1\!+\!\gamma_i^*\!\otimes\!{\color{red}\Psi(\mathsf s_i^{\rm op})^{-1}}\!\otimes\!1)\nonumber\right.\\
& & \mbox{}-\left.\sum_{i=1}^r(\beta_i\otimes{\color{red}1}\otimes\Psi(t_i)+\beta_i^*\otimes{\color{red}1}\otimes\Psi(t_i)^{-1})\right)^{-1} \nonumber\\
& & \mbox{}\times\delta_{j,k}\frac{i}{2}\left(\gamma_j\otimes[{\color{red}(\Psi(\mathsf s_j^{\rm op})-1)}\bullet(\Psi(s_j^{\rm op})-1)\,\cdot\,
(\Psi(\tilde{ s}_j^{\rm op})-1)]\right.\nonumber\\
& & \mbox{}+\left.\gamma_j\otimes[{\color{red}(\Psi(\mathsf s_j^{\rm op})^{-1}-1)}\bullet(\Psi(s_j^{\rm op})^{-1}-1)\,\cdot\,
(\Psi(\tilde{ s}_j^{\rm op})^{-1}-1)]\right)\otimes1\nonumber\\
& &\mbox{}\times\left(b\otimes1\otimes1-\sum_{i=1}^{r}(\gamma_i\otimes\Psi(\tilde{s}^{\rm op}_i)\otimes1+\gamma_i^*\otimes\Psi(\tilde{s}^{\rm op}_i)^{-1}\otimes1)\right.\nonumber\\
& & \mbox{}-\left.\sum_{i=1}^r(\beta_i\otimes1\otimes\Psi(t_i)+\beta_i^*\otimes1\otimes \Psi(t_i)^{-1})\right)^{-1}\nonumber\\
& & \mbox{}\times\left(\gamma_j\otimes\left(\Psi(\tilde{s}_j^{\rm op})-1\right)^2-\gamma_j^*\!\otimes\left(\Psi(\tilde{s}^{\rm op}_j)^{-1}-1\right)^2
\right)\!\otimes\!1\nonumber\\
&&\mbox{}\times\left(b\otimes1\otimes1-\sum_{i=1}^r(\gamma_i\otimes\Psi(\tilde{s}^{\rm op}_i)\otimes1+\gamma_i^*\otimes\Psi(\tilde{s}^{\rm op}_i)^{-1}\otimes1)\right.\nonumber\\
& & \mbox{}-\left.\sum_{i=1}^r(\beta_i\otimes1\otimes\Psi(t_i)+\beta_i^*\otimes1\otimes\Psi(t_i)^{-1})\right)^{-1}\nonumber\\
& & \mbox{}\times\frac{i}{2}\left(\gamma_k\!\otimes\![(\Psi(\tilde{s}_k^{\rm op})\!-\!1)\bullet{\color{blue}{(\Psi(\tilde{\sf s}_k^{\rm op})\!-\!1)}}]-\gamma_k^*\otimes[(\Psi(\tilde{s}_k^{\rm op})^{-1}\!-\!1)\bullet{\color{blue}{(\Psi(\tilde{\sf s}_k^{\rm op})^{-1}\!-\!1)}}]\right)\!\otimes\!1\nonumber\\
&&\mbox{}\times\left(b\otimes{\color{blue}{1}}\otimes1-\sum_{i=1}^r(\gamma_i\otimes{\color{blue}{\Psi(\tilde{\sf s}^{\rm op}_i)}}\otimes1+\gamma_i^*\otimes{\color{blue}{\Psi(\tilde{\sf s}^{\rm op}_i)^{-1}}}\otimes1)\right.\nonumber\\
& & \mbox{}-\left.\sum_{i=1}^r(\beta_i\otimes{\color{blue}{1}}\otimes\Psi(t_i)+\beta_i^*\otimes{\color{blue}{1}}\otimes\Psi(t_i)^{-1})\right)^{-1}\nonumber\quad\quad\text{[end 1st term] }\nonumber\text{END 4th TERM}\\   %
& & \!\!\!\!\!\!\mbox{}+\left(b\otimes{\color{red}1}\otimes1-\sum_{i=1}^r(\gamma_i\otimes{\color{red}\Psi(\mathsf s_i^{\rm op})}\!\otimes\!1\!+\!\gamma_i^*\!\otimes\!{\color{red}\Psi(\mathsf s_i^{\rm op})^{-1}}\!\otimes\!1)\nonumber\right.\\
& & \mbox{}-\left.\sum_{i=1}^r(\beta_i\otimes{\color{red}1}\otimes\Psi(t_i)+\beta_i^*\otimes{\color{red}1}\otimes\Psi(t_i)^{-1})\right)^{-1}\nonumber\\
& & \mbox{}\times\frac{i}{2}\left(\gamma_k\!\otimes\![{\color{red}(\Psi(\mathsf s_k^{\rm op})\!-\!1)}\bullet(\Psi(s_k^{\rm op})\!-\!1)]\!-\!\gamma_k^*\!\otimes\![{\color{red}(\Psi(\mathsf s_k^{\rm op})^{-1}\!-\!1)}\bullet(\Psi(s_k^{\rm op})^{-1}\!-\!1)]\right)\!\otimes\!1 \nonumber\\
& & \mbox{}\times\left(b\otimes1\otimes1-\sum_{i=1}^r(\gamma_i\otimes\Psi(s_i^{\rm op})\!\otimes\!1\!+\!\gamma_i^*\!\otimes\!\Psi(s_i^{\rm op})^{-1}\!\otimes\!1)\nonumber\right.\\
& & \mbox{}-\left.\sum_{i=1}^r(\beta_i\otimes1\otimes\Psi(t_i)+\beta_i^*\otimes1\otimes\Psi(t_i)^{-1})\right)^{-1}\nonumber\\%
& & \mbox{}\times\frac{i}{2}\delta_{j,k}\left(\gamma_j\!\otimes\!\left[(\Psi(s^{\rm op}_j)\!-\!1)\, \cdot\,\left\{ (\Psi(\tilde{s}^{\rm op}_j)\!-\!1)^2\bullet{\color{blue}{(\Psi(\tilde{\sf s}^{\rm op}_j)-1)}}+(\Psi(\tilde{s}^{\rm op}_j)\!-\!1)\bullet{\color{blue}{(\Psi(\tilde{\sf s}^{\rm op}_j)-1)^2}}\right\}\right.\right.\nonumber\\
& & +(\Psi(s_j^{\rm op})\!-\!1)^2\, \cdot\, (\Psi(\tilde{s}^{\rm op}_j)\!-\!1)\bullet{\color{blue}{(\Psi(\tilde{\sf s}^{\rm op}_j)\!-\!1)}}]\nonumber\\
& & \mbox{}-\gamma_j^*\otimes\left[(\Psi(s_j^{\rm op})^{-1}\!-\!1)^2\,\cdot\,(\Psi(\tilde{s}^{\rm op}_j)^{-1}\!-\!1)\bullet{\color{blue}{(\Psi(\tilde{\sf s}^{\rm op}_j)^{-1}\!-\!1)}}\right.\nonumber\\
& & \mbox{}\left.+\left.(\Psi(s_j^{\rm op})^{-1}\!\!-\!1)\cdot\left\{(\Psi(\tilde{s}^{\rm op}_j)^{-1}\!\!-\!1)^2\!\bullet{\color{blue}{(\Psi(\tilde{\sf s}^{\rm op}_j)^{-1}\!\!-\!1)}}\!+\!
(\Psi(\tilde{s}^{\rm op}_j)^{-1}\!\!-\!1)\!\bullet\!{\color{blue}{(\Psi(\tilde{\sf s}^{\rm op}_j)^{-1}\!\!-\!1)^2}}\right\}\!\right]\!\right)\!\otimes\!1\nonumber\\
&&\mbox{}\times\left(b\otimes{\color{blue}{1}}\otimes1-\sum_{i=1}^r(\gamma_i\otimes{\color{blue}{\Psi(\tilde{\sf s}^{\rm op}_i)}}\otimes1+\gamma_i^*\otimes{\color{blue}{\Psi(\tilde{\sf s}^{\rm op}_i)^{-1}}}\otimes1)\right.\nonumber\\
& & \mbox{}-\left.\sum_{i=1}^r(\beta_i\otimes{\color{blue}{1}}\otimes\Psi(t_i)+\beta_i^*\otimes{\color{blue}{1}}\otimes\Psi(t_i)^{-1})\right)^{-1}\nonumber\\ %
& & \!\!\!\!\!\!\mbox{}+\left(b\otimes{\color{red}1}\otimes1-\sum_{i=1}^r(\gamma_i\otimes{\color{red}\Psi(\mathsf s_i^{\rm op})}\!\otimes\!1\!+\!\gamma_i^*\!\otimes\!{\color{red}\Psi(\mathsf s_i^{\rm op})^{-1}}\!\otimes\!1)\nonumber\right.\\
& & \mbox{}-\left.\sum_{i=1}^r(\beta_i\otimes{\color{red}1}\otimes\Psi(t_i)+\beta_i^*\otimes{\color{red}1}\otimes\Psi(t_i)^{-1})\right)^{-1}\nonumber\\
& & \mbox{}\times\frac{-1}{4}\delta_{j,k}\left(\gamma_j\!\otimes\!\left[{\color{red}(\Psi(\mathsf s_k^{\rm op})\!-\!1)}\bullet(\Psi(s^{\rm op}_j)\!-\!1)\, \cdot\,\left\{ (\Psi(\tilde{s}^{\rm op}_j)\!-\!1)^2\bullet{\color{blue}{(\Psi(\tilde{\sf s}^{\rm op}_j)-1)}}\right.\right.\right.\nonumber\\
& & \mbox{}+\left.(\Psi(\tilde{s}^{\rm op}_j)\!-\!1)\bullet{\color{blue}{(\Psi(\tilde{\sf s}^{\rm op}_j)-1)^2}}\right\}\nonumber\\ %
& & +\left\{{\color{red}(\Psi(\mathsf s_k^{\rm op})\!-\!1)}\bullet(\Psi(s_j^{\rm op})\!-\!1)^2+{\color{red}(\Psi(\mathsf s_k^{\rm op})\!-\!1)^2}\bullet(\Psi(s_j^{\rm op})\!-\!1)\right\}\, \cdot\, (\Psi(\tilde{s}^{\rm op}_j)\!-\!1)\bullet{\color{blue}{(\Psi(\tilde{\sf s}^{\rm op}_j)\!-\!1)}}]\nonumber\\ %
& & \mbox{}+\gamma_j^*\otimes\left[\{{\color{red}(\Psi(\mathsf s_k^{\rm op})^{-1}\!-\!1)}\bullet(\Psi(s_j^{\rm op})^{-1}\!-\!1)^2\right.\nonumber\\
& & \mbox{}+{\color{red}(\Psi(\mathsf s_k^{\rm op})^{-1}\!-\!1)^2}\bullet(\Psi(s_j^{\rm op})^{-1}\!-\!1)\}\,\cdot\,(\Psi(\tilde{s}^{\rm op}_j)^{-1}\!-\!1)\bullet{\color{blue}{(\Psi(\tilde{\sf s}^{\rm op}_j)^{-1}\!-\!1)}}\nonumber\\
& & \mbox{}+{\color{red}(\Psi(\mathsf s_k^{\rm op})^{-1}\!-\!1)}\bullet(\Psi(s_j^{\rm op})^{-1}\!-\!1)\cdot\left\{(\Psi(\tilde{s}^{\rm op}_j)^{-1}\!-\!1)^2\!\bullet{\color{blue}{(\Psi(\tilde{\sf s}^{\rm op}_j)^{-1}\!-\!1)}}\right.\nonumber\\
& & \left.\left.\left.\mbox{}+(\Psi(\tilde{s}^{\rm op}_j)^{-1}\!-\!1)\!\bullet\!{\color{blue}{(\Psi(\tilde{\sf s}^{\rm op}_j)^{-1}\!-\!1)^2}}\right\}\right]\right)\!\otimes\!1\nonumber\\
&&\mbox{}\times\left(b\otimes{\color{blue}{1}}\otimes1-\sum_{i=1}^r(\gamma_i\otimes{\color{blue}{\Psi(\tilde{\sf s}^{\rm op}_i)}}\otimes1+\gamma_i^*\otimes{\color{blue}{\Psi(\tilde{\sf s}^{\rm op}_i)^{-1}}}\otimes1)\right.\nonumber\\
& & \mbox{}-\left.\sum_{i=1}^r(\beta_i\otimes{\color{blue}{1}}\otimes\Psi(t_i)+\beta_i^*\otimes{\color{blue}{1}}\otimes\Psi(t_i)^{-1})\right)^{-1}\text{END 5th TERM}\nonumber\\ %
& & \!\!\!\!\!\!\mbox{}+\left(b\otimes{\color{red}1}\otimes1-\sum_{i=1}^r(\gamma_i\otimes{\color{red}\Psi(\mathsf s_i^{\rm op})}\!\otimes\!1\!+\!\gamma_i^*\!\otimes\!{\color{red}\Psi(\mathsf s_i^{\rm op})^{-1}}\!\otimes\!1)\nonumber\right.\\
& & \mbox{}-\left.\sum_{i=1}^r(\beta_i\otimes{\color{red}1}\otimes\Psi(t_i)+\beta_i^*\otimes{\color{red}1}\otimes\Psi(t_i)^{-1})\right)^{-1}\nonumber\\
& & \mbox{}\times\frac{i}{2}\left(\gamma_k\!\otimes\![{\color{red}(\Psi(\mathsf s_k^{\rm op})\!-\!1)}\bullet(\Psi(s_k^{\rm op})\!-\!1)]\!-\!\gamma_k^*\!\otimes\![{\color{red}(\Psi(\mathsf s_k^{\rm op})^{-1}\!-\!1)}\bullet(\Psi(s_k^{\rm op})^{-1}\!-\!1)]\right)\!\otimes\!1 \nonumber\\
& & \mbox{}\times\left(b\otimes1\otimes1-\sum_{i=1}^r(\gamma_i\otimes\Psi(s_i^{\rm op})\!\otimes\!1\!+\!\gamma_i^*\!\otimes\!\Psi(s_i^{\rm op})^{-1}\!\otimes\!1)\nonumber\right.\\
& & \mbox{}-\left.\sum_{i=1}^r(\beta_i\otimes1\otimes\Psi(t_i)+\beta_i^*\otimes1\otimes\Psi(t_i)^{-1})\right)^{-1}\nonumber\\%
& & \mbox{}\times\left(\gamma_j\!\otimes\![(\Psi(s^{\rm op}_j)\!-\!1)\, \cdot\, (\Psi(\tilde{s}^{\rm op}_j)\!-\!1)^2+(\Psi(s_j^{\rm op})\!-\!1)^2\, \cdot\, (\Psi(\tilde{s}^{\rm op}_j)\!-\!1)]\right.\nonumber\\
& & \mbox{}\left.+\gamma_j^*\otimes[(\Psi(s_j^{\rm op})^{-1}\!-\!1)^2\,\cdot\,(\Psi(\tilde{s}^{\rm op}_j)^{-1}\!-\!1)+(\Psi(s_j^{\rm op})^{-1}\!-\!1)\,\cdot\,(\Psi(\tilde{s}^{\rm op}_j)^{-1}\!-\!1)^2]\right)\!\otimes\!1\nonumber\\
&&\mbox{}\times\left(b\otimes1\otimes1-\sum_{i=1}^r(\gamma_i\otimes\Psi(\tilde{s}^{\rm op}_i)\otimes1+\gamma_i^*\otimes\Psi(\tilde{s}^{\rm op}_i)^{-1}\otimes1)\right.\nonumber\\
& & \mbox{}-\left.\sum_{i=1}^r(\beta_i\otimes1\otimes\Psi(t_i)+\beta_i^*\otimes1\otimes\Psi(t_i)^{-1})\right)^{-1}\nonumber\\
& &\mbox{}\times\frac{i}{2}\left(\gamma_k\!\otimes\![(\Psi(\tilde{s}_k^{\rm op})\!-\!1)\bullet{\color{blue}{(\Psi(\tilde{\sf s}_k^{\rm op})\!-\!1)}}]-\gamma_k^*\otimes[(\Psi(\tilde{s}_k^{\rm op})^{-1}\!-\!1)\bullet{\color{blue}{(\Psi(\tilde{\sf s}_k^{\rm op})^{-1}\!-\!1)}}]\right)\!\otimes\!1\nonumber\\
&&\mbox{}\times\left(b\otimes{\color{blue}{1}}\otimes1-\sum_{i=1}^r(\gamma_i\otimes{\color{blue}{\Psi(\tilde{\sf s}^{\rm op}_i)}}\otimes1+\gamma_i^*\otimes{\color{blue}{\Psi(\tilde{\sf s}^{\rm op}_i)^{-1}}}\otimes1)\right.\nonumber\\
& & \mbox{}-\left.\sum_{i=1}^r(\beta_i\otimes{\color{blue}{1}}\otimes\Psi(t_i)+\beta_i^*\otimes{\color{blue}{1}}\otimes\Psi(t_i)^{-1})\right)^{-1}\ \star\star\star\nonumber\\
& & \!\!\!\!\!\!\mbox{}+\left(b\otimes{\color{red}1}\otimes1-\sum_{i=1}^r(\gamma_i\otimes{\color{red}\Psi(\mathsf s_i^{\rm op})}\!\otimes\!1\!+\!\gamma_i^*\!\otimes\!{\color{red}\Psi(\mathsf s_i^{\rm op})^{-1}}\!\otimes\!1)\nonumber\right.\\
& & \mbox{}-\left.\sum_{i=1}^r(\beta_i\otimes{\color{red}1}\otimes\Psi(t_i)+\beta_i^*\otimes{\color{red}1}\otimes\Psi(t_i)^{-1})\right)^{-1}\nonumber\\
& & \mbox{}\times\frac{i}{2}\delta_{j,k}\left(\gamma_j\!\otimes\![{\color{red}(\Psi(\mathsf s_k^{\rm op})\!-\!1)}\bullet(\Psi(s^{\rm op}_j)\!-\!1)\, \cdot\, (\Psi(\tilde{s}^{\rm op}_j)\!-\!1)^2\right.\nonumber\\
& & \mbox{}+\{{\color{red}(\Psi(\mathsf s_k^{\rm op})\!-\!1)}\bullet(\Psi(s_j^{\rm op})\!-\!1)^2+{\color{red}(\Psi(\mathsf s_k^{\rm op})\!-\!1)^2}\bullet(\Psi(s_j^{\rm op})\!-\!1)\}\, \cdot\, (\Psi(\tilde{s}^{\rm op}_j)\!-\!1)]\nonumber\\%
& & \mbox{}-\gamma_j^*\otimes[\{{\color{red}(\Psi(\mathsf s_k^{\rm op})^{-1}\!-\!1)}\bullet(\Psi(s_j^{\rm op})^{-1}\!-\!1)^2\nonumber\\
& & \mbox{}+{\color{red}(\Psi(\mathsf s_k^{\rm op})^{-1}\!-\!1)^2}\bullet(\Psi(s_j^{\rm op})^{-1}\!-\!1)\}\,\cdot\,(\Psi(\tilde{s}^{\rm op}_j)^{-1}\!-\!1)\nonumber\\
& & \left.\mbox{}+{\color{red}(\Psi(\mathsf s_k^{\rm op})^{-1}\!-\!1)}\bullet(\Psi(s_j^{\rm op})^{-1}\!-\!1)\,\cdot\,(\Psi(\tilde{s}^{\rm op}_j)^{-1}\!-\!1)^2]\right)\!\otimes\!1\nonumber\\
&&\mbox{}\times\left(b\otimes1\otimes1-\sum_{i=1}^r(\gamma_i\otimes\Psi(\tilde{s}^{\rm op}_i)\otimes1+\gamma_i^*\otimes\Psi(\tilde{s}^{\rm op}_i)^{-1}\otimes1)\right.\nonumber\\
& & \mbox{}-\left.\sum_{i=1}^r(\beta_i\otimes1\otimes\Psi(t_i)+\beta_i^*\otimes1\otimes\Psi(t_i)^{-1})\right)^{-1}\nonumber\\
& &\mbox{}\times\frac{i}{2}\left(\gamma_k\!\otimes\![(\Psi(\tilde{s}_k^{\rm op})\!-\!1)\bullet{\color{blue}{(\Psi(\tilde{\sf s}_k^{\rm op})\!-\!1)}}]-\gamma_k^*\otimes[(\Psi(\tilde{s}_k^{\rm op})^{-1}\!-\!1)\bullet{\color{blue}{(\Psi(\tilde{\sf s}_k^{\rm op})^{-1}\!-\!1)}}]\right)\!\otimes\!1\nonumber\\
&&\mbox{}\times\left(b\otimes{\color{blue}{1}}\otimes1-\sum_{i=1}^r(\gamma_i\otimes{\color{blue}{\Psi(\tilde{\sf s}^{\rm op}_i)}}\otimes1+\gamma_i^*\otimes{\color{blue}{\Psi(\tilde{\sf s}^{\rm op}_i)^{-1}}}\otimes1)\right.\nonumber\\
& & \mbox{}-\left.\sum_{i=1}^r(\beta_i\otimes{\color{blue}{1}}\otimes\Psi(t_i)+\beta_i^*\otimes{\color{blue}{1}}\otimes\Psi(t_i)^{-1})\right)^{-1}\quad\quad\text{[end 2nd term] }\text{END 6th TERM}\nonumber\\ %
& & \!\!\!\!\!\!+\left(b\otimes{\color{red}1}\otimes1-\sum_{i=1}^r(\gamma_i\otimes{\color{red}\Psi(\mathsf s_i^{\rm op})}\!\otimes\!1\!+\!\gamma_i^*\!\otimes\!{\color{red}\Psi(\mathsf s_i^{\rm op})^{-1}}\!\otimes\!1)\nonumber\right.\\
& & \mbox{}-\left.\sum_{i=1}^r(\beta_i\otimes{\color{red}1}\otimes\Psi(t_i)+\beta_i^*\otimes{\color{red}1}\otimes\Psi(t_i)^{-1})\right)^{-1}\nonumber\\
& & \mbox{}\times\frac{i}{2}\left(\gamma_k\!\otimes\![{\color{red}(\Psi(\mathsf s_k^{\rm op})\!-\!1)}\bullet(\Psi(s_k^{\rm op})\!-\!1)]\!-\!\gamma_k^*\!\otimes\![{\color{red}(\Psi(\mathsf s_k^{\rm op})^{-1}\!-\!1)}\bullet(\Psi(s_k^{\rm op})^{-1}\!-\!1)]\right)\!\otimes\!1 \nonumber\\
& & \mbox{}\times\left(b\otimes1\otimes1-\sum_{i=1}^r(\gamma_i\otimes\Psi(s_i^{\rm op})\!\otimes\!1\!+\!\gamma_i^*\!\otimes\!\Psi(s_i^{\rm op})^{-1}\!\otimes\!1)\nonumber\right.\\
& & \mbox{}-\left.\sum_{i=1}^r(\beta_i\otimes1\otimes\Psi(t_i)+\beta_i^*\otimes1\otimes\Psi(t_i)^{-1})\right)^{-1}\nonumber\\%
& & \mbox{}\times\left(\gamma_j\!\otimes\!(\Psi(s_j^{\rm op})\!-\!1)(\Psi(s_j^{\rm op})\!-\!1)-\gamma_j^*\otimes(\Psi(s_j^{\rm op})^{-1}\!-\!1)(\Psi(s^{\rm op}_j)^{-1}\!-\!1)\right)\!\otimes\!1\nonumber\\
& &\mbox{}\times\left(b\otimes1\otimes1-\sum_{i=1}^{r}(\gamma_i\otimes\Psi(s^{\rm op}_i))\otimes1+\gamma_i^*\otimes\Psi(s^{\rm op}_i)^{-1})\otimes1)\right.\nonumber\\
& & \mbox{}-\left.\sum_{i=1}^r(\beta_i\otimes1\otimes\Psi(t_i)+\beta_i^*\otimes1\otimes \Psi(t_i)^{-1})\right)^{-1}\nonumber\\
& & \mbox{}\times\delta_{j,k}\frac{i}{2}\left(\gamma_j\otimes\left[\left(\Psi(s^{\rm op}_j)-1\right)\,\cdot\,\left(\Psi(\tilde{s}^{\rm op}_j)-1\right)\bullet{\color{blue}{\left(\Psi(\tilde{\sf s}^{\rm op}_j)-1\right)}}\right]\right.\nonumber\\
& & \mbox{}+\left.\gamma_j^*\!\otimes\left[\left(\Psi(s^{\rm op}_j)^{-1}-1\right)\,\cdot\,\left(\Psi(\tilde{s}^{\rm op}_j)^{-1}-1\right)\bullet{\color{blue}{\left(\Psi(\tilde{\sf s}^{\rm op}_j)^{-1}-1\right)}}\right]\right)\!\otimes\!1\nonumber\\
&&\mbox{}\times\left(b\otimes{\color{blue}{1}}\otimes1-\sum_{i=1}^r(\gamma_i\otimes{\color{blue}{\Psi(\tilde{\sf s}^{\rm op}_i)}}\otimes1+\gamma_i^*\otimes{\color{blue}{\Psi(\tilde{\sf s}^{\rm op}_i)^{-1}}}\otimes1)\right.\nonumber\\
& & \mbox{}-\left.\sum_{i=1}^r(\beta_i\otimes{\color{blue}{1}}\otimes\Psi(t_i)+\beta_i^*\otimes{\color{blue}{1}}\otimes\Psi(t_i)^{-1})\right)^{-1}\nonumber\\ %
& & \!\!\!\!\!\!+\left(b\otimes{\color{red}1}\otimes1-\sum_{i=1}^r(\gamma_i\otimes{\color{red}\Psi(\mathsf s_i^{\rm op})}\!\otimes\!1\!+\!\gamma_i^*\!\otimes\!{\color{red}\Psi(\mathsf s_i^{\rm op})^{-1}}\!\otimes\!1)\nonumber\right.\\
& & \mbox{}-\left.\sum_{i=1}^r(\beta_i\otimes{\color{red}1}\otimes\Psi(t_i)+\beta_i^*\otimes{\color{red}1}\otimes\Psi(t_i)^{-1})\right)^{-1}\nonumber\\
& & \mbox{}\times\delta_{j,k}\frac{i}{2}\left(\gamma_j\!\otimes\![{\color{red}(\Psi(\mathsf s_k^{\rm op})\!-\!1)}\bullet(\Psi(s_k^{\rm op})\!-\!1)^2+{\color{red}(\Psi(\mathsf s_k^{\rm op})\!-\!1)^2}\bullet(\Psi(s_k^{\rm op})\!-\!1)]\right.\nonumber\\%
& & \mbox{}+\left.\gamma_j^*\otimes[{\color{red}(\Psi(\mathsf s_k^{\rm op})^{-1}\!-\!1)^2}\bullet(\Psi(s^{\rm op}_j)^{-1}\!-\!1)+{\color{red}(\Psi(\mathsf s_k^{\rm op})^{-1}\!-\!1)}\bullet(\Psi(s^{\rm op}_j)^{-1}\!-\!1)^2]\right)\!\otimes\!1\nonumber\\
& &\mbox{}\times\left(b\otimes1\otimes1-\sum_{i=1}^{r}(\gamma_i\otimes\Psi(s^{\rm op}_i))\otimes1+\gamma_i^*\otimes\Psi(s^{\rm op}_i)^{-1})\otimes1)\right.\nonumber\\
& & \mbox{}-\left.\sum_{i=1}^r(\beta_i\otimes1\otimes\Psi(t_i)+\beta_i^*\otimes1\otimes \Psi(t_i)^{-1})\right)^{-1}\nonumber\\
& & \mbox{}\times\delta_{j,k}\frac{i}{2}\left(\gamma_j\otimes\left[\left(\Psi(s^{\rm op}_j)-1\right)\,\cdot\,\left(\Psi(\tilde{s}^{\rm op}_j)-1\right)\bullet{\color{blue}{\left(\Psi(\tilde{\sf s}^{\rm op}_j)-1\right)}}\right]\right.\nonumber\\
& & \mbox{}+\left.\gamma_j^*\!\otimes\left[\left(\Psi(s^{\rm op}_j)^{-1}-1\right)\,\cdot\,\left(\Psi(\tilde{s}^{\rm op}_j)^{-1}-1\right)\bullet{\color{blue}{\left(\Psi(\tilde{\sf s}^{\rm op}_j)^{-1}-1\right)}}\right]\right)\!\otimes\!1\nonumber\\
&&\mbox{}\times\left(b\otimes{\color{blue}{1}}\otimes1-\sum_{i=1}^r(\gamma_i\otimes{\color{blue}{\Psi(\tilde{\sf s}^{\rm op}_i)}}\otimes1+\gamma_i^*\otimes{\color{blue}{\Psi(\tilde{\sf s}^{\rm op}_i)^{-1}}}\otimes1)\right.\nonumber\\
& & \mbox{}-\left.\sum_{i=1}^r(\beta_i\otimes{\color{blue}{1}}\otimes\Psi(t_i)+\beta_i^*\otimes{\color{blue}{1}}\otimes\Psi(t_i)^{-1})\right)^{-1}\nonumber\\%
& & \!\!\!\!\!\!+\left(b\otimes{\color{red}1}\otimes1-\sum_{i=1}^r(\gamma_i\otimes{\color{red}\Psi(\mathsf s_i^{\rm op})}\!\otimes\!1\!+\!\gamma_i^*\!\otimes\!{\color{red}\Psi(\mathsf s_i^{\rm op})^{-1}}\!\otimes\!1)\nonumber\right.\\
& & \mbox{}-\left.\sum_{i=1}^r(\beta_i\otimes{\color{red}1}\otimes\Psi(t_i)+\beta_i^*\otimes{\color{red}1}\otimes\Psi(t_i)^{-1})\right)^{-1}\nonumber\\%
& & \mbox{}\times\left(\gamma_j\!\otimes\!{\color{red}(\Psi(\mathsf s_j^{\rm op})\!-\!1)^2}-\gamma_j^*\otimes{\color{red}(\Psi(\mathsf s_j^{\rm op})^{-1}\!-\!1)^2}\right)\!\otimes\!1\nonumber\\%
& & \mbox{}\times\left(b\otimes{\color{red}1}\otimes1-\sum_{i=1}^r(\gamma_i\otimes{\color{red}\Psi(\mathsf s_i^{\rm op})}\!\otimes\!1\!+\!\gamma_i^*\!\otimes\!{\color{red}\Psi(\mathsf s_i^{\rm op})^{-1}}\!\otimes\!1)\nonumber\right.\\
& & \mbox{}-\left.\sum_{i=1}^r(\beta_i\otimes{\color{red}1}\otimes\Psi(t_i)+\beta_i^*\otimes{\color{red}1}\otimes\Psi(t_i)^{-1})\right)^{-1}\nonumber\\
& & \mbox{}\times\frac{i}{2}\left(\gamma_k\!\otimes\![{\color{red}(\Psi(\mathsf s_k^{\rm op})\!-\!1)}\bullet(\Psi(s_k^{\rm op})\!-\!1)]\!-\!\gamma_k^*\!\otimes\![{\color{red}(\Psi(\mathsf s_k^{\rm op})^{-1}\!-\!1)}\bullet(\Psi(s_k^{\rm op})^{-1}\!-\!1)]\right)\!\otimes\!1 \nonumber\\
& &\mbox{}\times\left(b\otimes1\otimes1-\sum_{i=1}^{r}(\gamma_i\otimes\Psi(s^{\rm op}_i))\otimes1+\gamma_i^*\otimes\Psi(s^{\rm op}_i)^{-1})\otimes1)\right.\nonumber\\
& & \mbox{}-\left.\sum_{i=1}^r(\beta_i\otimes1\otimes\Psi(t_i)+\beta_i^*\otimes1\otimes \Psi(t_i)^{-1})\right)^{-1}\nonumber\\
& & \mbox{}\times\delta_{j,k}\frac{i}{2}\left(\gamma_j\otimes\left[\left(\Psi(s^{\rm op}_j)-1\right)\,\cdot\,\left(\Psi(\tilde{s}^{\rm op}_j)-1\right)\bullet{\color{blue}{\left(\Psi(\tilde{\sf s}^{\rm op}_j)-1\right)}}\right]\right.\nonumber\\
& & \mbox{}+\left.\gamma_j^*\!\otimes\left[\left(\Psi(s^{\rm op}_j)^{-1}-1\right)\,\cdot\,\left(\Psi(\tilde{s}^{\rm op}_j)^{-1}-1\right)\bullet{\color{blue}{\left(\Psi(\tilde{\sf s}^{\rm op}_j)^{-1}-1\right)}}\right]\right)\!\otimes\!1\nonumber\\
&&\mbox{}\times\left(b\otimes{\color{blue}{1}}\otimes1-\sum_{i=1}^r(\gamma_i\otimes{\color{blue}{\Psi(\tilde{\sf s}^{\rm op}_i)}}\otimes1+\gamma_i^*\otimes{\color{blue}{\Psi(\tilde{\sf s}^{\rm op}_i)^{-1}}}\otimes1)\right.\nonumber\\
& & \mbox{}-\left.\sum_{i=1}^r(\beta_i\otimes{\color{blue}{1}}\otimes\Psi(t_i)+\beta_i^*\otimes{\color{blue}{1}}\otimes\Psi(t_i)^{-1})\right)^{-1}\nonumber\\%
& & \!\!\!\!\!\!+\left(b\otimes{\color{red}1}\otimes1-\sum_{i=1}^r(\gamma_i\otimes{\color{red}\Psi(\mathsf s_i^{\rm op})}\!\otimes\!1\!+\!\gamma_i^*\!\otimes\!{\color{red}\Psi(\mathsf s_i^{\rm op})^{-1}}\!\otimes\!1)\nonumber\right.\\
& & \mbox{}-\left.\sum_{i=1}^r(\beta_i\otimes{\color{red}1}\otimes\Psi(t_i)+\beta_i^*\otimes{\color{red}1}\otimes\Psi(t_i)^{-1})\right)^{-1}\nonumber\\%
& & \mbox{}\times\left(\gamma_j\!\otimes\!{\color{red}(\Psi(\mathsf s_j^{\rm op})\!-\!1)^2}-\gamma_j^*\otimes{\color{red}(\Psi(\mathsf s_j^{\rm op})^{-1}\!-\!1)^2}\right)\!\otimes\!1\nonumber\\%
& & \mbox{}\times\left(b\otimes{\color{red}1}\otimes1-\sum_{i=1}^r(\gamma_i\otimes{\color{red}\Psi(\mathsf s_i^{\rm op})}\!\otimes\!1\!+\!\gamma_i^*\!\otimes\!{\color{red}\Psi(\mathsf s_i^{\rm op})^{-1}}\!\otimes\!1)\nonumber\right.\\
& & \mbox{}-\left.\sum_{i=1}^r(\beta_i\otimes{\color{red}1}\otimes\Psi(t_i)+\beta_i^*\otimes{\color{red}1}\otimes\Psi(t_i)^{-1})\right)^{-1}\nonumber\\ %
& & \mbox{}\times\delta_{j,k}\frac{-1}{4}\left(\gamma_j\otimes\left[{\color{red}\left(\Psi(s^{\rm op}_j)-1\right)}\bullet\left(\Psi(s^{\rm op}_j)-1\right)\,\cdot\,\left(\Psi(\tilde{s}^{\rm op}_j)-1\right)\bullet{\color{blue}{\left(\Psi(\tilde{\sf s}^{\rm op}_j)-1\right)}}\right]\right.\nonumber\\
& & \mbox{}-\left.\gamma_j^*\!\otimes\left[{\color{red}\left(\Psi(s^{\rm op}_j)^{-1}-1\right)}\bullet\left(\Psi(s^{\rm op}_j)^{-1}-1\right)\,\cdot\,\left(\Psi(\tilde{s}^{\rm op}_j)^{-1}-1\right)\bullet{\color{blue}{\left(\Psi(\tilde{\sf s}^{\rm op}_j)^{-1}-1\right)}}\right]\right)\!\otimes\!1\nonumber\\
&&\mbox{}\times\left(b\otimes{\color{blue}{1}}\otimes1-\sum_{i=1}^r(\gamma_i\otimes{\color{blue}{\Psi(\tilde{\sf s}^{\rm op}_i)}}\otimes1+\gamma_i^*\otimes{\color{blue}{\Psi(\tilde{\sf s}^{\rm op}_i)^{-1}}}\otimes1)\right.\nonumber\\
& & \mbox{}-\left.\sum_{i=1}^r(\beta_i\otimes{\color{blue}{1}}\otimes\Psi(t_i)+\beta_i^*\otimes{\color{blue}{1}}\otimes\Psi(t_i)^{-1})\right)^{-1}\nonumber\text{END 7th TERM}\\%
& & \!\!\!\!\!\!+\left(b\otimes{\color{red}1}\otimes1-\sum_{i=1}^r(\gamma_i\otimes{\color{red}\Psi(\mathsf s_i^{\rm op})}\!\otimes\!1\!+\!\gamma_i^*\!\otimes\!{\color{red}\Psi(\mathsf s_i^{\rm op})^{-1}}\!\otimes\!1)\nonumber\right.\\
& & \mbox{}-\left.\sum_{i=1}^r(\beta_i\otimes{\color{red}1}\otimes\Psi(t_i)+\beta_i^*\otimes{\color{red}1}\otimes\Psi(t_i)^{-1})\right)^{-1}\nonumber\\
& & \mbox{}\times\frac{i}{2}\left(\gamma_k\!\otimes\![{\color{red}(\Psi(\mathsf s_k^{\rm op})\!-\!1)}\bullet(\Psi(s_k^{\rm op})\!-\!1)]\!-\!\gamma_k^*\!\otimes\![{\color{red}(\Psi(\mathsf s_k^{\rm op})^{-1}\!-\!1)}\bullet(\Psi(s_k^{\rm op})^{-1}\!-\!1)]\right)\!\otimes\!1 \nonumber\\
& & \mbox{}\times\left(b\otimes1\otimes1-\sum_{i=1}^r(\gamma_i\otimes\Psi(s_i^{\rm op})\!\otimes\!1\!+\!\gamma_i^*\!\otimes\!\Psi(s_i^{\rm op})^{-1}\!\otimes\!1)\nonumber\right.\\
& & \mbox{}-\left.\sum_{i=1}^r(\beta_i\otimes1\otimes\Psi(t_i)+\beta_i^*\otimes1\otimes\Psi(t_i)^{-1})\right)^{-1}\nonumber\\
& & \mbox{}\times\left(\gamma_j\!\otimes\!(\Psi(s_j^{\rm op})\!-\!1)(\Psi(s_j^{\rm op})\!-\!1)-\gamma_j^*\otimes(\Psi(s_j^{\rm op})^{-1}\!-\!1)(\Psi(s^{\rm op}_j)^{-1}\!-\!1)]\right)\!\otimes\!1\nonumber\\
& &\mbox{}\times\left(b\otimes1\otimes1-\sum_{i=1}^{r}(\gamma_i\otimes\Psi(s^{\rm op}_i))\otimes1+\gamma_i^*\otimes\Psi(s^{\rm op}_i)^{-1})\otimes1)\right.\nonumber\\
& & \mbox{}-\left.\sum_{i=1}^r(\beta_i\otimes1\otimes\Psi(t_i)+\beta_i^*\otimes1\otimes \Psi(t_i)^{-1})\right)^{-1}\nonumber\\
& & \mbox{}\times\left(\gamma_j\!\otimes\!\left[\left(\Psi(s^{\rm op}_j)-1\right)\,\cdot\,\left(\Psi(\tilde{s}^{\rm op}_j)-1\right)\right]-\!\gamma_j^*\!\!\otimes\!\left[\left(\Psi(s^{\rm op}_j)^{-1}\!-1\right)\,\cdot\,\left(\Psi(\tilde{s}^{\rm op}_j)^{-1}\!-1\right)\right]\right)\!\otimes\!1\nonumber\\
&&\mbox{}\times\left(b\otimes1\otimes1-\sum_{i=1}^r(\gamma_i\otimes\Psi(\tilde{s}^{\rm op}_i)\otimes1+\gamma_i^*\otimes\Psi(\tilde{s}^{\rm op}_i)^{-1}\otimes1)\right.\nonumber\\
& & \mbox{}-\left.\sum_{i=1}^r(\beta_i\otimes1\otimes\Psi(t_i)+\beta_i^*\otimes1\otimes\Psi(t_i)^{-1})\right)^{-1}\nonumber\\
& &\mbox{}\times\frac{i}{2}\left(\gamma_k\!\otimes\![(\Psi(\tilde{s}_k^{\rm op})\!-\!1)\bullet{\color{blue}{(\Psi(\tilde{\sf s}_k^{\rm op})\!-\!1)}}]-\gamma_k^*\otimes[(\Psi(\tilde{s}_k^{\rm op})^{-1}\!-\!1)\bullet{\color{blue}{(\Psi(\tilde{\sf s}_k^{\rm op})^{-1}\!-\!1)}}]\right)\!\otimes\!1\nonumber\\
&&\mbox{}\times\left(b\otimes{\color{blue}{1}}\otimes1-\sum_{i=1}^r(\gamma_i\otimes{\color{blue}{\Psi(\tilde{\sf s}^{\rm op}_i)}}\otimes1+\gamma_i^*\otimes{\color{blue}{\Psi(\tilde{\sf s}^{\rm op}_i)^{-1}}}\otimes1)\right.\nonumber\\
& & \mbox{}-\left.\sum_{i=1}^r(\beta_i\otimes{\color{blue}{1}}\otimes\Psi(t_i)+\beta_i^*\otimes{\color{blue}{1}}\otimes\Psi(t_i)^{-1})\right)^{-1}\nonumber\\%
& & \!\!\!\!\!\!+\left(b\otimes{\color{red}1}\otimes1-\sum_{i=1}^r(\gamma_i\otimes{\color{red}\Psi(\mathsf s_i^{\rm op})}\!\otimes\!1\!+\!\gamma_i^*\!\otimes\!{\color{red}\Psi(\mathsf s_i^{\rm op})^{-1}}\!\otimes\!1)\nonumber\right.\\
& & \mbox{}-\left.\sum_{i=1}^r(\beta_i\otimes{\color{red}1}\otimes\Psi(t_i)+\beta_i^*\otimes{\color{red}1}\otimes\Psi(t_i)^{-1})\right)^{-1}\nonumber\\%
& & \mbox{}\times\delta_{j,k}\frac{i}{2}\left(\gamma_j\!\otimes\![{\color{red}(\Psi(\mathsf s_k^{\rm op})\!-\!1)}\bullet(\Psi(s_k^{\rm op})\!-\!1)^2+{\color{red}(\Psi(\mathsf s_k^{\rm op})\!-\!1)^2}\bullet(\Psi(s_k^{\rm op})\!-\!1)]\right.\nonumber\\%
& & \mbox{}+\left.\gamma_j^*\otimes[{\color{red}(\Psi(\mathsf s_k^{\rm op})^{-1}\!-\!1)^2}\bullet(\Psi(s^{\rm op}_j)^{-1}\!-\!1)+{\color{red}(\Psi(\mathsf s_k^{\rm op})^{-1}\!-\!1)}\bullet(\Psi(s^{\rm op}_j)^{-1}\!-\!1)^2]\right)\!\otimes\!1\nonumber\\
& &\mbox{}\times\left(b\otimes1\otimes1-\sum_{i=1}^{r}(\gamma_i\otimes\Psi(s^{\rm op}_i))\otimes1+\gamma_i^*\otimes\Psi(s^{\rm op}_i)^{-1})\otimes1)\right.\nonumber\\
& & \mbox{}-\left.\sum_{i=1}^r(\beta_i\otimes1\otimes\Psi(t_i)+\beta_i^*\otimes1\otimes \Psi(t_i)^{-1})\right)^{-1}\nonumber\\
& & \mbox{}\times\left(\gamma_j\otimes\left[\left(\Psi(s^{\rm op}_j)-1\right)\,\cdot\,\left(\Psi(\tilde{s}^{\rm op}_j)-1\right)\right]\right.\nonumber\\
& & \mbox{}-\left.\gamma_j^*\!\otimes\left[\left(\Psi(s^{\rm op}_j)^{-1}-1\right)\,\cdot\,\left(\Psi(\tilde{s}^{\rm op}_j)^{-1}-1\right)\right]\right)\!\otimes\!1\nonumber\\
&&\mbox{}\times\left(b\otimes1\otimes1-\sum_{i=1}^r(\gamma_i\otimes\Psi(\tilde{s}^{\rm op}_i)\otimes1+\gamma_i^*\otimes\Psi(\tilde{s}^{\rm op}_i)^{-1}\otimes1)\right.\nonumber\\
& & \mbox{}-\left.\sum_{i=1}^r(\beta_i\otimes1\otimes\Psi(t_i)+\beta_i^*\otimes1\otimes\Psi(t_i)^{-1})\right)^{-1}\nonumber\\
& &\mbox{}\times\frac{i}{2}\left(\gamma_k\!\otimes\![(\Psi(\tilde{s}_k^{\rm op})\!-\!1)\bullet{\color{blue}{(\Psi(\tilde{\sf s}_k^{\rm op})\!-\!1)}}]-\gamma_k^*\otimes[(\Psi(\tilde{s}_k^{\rm op})^{-1}\!-\!1)\bullet{\color{blue}{(\Psi(\tilde{\sf s}_k^{\rm op})^{-1}\!-\!1)}}]\right)\!\otimes\!1\nonumber\\
&&\mbox{}\times\left(b\otimes{\color{blue}{1}}\otimes1-\sum_{i=1}^r(\gamma_i\otimes{\color{blue}{\Psi(\tilde{\sf s}^{\rm op}_i)}}\otimes1+\gamma_i^*\otimes{\color{blue}{\Psi(\tilde{\sf s}^{\rm op}_i)^{-1}}}\otimes1)\right.\nonumber\\
& & \mbox{}-\left.\sum_{i=1}^r(\beta_i\otimes{\color{blue}{1}}\otimes\Psi(t_i)+\beta_i^*\otimes{\color{blue}{1}}\otimes\Psi(t_i)^{-1})\right)^{-1}\nonumber\\%
& & \!\!\!\!\!\!+\left(b\otimes{\color{red}1}\otimes1-\sum_{i=1}^r(\gamma_i\otimes{\color{red}\Psi(\mathsf s_i^{\rm op})}\!\otimes\!1\!+\!\gamma_i^*\!\otimes\!{\color{red}\Psi(\mathsf s_i^{\rm op})^{-1}}\!\otimes\!1)\nonumber\right.\\
& & \mbox{}-\left.\sum_{i=1}^r(\beta_i\otimes{\color{red}1}\otimes\Psi(t_i)+\beta_i^*\otimes{\color{red}1}\otimes\Psi(t_i)^{-1})\right)^{-1}\nonumber\\%
& & \mbox{}\times\left(\gamma_j\!\otimes\!{\color{red}(\Psi(\mathsf s_j^{\rm op})\!-\!1)^2}-\gamma_j^*\otimes{\color{red}(\Psi(\mathsf s_j^{\rm op})^{-1}\!-\!1)^2}\right)\!\otimes\!1\nonumber\\%
& & \mbox{}\times\left(b\otimes{\color{red}1}\otimes1-\sum_{i=1}^r(\gamma_i\otimes{\color{red}\Psi(\mathsf s_i^{\rm op})}\!\otimes\!1\!+\!\gamma_i^*\!\otimes\!{\color{red}\Psi(\mathsf s_i^{\rm op})^{-1}}\!\otimes\!1)\nonumber\right.\\
& & \mbox{}-\left.\sum_{i=1}^r(\beta_i\otimes{\color{red}1}\otimes\Psi(t_i)+\beta_i^*\otimes{\color{red}1}\otimes\Psi(t_i)^{-1})\right)^{-1}\nonumber\\
& & \mbox{}\times\frac{i}{2}\left(\gamma_k\!\otimes\![{\color{red}(\Psi(\mathsf s_k^{\rm op})\!-\!1)}\bullet(\Psi(s_k^{\rm op})\!-\!1)]\!-\!\gamma_k^*\!\otimes\![{\color{red}(\Psi(\mathsf s_k^{\rm op})^{-1}\!-\!1)}\bullet(\Psi(s_k^{\rm op})^{-1}\!-\!1)]\right)\!\otimes\!1 \nonumber\\
& &\mbox{}\times\left(b\otimes1\otimes1-\sum_{i=1}^{r}(\gamma_i\otimes\Psi(s^{\rm op}_i))\otimes1+\gamma_i^*\otimes\Psi(s^{\rm op}_i)^{-1})\otimes1)\right.\nonumber\\
& & \mbox{}-\left.\sum_{i=1}^r(\beta_i\otimes1\otimes\Psi(t_i)+\beta_i^*\otimes1\otimes \Psi(t_i)^{-1})\right)^{-1}\nonumber\\
& & \mbox{}\times\left(\gamma_j\otimes\left[\left(\Psi(s^{\rm op}_j)-1\right)\,\cdot\,\left(\Psi(\tilde{s}^{\rm op}_j)-1\right)\right]-\gamma_j^*\!\otimes\left[\left(\Psi(s^{\rm op}_j)^{-1}-1\right)\,\cdot\,\left(\Psi(\tilde{s}^{\rm op}_j)^{-1}-1\right)\right]\right)\!\otimes\!1\nonumber\\
&&\mbox{}\times\left(b\otimes1\otimes1-\sum_{i=1}^r(\gamma_i\otimes\Psi(\tilde{s}^{\rm op}_i)\otimes1+\gamma_i^*\otimes\Psi(\tilde{s}^{\rm op}_i)^{-1}\otimes1)\right.\nonumber\\
& & \mbox{}-\left.\sum_{i=1}^r(\beta_i\otimes1\otimes\Psi(t_i)+\beta_i^*\otimes1\otimes\Psi(t_i)^{-1})\right)^{-1}\nonumber\\
& &\mbox{}\times\frac{i}{2}\left(\gamma_k\!\otimes\![(\Psi(\tilde{s}_k^{\rm op})\!-\!1)\bullet{\color{blue}{(\Psi(\tilde{\sf s}_k^{\rm op})\!-\!1)}}]-\gamma_k^*\otimes[(\Psi(\tilde{s}_k^{\rm op})^{-1}\!-\!1)\bullet{\color{blue}{(\Psi(\tilde{\sf s}_k^{\rm op})^{-1}\!-\!1)}}]\right)\!\otimes\!1\nonumber\\
&&\mbox{}\times\left(b\otimes{\color{blue}{1}}\otimes1-\sum_{i=1}^r(\gamma_i\otimes{\color{blue}{\Psi(\tilde{\sf s}^{\rm op}_i)}}\otimes1+\gamma_i^*\otimes{\color{blue}{\Psi(\tilde{\sf s}^{\rm op}_i)^{-1}}}\otimes1)\right.\nonumber\\
& & \mbox{}-\left.\sum_{i=1}^r(\beta_i\otimes{\color{blue}{1}}\otimes\Psi(t_i)+\beta_i^*\otimes{\color{blue}{1}}\otimes\Psi(t_i)^{-1})\right)^{-1}\nonumber\\%
& & \!\!\!\!\!\!+\left(b\otimes{\color{red}1}\otimes1-\sum_{i=1}^r(\gamma_i\otimes{\color{red}\Psi(\mathsf s_i^{\rm op})}\!\otimes\!1\!+\!\gamma_i^*\!\otimes\!{\color{red}\Psi(\mathsf s_i^{\rm op})^{-1}}\!\otimes\!1)\nonumber\right.\\
& & \mbox{}-\left.\sum_{i=1}^r(\beta_i\otimes{\color{red}1}\otimes\Psi(t_i)+\beta_i^*\otimes{\color{red}1}\otimes\Psi(t_i)^{-1})\right)^{-1}\nonumber\\
& & \mbox{}\times\left(\gamma_j\!\otimes\!{\color{red}(\Psi(\mathsf s_j^{\rm op})\!-\!1)^2}-\gamma_j^*\otimes{\color{red}(\Psi(\mathsf s_j^{\rm op})^{-1}\!-\!1)^2}\right)\!\otimes\!1\nonumber\\
& &\mbox{}\times\left(b\otimes{\color{red}1}\otimes1-\sum_{i=1}^r(\gamma_i\otimes{\color{red}\Psi(\mathsf s_i^{\rm op})}\!\otimes\!1\!+\!\gamma_i^*\!\otimes\!{\color{red}\Psi(\mathsf s_i^{\rm op})^{-1}}\!\otimes\!1)\nonumber\right.\\
& & \mbox{}-\left.\sum_{i=1}^r(\beta_i\otimes{\color{red}1}\otimes\Psi(t_i)+\beta_i^*\otimes{\color{red}1}\otimes\Psi(t_i)^{-1})\right)^{-1}\nonumber\\%
& & \mbox{}\times\frac{i}{2}\delta_{j,k}\left(\gamma_j\otimes\left[{\color{red}\left(\Psi({\sf s}^{\rm op}_j)-1\right)}\bullet\left(\Psi(s^{\rm op}_j)-1\right)\,\cdot\,\left(\Psi(\tilde{s}^{\rm op}_j)-1\right)\right]\right.\nonumber\\%
& & \mbox{}+\left.\gamma_j^*\!\otimes\left[{\color{red}\left(\Psi(\mathsf s^{\rm op}_j)^{-1}-1\right)}\bullet\left(\Psi(s^{\rm op}_j)^{-1}-1\right)\,\cdot\,\left(\Psi(\tilde{s}^{\rm op}_j)^{-1}-1\right)\right]\right)\!\otimes\!1\nonumber\\
&&\mbox{}\times\left(b\otimes1\otimes1-\sum_{i=1}^r(\gamma_i\otimes\Psi(\tilde{s}^{\rm op}_i)\otimes1+\gamma_i^*\otimes\Psi(\tilde{s}^{\rm op}_i)^{-1}\otimes1)\right.\nonumber\\
& & \mbox{}-\left.\sum_{i=1}^r(\beta_i\otimes1\otimes\Psi(t_i)+\beta_i^*\otimes1\otimes\Psi(t_i)^{-1})\right)^{-1}\nonumber\\
& &\mbox{}\times\frac{i}{2}\left(\gamma_k\!\otimes\![(\Psi(\tilde{s}_k^{\rm op})\!-\!1)\bullet{\color{blue}{(\Psi(\tilde{\sf s}_k^{\rm op})\!-\!1)}}]-\gamma_k^*\otimes[(\Psi(\tilde{s}_k^{\rm op})^{-1}\!-\!1)\bullet{\color{blue}{(\Psi(\tilde{\sf s}_k^{\rm op})^{-1}\!-\!1)}}]\right)\!\otimes\!1\nonumber\\
&&\mbox{}\times\left(b\otimes{\color{blue}{1}}\otimes1-\sum_{i=1}^r(\gamma_i\otimes{\color{blue}{\Psi(\tilde{\sf s}^{\rm op}_i)}}\otimes1+\gamma_i^*\otimes{\color{blue}{\Psi(\tilde{\sf s}^{\rm op}_i)^{-1}}}\otimes1)\right.\nonumber\\
& & \mbox{}-\left.\sum_{i=1}^r(\beta_i\otimes{\color{blue}{1}}\otimes\Psi(t_i)+\beta_i^*\otimes{\color{blue}{1}}\otimes\Psi(t_i)^{-1})\right)^{-1}\quad\text{[end of 3rd term] }\text{END 8th TERM }
\label{45}\end{eqnarray}}}\noindent
In Parraud's expression \eqref{defRP} for $\nu_1,\nu_1^{(N)}$, one sums after all $j,k=1,\dots,r$ (of course, we will not write that expression down). 
Already before integrating with respect to the real variables $t_1$ and $t_2$ (see \eqref{defRP} and before), the result does not change regardless of whether we 
consider the tuples $s,\mathsf s,\tilde{s},\tilde{\sf s}$  - each a free tuple with the same distribution - as belonging to $\mathcal A$ or to $\mathcal A^{\rm op}$. 
However, this observation is not needed for our purposes. Indeed, obtaining $R_1((z-\mathcal S)^{-1})$ (where $R_1$ is applied to $(z-\mathcal S)^{-1}$ viewed as a
noncommutative function in the variables on the second tensor coordinate - strictly speaking, $({\rm id}_{M_m(\mathbb C)}\otimes R_1\otimes{\rm id}_\mathcal A)((z-\mathcal S)^{-1})
$) from formula \eqref{45} above comes (up to a complex multiplicative constant) to performing the following steps: first, sum after $j,k\in\{1,\dots,r\}$; second, identify the four
noncommutative functions with values in $M_m(\mathbb C)\otimes\mathcal A^{\rm op}\otimes\mathcal A$ (spatial tensor product) that appear left and right of each of the 
two $\bullet$ (and, of course, left and right of $\cdot$); third, evaluate these functions on the desired $r$-tuples of variables to obtain an element in the complex Banach 
space $M_m(\mathbb C)\otimes\mathcal A^{\rm op}\otimes\mathcal A$ (spatial tensor product again);
fourth, insert $1\in\mathcal A$ in the place of each placeholder $\bullet$, $\cdot$, and perform the needed operations of addition and multiplication in 
the von Neumann algebra $M_m(\mathbb C)\otimes\mathcal A\otimes\mathcal A$ endowed with the canonical trace\footnote{To
reiterate, instead of the second step, one may also just use the observation that the formula obtained after the first step does not change regardless of whether it is 
viewed in the von Neumann algebra $M_m(\mathbb C)\otimes\mathcal A^{\rm op}\otimes\mathcal A$ or in $M_m(\mathbb C)\otimes\mathcal A\otimes\mathcal A$,
``drop the op'' in the formula obtained from the first step, and then replace directly the placeholders with algebra units.}. Then computing 
$({\rm tr}_m\otimes\nu_1\otimes\tau)((z-\mathcal S)^{-1})$ becomes trivial. Computing $({\rm tr}_m\otimes\nu_1\otimes\nu_1)((z-\mathcal S)^{-1})$ is almost 
as easy: after performing the four steps above (or the three steps described in the footnote) for the second tensor coordinate, one performs precisely the same operations 
for the third tensor coordinate, and only then one applies ${\rm tr}_m\otimes\tau\otimes\tau$ to the element of $M_m(\mathbb C)\otimes\mathcal A\otimes\mathcal A$
thus obtained. Observing that the $r$-tuples $z_{t_1}^1,z_{t_1}^2,\tilde{z}_{t_1}^1,\tilde{z}_{t_1}^2$ have distributions that do not depend on
$t_1,t_2$, one is spared the need to integrates with respect to the real variables $t_1,t_2$. As the reader will see in the next section, this suffices in order to 
find the analytic function $E(z)$ from \eqref{estimdiffeqno}. Moreover, \eqref{45} shows that the domain of analyticity of $g(z)$ and the domain of analyticity of
$E(z)$ as functions of the complex number $z$ coincide - see the following section for details.

\section{Expansion of the Cauchy transform}\label{Sec:pfHT}

It is clear that if $|z|>\|\xi\|+2\sum_{i=1}^{r}(\|\gamma_i\|+\|\beta_i\|)$, then
{\small\begin{align}
&\Big(\!(zI_m\!\!-\xi)\!\otimes\! I_N\!\otimes\! I_N\!-\!\sum_{i=1}^r(\{\gamma_i\!\otimes\! U_i+\gamma_i^*\!\otimes\! U_i^{-1}\}\!\otimes \!I_N\!+
\beta_i\!\otimes\!I_N\!\otimes\!V_i\!+\beta_i^*\!\otimes\!I_N\!\otimes\!V_i^{-1})\Big)^{-1}\nonumber\\
&=\sum_{n=0}^\infty\frac{\left[\xi\!\otimes\!I_N\!\otimes\!I_N\!+\!\sum_{i=1}^r(\{\gamma_i\!\otimes\!U_i+\gamma_i^*\!\otimes\!U_i^{-1}\}\!\otimes\!I_N\!+
\beta_i\!\otimes\!I_N\!\otimes\!V_i\!+\beta_i^*\!\otimes\!I_N\!\otimes\!V_i^{-1})\right]^n}{z^{n+1}}\label{infty},
\end{align}}\noindent
where the right hand side series converges in norm. It can be easily seen that $(T_0+T_1+\cdots+T_{r})^n=\sum_{0\le i_1,\dots,i_n\le r}T_{i_1}T_{i_2}\cdots T_{i_n}$, a sum
of $(r+1)^n$ terms. For each $k\in\{0,1,\dots,n\}$, there are $r^{n-k}\frac{n!}{k!(n-k)!}$ terms that contain exactly $k$ instances of $T_0$ in this sum (i.e. for which exactly $k$ 
of the $i_1,\dots,i_n$ are equal to zero). We isolate the Cayley transforms $V_j$ in the above-displayed numerator by applying this formula with $T_0=\xi\otimes I_N\otimes I_N
+\sum_{i=1}^r(\gamma_i\otimes U_i+\gamma_i^*\otimes U^{-1}_i)\otimes I_N$ and $T_j=\beta_j\otimes I_N\otimes V_j+\beta_j^*\otimes I_N\otimes V_j^{-1},1\le j\le r$. Pick $i_1,\dots,i_n$ with $k$ of the $i$'s being equal to $0$,
and the other $n-k$ being $i_{\iota_1},\dots,i_{\iota_{n-k}}$. Then 
$$
T_{i_1}\cdots T_{i_n}=\sum_{\epsilon_{\iota_1},\dots,\epsilon_{\iota_{n-k}}\in\{\pm1\}}\tilde{T}_{i_1}\cdots\tilde{T}_{i_n}\otimes V_{i_{\iota_1}}^{\epsilon_{\iota_1}}\cdots 
V^{\epsilon_{\iota_{n-k}}}_{i_{\iota_{n-k}}},
$$
where $\tilde{T}_{i_j}\in\{\beta_{i_j}\otimes I_N,\beta_{i_j}^*\otimes I_N\}$ if $i_j\neq0$ and $\tilde{T}_{i_j}=\xi\otimes I_N+
\sum_{i=1}^{r}(\gamma_i\otimes U_i+\gamma_i^*\otimes U_i^{-1})$ if $i_j=0$. 
Thus, 
\begin{eqnarray*}
\lefteqn{(\mathrm{id}_{M_m(\mathbb C)}\!\otimes\!\mathrm{id}_{M_N(\mathbb C)}\!\otimes\!\mathbb{E}(\tr_N))\Big(\!(zI_m-\xi)\otimes I_N\otimes I_N}\\
& &\mbox{}-\sum_{i=1}^r(\{\gamma_i\otimes U_i+\gamma_i^*\otimes U_i^{-1}\}\otimes I_N+\beta_i\otimes I_N\otimes V_i+\beta_i^*\otimes I_N\otimes V_i^{-1})\Big)^{-1}\\
&\!\!\!=&\!\!\sum_{n=0}^\infty \sum_{k=0}^n\!\!\!\!\!\sum_{\begin{array}{ccc}0\le i_1,\dots,i_n\le r\\\mbox{\tiny $k$ of the $i$'s being equal to $0$,}\\\mbox{\tiny
and the other $n-k$ being $i_{\iota_1},\dots,i_{\iota_{n-k}}$}\end{array}}\!\!\!\!\!\!\!\!
\sum_{\epsilon_{\iota_1},\dots,\epsilon_{\iota_{n-k}}\in\{\pm1\}}\!\!\frac{\tilde{T}_{i_1}\cdots\tilde{T}_{i_n}\mathbb{E}\!\left(\!\tr_N(V_{i_{\iota_1}}^{\epsilon_{\iota_1}}\cdots 
V^{\epsilon_{\iota_{n-k}}}_{i_{\iota_{n-k}}})\!\right)}{z^{n+1}}.
\end{eqnarray*}
Similarly,
 \begin{eqnarray*}
\lefteqn{({\rm id}_{M_m(\mathbb C)}\!\otimes\!{\rm id}_{M_N(\mathbb C)}\!\otimes\!\tau)\Big(\!(zI_m-\xi)\otimes I_N\otimes1}\\
& & \mbox{}-\sum_{i=1}^r(\{\gamma_i\otimes U_i+\gamma_i^*\otimes U_i^{-1}\}\otimes1+\beta_i\otimes I_N\otimes v_i+\beta_i^*\otimes I_N\otimes v_i^{-1})\!\Big)^{-1}\\
&=&\sum_{n=0}^\infty \sum_{k=0}^n\!\!\! \sum_{\begin{array}{ccc}0\le i_1,\dots,i_n\le r\\\mbox{\tiny $k$ of the $i$'s being equal to $0$,}\\\mbox{\tiny
and the other $n-k$ being $i_{\iota_1},\dots,i_{\iota_{n-k}}$}\end{array}}\!\!\!\sum_{\epsilon_{\iota_1},\dots,\epsilon_{\iota_{n-k}}\in\{\pm1\}}\!\!\frac{\tilde{T}_{i_1}\cdots\tilde{T}_{i_n}\tau\left(v_{i_{\iota_1}}^{\epsilon_{\iota_1}}\cdots 
v^{\epsilon_{\iota_{n-k}}}_{i_{\iota_{n-k}}}\right)}{z^{n+1}}.
\end{eqnarray*}

Now, according to Proposition \ref{dvptpourCayley}.
\begin{eqnarray*}
\mathbb{E}\left(\!\tr_N(V_{i_{\iota_1}}^{\epsilon_{\iota_1}}\cdots V^{\epsilon_{\iota_{n-k}}}_{i_{\iota_{n-k}}})\!\right)&=&\tau\left(\!v_{i_{\iota_1}}^{\epsilon_{\iota_1}}\cdots 
v^{\epsilon_{\iota_{n-k}}}_{i_{\iota_{n-k}}}\!\right) + \frac{\nu_1\left(\!\Psi( \mathbf x_{i_{\iota_1}})^{\epsilon_{\iota_1}}\cdots 
\Psi( \mathbf x_{i_{\iota_{n-k}}})^{\epsilon_{\iota_{n-k}}}\!\right)}{N^2}\\&&+ \frac{\nu_2^{(N)}\left(\!\Psi( \mathbf x_{i_{\iota_1}})^{\epsilon_{\iota_1}}\cdots 
\Psi( \mathbf x_{i_{\iota_{n-k}}})^{\epsilon_{\iota_{n-k}}}\!\right)}{N^4}.
\end{eqnarray*}
As noted in Section \ref{domains}, for any fixed $z\in \mathbb C\setminus\mathbb R$, any fixed $N\times N$ unitary matrices $U_i$'s, the function
\begin{align*}
&({\mathbf x}_1,\ldots,{\mathbf x}_r)\mapsto\Big(z-1\otimes\xi\otimes I_N\\
& -\sum_{i=1}^r(\Psi({\mathbf x}_i)\otimes\beta_i\otimes I_N\!+\Psi({\mathbf x}_i)^{-1}\!\otimes\beta_i^* \otimes I_N\!
+1\otimes \gamma_i \otimes U_i+1\otimes \gamma_i^* \otimes U_i^{-1})\Big)^{\!-1}
\end{align*}
is a rational noncommutative function on $((I_r({\mathcal A}))_n)_{n\ge1}$
 and {\footnotesize 
\begin{eqnarray*}
\lefteqn{ {}^1\mathsf{flip}^0 {}^0\mathsf{flip}^1 \Big(z-1\otimes \xi \otimes I_N}\\
& & \mbox{}-\sum_{i=1}^{r}(\Psi({\mathbf x}_i)\otimes \beta_i \otimes I_N+\Psi({\mathbf x}_i)^{-1}\otimes 
\beta_i^* \otimes I_N+1\otimes \gamma_i \otimes U_i+1\otimes \gamma_i^* \otimes U_i^{-1})\Big)^{-1}\\
&=&\sum_{n=0}^\infty \frac{1}{z^{n+1}}\sum_{k=0}^n\!\sum_{\begin{array}{ccc}\mbox{\scriptsize$0\le i_1,\dots,i_n\le r$}\\ \mbox{\tiny $k$ of the $i$'s being equal to $0$,}\\\mbox{\tiny
and the other $n-k$ being $i_{\iota_1},\dots,i_{\iota_{n-k}}$}\end{array}}\\
& & \sum_{\epsilon_{\iota_1},\dots,\epsilon_{\iota_{n-k}}\in\{\pm1\}}\!\!{}^1\mathsf{flip}^0{}^0\mathsf{flip}^1\left\{{\Psi( \mathbf x_{i_{\iota_1}})^{\epsilon_{\iota_1}}\cdots 
\Psi( \mathbf x_{i_{\iota_{n-k}}})^{\epsilon_{\iota_{n-k}}}\otimes \tilde{T}_{i_1}\cdots\tilde{T}_{i_n}}\right\}.
\end{eqnarray*}}\noindent
We have 
$$ R_1^{\mathbf x}\left[ {}^1\mathsf{flip}^0 {}^0\mathsf{flip}^1 \left\{{\Psi( \mathbf x_{i_{\iota_1}})^{\epsilon_{\iota_1}}\cdots 
\Psi( \mathbf x_{i_{\iota_{n-k}}})^{\epsilon_{\iota_{n-k}}}\otimes \tilde{T}_{i_1}\cdots\tilde{T}_{i_n}}\right\}\right]$$
$$={}^1\mathsf{flip}^0 {}^0\mathsf{flip}^1  \left\{R_1^{\mathbf x}\left[ \Psi( \mathbf x_{i_{\iota_1}})^{\epsilon_{\iota_1}}\cdots 
\Psi( \mathbf x_{i_{\iota_{n-k}}})^{\epsilon_{\iota_{n-k}}}\right]\otimes \tilde{T}_{i_1}\cdots\tilde{T}_{i_n}\right\}.$$
Then, for any tuple  $(x^1, \tilde x^1, \tilde x^2, x^2)$ of $r$-tuples of selfadjoint elements in $\mathcal A_N$, 
$$(R_1\otimes{\rm id}_{M_m(\mathbb C)}\otimes{\rm id}_\mathcal A)\!\left[\!{}^1\mathsf{flip}^0{}^0\mathsf{flip}^1\left\{{\Psi(\mathbf x_{i_{\iota_1}})^{\epsilon_{\iota_1}}\cdots 
\Psi(\mathbf x_{i_{\iota_{n-k}}})^{\epsilon_{\iota_{n-k}}}\otimes\tilde{T}_{i_1}\cdots\tilde{T}_{i_n}}\!\right\}\!\right]\!(x^1, \tilde x^1, \tilde x^2, x^2)$$
$$= \tilde{T}_{i_1}\cdots\tilde{T}_{i_n} \otimes  R_1\left[ \Psi( \mathbf x_{i_{\iota_1}})^{\epsilon_{\iota_1}}\cdots 
\Psi( \mathbf x_{i_{\iota_{n-k}}})^{\epsilon_{\iota_{n-k}}}\right](x^1, \tilde x^1, \tilde x^2, x^2).$$
Similarly, for any $R\in \{R_2^{(1)}, R_2^{(2)}, R_2^{(3)}\}$, any tuple $(x^1,\tilde x^1, \tilde x^2, x^2)$ of 4-tuples of $r$-tuples of selfadjoint elements in $\mathcal A_N$, 
$$R^{\mathbf x}R_1^{\mathbf x}\left[ {}^1\mathsf{flip}^0 {}^0\mathsf{flip}^1 \left\{{\Psi( \mathbf x_{i_{\iota_1}})^{\epsilon_{\iota_1}}\cdots 
\Psi( \mathbf x_{i_{\iota_{n-k}}})^{\epsilon_{\iota_{n-k}}}\otimes \tilde{T}_{i_1}\cdots\tilde{T}_{i_n}}\right\}\right](x^1, \tilde x^1, \tilde x^2, x^2)$$
$$= \tilde{T}_{i_1}\cdots\tilde{T}_{i_n} \otimes  R^{\mathbf x}R_1^{\mathbf x}\left[ \Psi( \mathbf x_{i_{\iota_1}})^{\epsilon_{\iota_1}}\cdots 
\Psi( \mathbf x_{i_{\iota_{n-k}}})^{\epsilon_{\iota_{n-k}}}\right](x^1, \tilde x^1, \tilde x^2, x^2).$$
Applying the trace $\tau$ in the last tensor coordinate of the above,
{\footnotesize
\begin{eqnarray*} 
\lefteqn{({\rm id}_{M_m(\mathbb C)}\otimes{\rm id}_{M_N(\mathbb C)}\otimes\tau)\Big\{{}^1{\sf flip}^0 {}^0\mathsf{flip}^1\Big[R_1^{\mathbf x}\big(z-1\otimes\xi\otimes I_N}\\
& & \mbox{}-\sum_{i=1}^r(\Psi({\mathbf x}_i)\otimes \beta_i \otimes I_N +\Psi({\mathbf x}_i)^{-1}\otimes \beta_i^* \otimes I_N+1\otimes \gamma_i \otimes U_i+1\otimes \gamma_i^* \otimes U_i^{-1})\big)^{-1}\Big]\Big\}\\
&=&\sum_{n=0}^\infty \frac{1}{z^{n+1}} \sum_{k=0}^n\sum_{\begin{array}{ccc}\mbox{\scriptsize$0\le i_1,\dots,i_n\le r$}\\\mbox{\tiny $k$ of the $i$'s being equal to $0$,}\\\mbox{\tiny and the other $n-k$ being $i_{\iota_1},\dots,i_{\iota_{n-k}}$}\end{array}}\\
& & \sum_{\epsilon_{\iota_1},\dots,\epsilon_{\iota_{n-k}}\in\{\pm1\}}\tau\left\{ R_1^{\mathbf x}\left[\Psi( \mathbf x_{i_{\iota_1}})^{\epsilon_{\iota_1}}\cdots 
\Psi( \mathbf x_{i_{\iota_{n-k}}})^{\epsilon_{\iota_{n-k}}}\right](z_{t_1}^1,\tilde z_{t_1}^1, \tilde z_{t_1}^2, z_{t_1}^2) \right\}\tilde{T}_{i_1}\cdots\tilde{T}_{i_n}.
\end{eqnarray*}}\noindent
Similarly, for any $R\in \{R_2^{(1)}, R_2^{(2)}, R_2^{(3)}\}$,
{\tiny
\begin{eqnarray*} 
\lefteqn{({\rm id}_{M_m(\mathbb C)}\otimes{\rm id}_{M_N(\mathbb C)}\otimes\tau_N)\left\{{}^1{\sf flip}^0{}^0{\sf flip}^1\left[R^{\bf x}R_1^{\bf x}(z-1\otimes\xi\otimes I_N -
\sum_{i=1}^r(\Psi({\mathbf x}_i)\otimes \beta_i \otimes I_N+\Psi({\mathbf x}_i)^{-1}\otimes \beta_i^* \otimes I_N
+1\otimes \gamma_i \otimes U_i+1\otimes \gamma_i^* \otimes U_i^{-1}))^{-1}\right]\right\}}\\
& = & \sum_{n=0}^\infty\frac{1}{z^{n+1}} \sum_{k=0}^n\sum_{\begin{array}{ccc}0\le i_1,\dots,i_n\le r\\\mbox{\tiny $k$ of the $i$'s being equal to $0$,}\\\mbox{\tiny
and the other $n-k$ being $i_{\iota_1},\dots,i_{\iota_{n-k}}$}\end{array}}\sum_{\epsilon_{\iota_1},\dots,\epsilon_{\iota_{n-k}}\in\{\pm1\}}\\
& & \tau_N\left\{ R^{\mathbf x} R_1^{\mathbf x}\left[\Psi( \mathbf x_{i_{\iota_1}})^{\epsilon_{\iota_1}}\cdots 
\Psi( \mathbf x_{i_{\iota_{n-k}}})^{\epsilon_{\iota_{n-k}}}\right]({X^{N,T_2}_{3,1}, \tilde X^{N,T_2}_{3,1}, \tilde X^{N,T_2}_{3,2},X^{N,T_2}_{3,2}}) \right\}\tilde{T}_{i_1}\cdots
\tilde{T}_{i_n}.\quad\quad\quad\quad\quad\quad\quad\quad\quad
\end{eqnarray*}}\noindent
It readily follows that 
{\tiny  \begin{eqnarray*}
\lefteqn{({\rm id}_{M_m(\mathbb C)}\otimes{\rm id}_{M_N(\mathbb C)}\otimes \mathbb{E}(\tr_N))\Big((zI_m-\xi)\otimes I_N\otimes I_N-2\Re\sum_{i=1}^r(\gamma_i\otimes U_i\otimes I_N+\beta_i\otimes I_N\otimes V_i)\Big)^{-1}}\\
& = & ({\rm id}_{M_m(\mathbb C)}\otimes{\rm id}_{M_N(\mathbb C)} \otimes \tau)\Big((zI_m-\xi)\otimes I_N\otimes I_{\mathcal A}-2\Re\sum_{i=1}^r(\gamma_i\otimes U_i\otimes I_{\mathcal A}+\beta_i\otimes I_N\otimes v_i)\Big)^{-1}\\
& & \mbox{}+\frac{1}{2N^2}\int_0^{+\infty}\int_0^{t_2}  e^{-t_2-t_1}({\rm id}_{M_m(\mathbb C)}\otimes{\rm id}_{M_N(\mathbb C)} \otimes \tau)\\
& &\left\{{}^1\mathsf{flip}^0 {}^0\mathsf{flip}^1\left[R_1^{\mathbf x}\left(z-1\otimes \xi \otimes I_N -\sum_{i=1}^r(\Psi(\mathbf x_i)\otimes \beta_i \otimes I_N+\Psi(\mathbf x_i)^{-1}\otimes \beta_i^* \otimes I_N+1\otimes \gamma_i \otimes U_i+1\otimes \gamma_i^* \otimes U_i^{-1})\right)^{-1}\right]\right.\\
& &\left.( z^1_{t_1}, \tilde z^1_{t_1}, \tilde z^2_{t_1}, z_{t_1}^2)\right\}dt_1 dt_2\\
\\
& &\mbox{}+\frac1{4N^4}\int_{A_2}e^{-t_4-t_3-t_2-t_1}\left\{\1_{[t_2,t_4]}(t_3)({\rm id}_{M_m(\mathbb C)}\otimes{\rm id}_{M_N(\mathbb C)}\otimes\mathbb{E}(\tau_N))\right.\\
& & \left({}^1\mathsf{flip}^0{}^0\mathsf{flip}^1\left[R_{2}^{(1){\mathbf x}}R_1^{\mathbf x}\left(z-1\otimes \xi \otimes I_N -\sum_{i=1}^r(\Psi(\mathbf x_i)\otimes \beta_i \otimes I_N +\Psi(\mathbf x_i)^{-1}\otimes \beta_i^* \otimes I_N+1\otimes \gamma_i \otimes U_i+1\otimes \gamma_i^* \otimes U_i^{-1})\right)^{-1}\right]\right.\\
& &\left.\frac{}{}\!\!({X^{N,T_2}_{3,1}, \tilde X^{N,T_2}_{3,1}, \tilde X^{N,T_2}_{3,2},X^{N,T_2}_{3,2}})\right)\\
& &\mbox{}+\1_{[0,t_1]}(t_3)({\rm id}_{M_m(\mathbb C)} \otimes{\rm id}_{M_N(\mathbb C)}\otimes\mathbb{E} (\tau_N)) \\
& &\left({}^1{\sf flip}^0 {}^0{\sf flip}^1\left[R_{2}^{(2),{\bf x}}R_1^{\bf x}\left(z-1\otimes \xi \otimes I_N -\sum_{i=1}^r(\Psi(t_i)\otimes \beta_i \otimes I_N+\Psi(t_i)^{-1}\otimes \beta_i^* \otimes I_N+1\otimes \gamma_i \otimes U_i+1\otimes \gamma_i^* \otimes U_i^{-1})\right)^{-1}\right]\right.\\
& &(X^{N,T_2}_{1,1}, \tilde X^{N,T_2}_{1,1}, \tilde X^{N,T_2}_{1,2},X^{N,T_2}_{1,2})\\
& & \mbox{}+\1_{[t_1,t_2]}(t_3)({\rm id}_{M_m(\mathbb C)}\otimes{\rm id}_{M_N(\mathbb C)}\otimes  \mathbb{E} (\tau_N))\\ 
& & \left({}^1\mathsf{flip}^0 {}^0\mathsf{flip}^1\left[R_{2}^{(3),{\mathbf x}}R_1^{\mathbf x}\left(z-1\otimes \xi \otimes I_N -\sum_{i=1}^r(\Psi(t_i)\otimes \beta_i \otimes I_N +
\Psi(t_i)^{-1}\otimes \beta_i^* \otimes I_N+1\otimes \gamma_i \otimes U_i+1\otimes \gamma_i^* \otimes U_i^{-1})\right)^{-1}\right] \right.\\
& & \left. \left.(X^{N,T_2}_{2,1}, \tilde X^{N,T_2}_{2,1}, \tilde X^{N,T_2}_{2,2},X^{N,T_2}_{2,2})\right)\right\}{\rm d}t_1{\rm d}t_2{\rm d}t_3{\rm d}t_4.
\end{eqnarray*}}\noindent
Let us iterate the process with respect to the $U_i$'s.
It is clear that if $|z|>\|\xi\|+2\sum_{i=1}^{r}(\|\gamma_i\|+\|\beta_i\|)$, then
{\small\begin{align}
&\Big(\!(zI_m\!\!-\xi)\!\otimes\! I_N\!\otimes\!1-\!\sum_{i=1}^r(\{\gamma_i\!\otimes\! U_i+\gamma_i^*\!\otimes\! U_i^{-1}\}\!\otimes1\!+
\beta_i\!\otimes\!I_N\!\otimes\!v_i\!+\beta_i^*\!\otimes\!I_N\!\otimes\!v_i^{-1})\Big)^{-1}\nonumber\\
&=\sum_{n=0}^\infty\frac{\left[\xi\!\otimes\!I_N\!\otimes1\!+\!\sum_{i=1}^r(\{\gamma_i\!\otimes\!U_i+\gamma_i^*\!\otimes\!U_i^{-1}\}\!\otimes1\!+
\beta_i\!\otimes\!I_N\!\otimes\!v_i\!+\beta_i^*\!\otimes\!I_N\!\otimes\!v_i^{-1})\right]^n}{z^{n+1}}\label{infty2},
\end{align}}\noindent
where the right hand side series converges in norm. It can be easily seen that $(T_0+T_1+\cdots+T_{r})^n=\sum_{0\le i_1,\dots,i_n\le r}T_{i_1}T_{i_2}\cdots T_{i_n}$, a sum
of $(r+1)^n$ terms. For each $k\in\{0,1,\dots,n\}$, there are $r^{n-k}\frac{n!}{k!(n-k)!}$ terms that contain exactly 
$k$ instances of $T_0$ in this sum (i.e. for which exactly $k$ 
of the $i_1,\dots,i_n$ are equal to zero). We isolate the Cayley transforms $U_j$ in the above-displayed numerator by applying this formula with $T_0=\xi\otimes I_N\otimes1
+\sum_{i=1}^r(\beta_i\otimes I_N\otimes v_i+ \beta_i^*\otimes I_N\otimes v_i^{-1}$ and $T_j=\{\gamma_j\otimes U_j+\gamma_j^*\otimes U_j^{-1}\}\otimes1,1\le j\le r$. Pick 
$i_1,\dots,i_n$ with $k$ of the $i$'s being equal to $0$, and the other $n-k$ being $i_{\iota_1},\dots,i_{\iota_{n-k}}$. Then 
$$
T_{i_1}\cdots T_{i_n}=\sum_{\epsilon_{\iota_1},\dots,\epsilon_{\iota_{n-k}}\in\{\pm1\}}\tilde{T}_{i_1}\cdots\tilde{T}_{i_n}(I_m\otimes U_{i_{\iota_1}}^{\epsilon_{\iota_1}}\otimes1)
\cdots (I_m\otimes U^{\epsilon_{\iota_{n-k}}}_{i_{\iota_{n-k}}}\otimes1),
$$
where $\tilde{T}_{i_j}\in\{\gamma_{i_j}\otimes I_N\otimes1,\gamma_{i_j}^*\otimes I_N\otimes1\}$ if $i_j\neq0$ and $\tilde{T}_{i_j}=\xi\otimes I_N\otimes1+
\sum_{i=1}^r(\beta_i\otimes I_N\otimes v_i+\beta_i^*\otimes I_N\otimes v_i^{-1})$ if $i_j=0$. 
Thus \begin{eqnarray*}
\lefteqn{({\rm id}_{M_m(\mathbb C)}\otimes\mathbb{E}(\tr_N)\otimes{\rm id}_{\mathcal A})\Big((zI_m-\xi)\otimes I_N\otimes1}\\
& & \mbox{}-\sum_{i=1}^r(\{\gamma_i\otimes U_i+\gamma_i^*\otimes U_i^{-1}\}\otimes1+\beta_i\otimes I_N\otimes v_i+\beta_i^*\otimes I_N\otimes v_i^{-1}\Big)^{-1}\\
&\!\!=&\!\sum_{n=0}^\infty \sum_{k=0}^n\!\!\!\!\!\sum_{\begin{array}{ccc}\mbox{\scriptsize$0\le i_1,\dots,i_n\le r$}\\\mbox{\tiny $k$ of the $i$'s being equal to $0$,}\\\mbox{\tiny
and the other $n-k$ being $i_{\iota_1},\dots,i_{\iota_{n-k}}$}\end{array}}\!\!\!\!\!\!\sum_{\epsilon_{\iota_1},\dots,\epsilon_{\iota_{n-k}}\in\{\pm1\}}\!\!\!
\frac{\hat{T}_{i_1}\cdots\hat{T}_{i_n}\mathbb{E}\left(\tr_N(U_{i_{\iota_1}}^{\epsilon_{\iota_1}}\cdots U^{\epsilon_{\iota_{n-k}}}_{i_{\iota_{n-k}}})\right)}{z^{n+1}},
\end{eqnarray*}
where $\hat{T}_{i_j}\in\{\gamma_{i_j}\otimes1,\gamma_{i_j}^*\otimes1\}$ if $i_j\neq0$ and $\hat{T}_{i_j}=\xi\otimes 1+
\sum_{i=1}^r(\beta_i\otimes v_i+\beta_i^*\otimes v_i^{-1})$ if $i_j=0$. 
Similarly,
{\small\begin{eqnarray*}
\lefteqn{({\rm id}_{M_m(\mathbb C)}\otimes\tau\otimes{\rm id}_{\mathcal A})\Big((zI_m-\xi)\otimes1\otimes1}\\
& & \mbox{}-\sum_{i=1}^r(\{\gamma_i\otimes u_i+\gamma_i^*\otimes u_i^{-1}\}\otimes1+\beta_i\otimes1\otimes \Psi(s_i)+\beta_i^*\otimes1\otimes \Psi(s_i)^{-1})\Big)^{-1}\\
&=&\sum_{n=0}^\infty \sum_{k=0}^n \!\!\!\!\!\sum_{\begin{array}{ccc}\mbox{\scriptsize$0\le i_1,\dots,i_n\le r$}\\\mbox{\tiny $k$ of the $i$'s being equal to $0$,}\\\mbox{\tiny
and the other $n-k$ being $i_{\iota_1},\dots,i_{\iota_{n-k}}$}\end{array}}\!\!\!\!\!\!\!\sum_{\epsilon_{\iota_1},\dots,\epsilon_{\iota_{n-k}}\in\{\pm1\}}\frac{\hat{T}_{i_1}\cdots\hat{T}_{i_n}\tau\left(u_{i_{\iota_1}}^{\epsilon_{\iota_1}}\cdots 
u^{\epsilon_{\iota_{n-k}}}_{i_{\iota_{n-k}}}\right)}{z^{n+1}}.
\end{eqnarray*}}\noindent
Now, according to Proposition \ref{dvptpourCayley},
\begin{eqnarray*}\mathbb{E}\!\left(\!\tr_N(U_{i_{\iota_1}}^{\epsilon_{\iota_1}}\cdots U^{\epsilon_{\iota_{n-k}}}_{i_{\iota_{n-k}}})\right)
& \!= &\! \tau\!\left(\!u_{i_{\iota_1}}^{\epsilon_{\iota_1}}\cdots u^{\epsilon_{\iota_{n-k}}}_{i_{\iota_{n-k}}}\right)
 + \frac{\nu_1^{(N)}\left(\Psi( \mathbf x_{i_{\iota_1}})^{\epsilon_{\iota_1}}\cdots \Psi( \mathbf x_{i_{\iota_{n-k}}})^{\epsilon_{\iota_{n-k}}}\right)}{N^2}\\
& \!= &\! \tau\!\left(\!u_{i_{\iota_1}}^{\epsilon_{\iota_1}}\cdots 
u^{\epsilon_{\iota_{n-k}}}_{i_{\iota_{n-k}}}\right)
 + \frac{\nu_1\left(\Psi( \mathbf x_{i_{\iota_1}})^{\epsilon_{\iota_1}}\cdots 
\Psi( \mathbf x_{i_{\iota_{n-k}}})^{\epsilon_{\iota_{n-k}}}\right)}{N^2}\\
&& \mbox{}+ \frac{\nu_2^{(N)}\left(\Psi( \mathbf x_{i_{\iota_1}})^{\epsilon_{\iota_1}}\cdots 
\Psi( \mathbf x_{i_{\iota_{n-k}}})^{\epsilon_{\iota_{n-k}}}\right)}{N^4}.
\end{eqnarray*} 
For any fixed $z\in \mathbb C\setminus \mathbb R$, any fixed  $v_i$'s, the function
\begin{align*}
&({\mathbf y}_1,\ldots,{\mathbf y}_r)\mapsto\Big(z-1\otimes \xi \otimes 1 \\
&-\sum_{i=1}^r(\Psi({\mathbf y}_i)\otimes \gamma_i \otimes1+\Psi({\mathbf y}_i)^{-1}\otimes \gamma_i^* \otimes1
+1\otimes \beta_i \otimes v_i+1\otimes \beta_i^* \otimes v_i^{-1})\Big)^{-1}
\end{align*}
 is a rational noncommutative function on $((I_r({\mathcal A}))_n)_{n\geq 1}$ and 
{\small 
 \begin{eqnarray*}
 \lefteqn{{}^0\mathsf{flip}^1\Big(z-1\otimes \xi \otimes 1}\\
& & \mbox{}-\sum_{i=1}^r(1\otimes \beta_i \otimes v_i +1\otimes \beta_i^* \otimes v_i^{-1}
+\Psi({\mathbf y}_i)\otimes \gamma_i \otimes1+\Psi({\mathbf y}_i)^{-1}\otimes \gamma_i^* \otimes 1\Big)^{-1}\\
&=&\sum_{n=0}^\infty \frac{1}{z^{n+1}}\sum_{k=0}^n\!\!\!\!\!\sum_{\begin{array}{ccc}\mbox{\scriptsize$0\le i_1,\dots,i_n\le r$}\\ 
\mbox{\tiny $k$ of the $i$'s being equal to $0$,}\\\mbox{\tiny and the other $n-k$ being $i_{\iota_1},\dots,i_{\iota_{n-k}}$}\end{array}}\\
& & \sum_{\epsilon_{\iota_1},\dots,\epsilon_{\iota_{n-k}}\in\{\pm1\}}{}^0\mathsf{flip}^1 \left\{{\Psi( \mathbf y_{i_{\iota_1}})^{\epsilon_{\iota_1}}\cdots 
\Psi( \mathbf y_{i_{\iota_{n-k}}})^{\epsilon_{\iota_{n-k}}}\otimes \hat{T}_{i_1}\cdots\hat{T}_{i_n}}\right\}.
\end{eqnarray*}}\noindent
Thus,  we can  deduce as below that 
{\tiny
\begin{eqnarray*}
\lefteqn{({\rm id}_{M_m(\mathbb C)}\otimes \mathbb{E}(\tr_N)\otimes{\rm id}_{\mathcal A})\Big((zI_m-\xi)\otimes I_{N}\otimes1-\sum_{i=1}^r(\{\gamma_i\otimes U_i+\gamma_i^*
\otimes U_i^{-1}\}\otimes1+\beta_i\otimes I_N\otimes v_i+\beta_i^*\otimes I_N\otimes v_i^{-1})\Big)^{-1}}\\
&= &({\rm id}_{M_m(\mathbb C)}\otimes\tau\otimes{\rm id}_{\mathcal A})\Big((zI_m-\xi)\otimes1\otimes1-\sum_{i=1}^r(\{\gamma_i\otimes u_i+\gamma_i^*\otimes u_i^{-1}\}\otimes1
+\beta_i\otimes1\otimes v_i+\beta_i^*\otimes1\otimes v_i^{-1})\Big)^{-1}\\
&  & \mbox{}+\frac{1}{2N^2}\int_0^{+\infty}\!\!\int_0^{t_2} e^{-t_2-t_1}({\rm id}_{M_m(\mathbb C)}\otimes\tau\otimes{\rm id}_{\mathcal A}) \\
&  &\left\{{}^0\mathsf{flip}^1\left[R_1^{\mathbf y}\left(z-1\otimes\xi\otimes 1-\sum_{i=1}^r(\Psi(\mathbf y_i)\otimes \gamma_i \otimes1-\Psi(\mathbf y_i)^{-1}\otimes\gamma_i^*
\otimes1+1\otimes\beta_i\otimes v_i+1\otimes\beta_i^* \otimes v_i^{-1})\right)^{-1}\right]\right.\\
&  & \left.( z^1_{t_1}, \tilde z^1_{t_1}, \tilde z^2_{t_1}, z_{t_1\frac{}{}}^{2}\!\!)\right\}\,{\rm d}t_1{\rm d}t_2\\
& &\mbox{}+\frac1{4N^4}\int_{A_2}e^{-t_4-t_3-t_2-t_1}\left\{\1_{[t_2,t_4]}(t_3)({\rm id}_{M_m(\mathbb C)}\otimes\mathbb{E}(\tau_N)\otimes{\rm id}_{\mathcal A})\right.\\
&  &\left({}^0\mathsf{flip}^1\left[R_{2}^{(1),{\mathbf y}}R_1^{\mathbf y}\left(z-1\otimes \xi \otimes1-\sum_{i=1}^r(\Psi(\mathbf y_i)\otimes\gamma_i\otimes1
+\Psi(\mathbf y_i)^{-1}\otimes \gamma_i^* \otimes1+1\otimes\beta_i\otimes v_i+1\otimes \gamma_i^* \otimes v_i^{-1})\right)^{-1}\right]\right.\\
&  & \left.({X^{N,T_2}_{3,1}, \tilde X^{N,T_2}_{3,1}, \tilde X^{N,T_2}_{3,2},X^{N,T_2}_{3,2}})\right)\\
& & \mbox{}+\1_{[0,t_1]}(t_3)({\rm id}_{M_m(\mathbb C)}\otimes \mathbb{E}(\tau_N) \otimes{\rm id}_{\mathcal A})\\
& & \left({}^0\mathsf{flip}^1\left[R_{2}^{(2),{\mathbf y}}R_1^{\mathbf y}\left(z-1\otimes \xi\otimes1-\sum_{i=1}^r(\Psi(\mathbf y_i)\otimes\gamma_i\otimes1
+\Psi(\mathbf y_i)^{-1}\otimes \gamma_i^* \otimes1+1\otimes \beta_i \otimes v_i+1\otimes \gamma_i^* \otimes v_i^{-1})\right)^{-1}\right]\right.\\
& & (X^{N,T_2}_{1,1}, \tilde X^{N,T_2}_{1,1}, \tilde X^{N,T_2}_{1,2},X^{N,T_2}_{1,2})\\
& & \mbox{}+\1_{[t_1,t_2]}(t_3)({\rm id}_{M_m(\mathbb C)}\otimes \mathbb{E}(\tau_N) \otimes{\rm id}_{\mathcal A}) \\
& & \left({}^0\mathsf{flip}^1\left[R_{2}^{(3),{\mathbf y}}R_1^{\mathbf y}\left(z-1\otimes\xi\otimes1-\sum_{i=1}^r(\Psi(\mathbf y_i)\otimes \gamma_i \otimes1
+\Psi(\mathbf y_i)^{-1}\otimes \gamma_i^* \otimes1+1\otimes \beta_i \otimes v_i+1\otimes \gamma_i^* \otimes v_i^{-1})\right)^{-1}\right]\right.\\
&  & \left. \left.(X^{N,T_2}_{2,1}, \tilde X^{N,T_2}_{2,1}, \tilde X^{N,T_2}_{2,2},X^{N,T_2}_{2,2})\right)\right\}\\
& & {\rm d}t_1{\rm d}t_2{\rm d}t_3{\rm d}t_4.
\end{eqnarray*}}\noindent
Now, 
{\tiny\begin{align*}({\rm id}_{M_m(\mathbb C)}\otimes \mathbb{E}(\tr_N) \otimes{\rm id}_{\mathcal A})&\hskip10truecm\quad\quad\quad\end{align*}}\vskip-.5truecm
{\tiny \begin{eqnarray*}
\lefteqn{\left[{}^1\mathsf{flip}^0{}^0\mathsf{flip}^1\left\{R_1^{\mathbf x}\left(z-1\otimes\xi\otimes I_N -\sum_{i=1}^r(\Psi(\mathbf x_i)\otimes\beta_i\otimes I_N
+\Psi(\mathbf x_i)^{-1}\otimes\beta_i^*\otimes I_N+1\otimes \gamma_i \otimes U_i+1\otimes \gamma_i^* \otimes U_i^{-1})\right)^{-1}\right]\right.}\\
& & \left.( z^1_{t_1\frac{}{}}\!, \tilde z^1_{t_1}, \tilde z^2_{t_1}, z_{t_1}^2)\right\}\\
&\!\!\! =&\!({\rm id}_{M_m(\mathbb C)}\otimes\tau\otimes{\rm id}_{\mathcal A})\\
&  & \left\{{}^1\mathsf{flip}^0 {}^0\mathsf{flip}^1\left[R_1^{\mathbf x}\left(z-1\otimes \xi \otimes1+\sum_{i=1}^r(\Psi(\mathbf x_i)\otimes \beta_i \otimes1
+\Psi(\mathbf x_i)^{-1}\otimes \beta_i^* \otimes1+1\otimes\gamma_i \otimes u_i+\otimes \gamma_i^* \otimes u_i^{-1})\right)^{-1}\right]\right.\\
&  & \left.( z^1_{t_1\frac{}{}}\!, \tilde z^1_{t_1}, \tilde z^2_{t_1}, z_{t_1}^2)\right\}\\
& &\mbox{}+\frac{1}{N^2}({\rm id}_{M_m(\mathbb C)}\otimes\mathbb E\left(\tau_N \right)\otimes{\rm id}_{\mathcal A})\\
&  & {}^0{\sf flip}^1R_1^{\mathbf y}{}^0{\sf flip}^1\left\{\left[{}^1\mathsf{flip}^0{}^0\mathsf{flip}^1R_1^{\mathbf x}\left(z-1\otimes \xi \otimes1+\sum_{i=1}^r(\Psi(\mathbf x_i)\otimes \beta_i \otimes1
+\Psi(\mathbf x_i)^{-1}\otimes \beta_i^* \otimes1+1\otimes \gamma_i\otimes\Psi(y_i)+1\otimes \gamma_i^* \otimes \Psi(y_i)^{-1})\right)^{-1}\right]\right.\\
&  & \left.( z^1_{t_1\frac{}{}}\!, \tilde z^1_{t_1}, \tilde z^2_{t_1}, z_{t_1}^2)\right\}(z^1_{t_1}(N), \tilde z^1_{t_1}(N), \tilde z^2_{t_1}(N), z_{t_1}^2(N))
\end{eqnarray*}}\noindent
Finally, for $|z|>\|\xi\|+2\sum_{i=1}^{r_1}\|\gamma_i\|+2\sum_{i=1}^{r_2}\|\beta_i\|$, with $g_N$ and $g$ defined by \eqref{defgn} and \eqref{defg}, we have 
$$
g_N(z)=g(z)+\frac{E(z)}{N^2} +\frac{\Delta_N(z)}{N^4},
$$
{\tiny\begin{eqnarray*}
E(z) & \!\!= & \!\!\frac{1}{2}\int_0^{+\infty}\!\!\int_0^{t_2}   e^{-t_2-t_1}(\tr_m\otimes \tau \otimes \tau) \\
& & \left\{{}^0\mathsf{flip}^1\left[R_1^{\mathbf y}\left(z-1\otimes\xi\otimes1-\sum_{i=1}^r(\Psi(\mathbf y_i)\otimes\gamma_i\otimes1+\Psi(\mathbf y_i)^{-1}\otimes\gamma_i^*\otimes1
+1\otimes \beta_i \otimes v_i+1\otimes \gamma_i^* \otimes v_i^{-1})\right)^{-1}\right]\right.\\
& & \left.( z^1_{t_1\frac{}{}}\!\!, \tilde z^1_{t_1}, \tilde z^2_{t_1}, z_{t_1}^2)\right\}\,{\rm d}t_1{\rm  d}t_2\\
& &\mbox{}+\frac12\int_0^{+\infty}\!\!\int_0^{t_2}   e^{-t_2-t_1}(\tr_m \otimes \tau \otimes \tau)\\
& & \left\{{}^1\mathsf{flip}^0{}^0\mathsf{flip}^1\left[R_1^{\mathbf x}\left(z-1\otimes \xi \otimes I_N -\sum_{i=1}^r(\Psi(\mathbf x_i)\otimes\beta_i \otimes1
+\Psi(\mathbf x_i)^{-1}\otimes \beta_i^* \otimes1+1\otimes \gamma_i \otimes u_i+1\otimes \gamma_i^* \otimes u_i^{-1})\right)^{-1}\right]\right.\\
& & \left.( z^1_{t_1\frac{}{}}\!\!, \tilde z^1_{t_1}, \tilde z^2_{t_1}, z_{t_1}^2)\right\}\,{\rm d}t_1{\rm d}t_2,
\end{eqnarray*}
\begin{eqnarray*}
\Delta_N(z)&\!\!=&\!\!\frac{1}{4}\int_{A_2}  e^{-t_4-t_3-t_2-t_1}\left\{\1_{[t_2,t_4]}(t_3)(\tr_m \otimes \mathbb{E} (\tr_N)\otimes  \mathbb{E}(\tau_N))\frac{}{}\right.\\
& & \left({}^1\mathsf{flip}^0{}^0\mathsf{flip}^1\left[R_{2}^{(1){\mathbf x}}R_1^{\mathbf x}\left(z-1\otimes\xi\otimes I_N-\sum_{i=1}^r(\Psi(\mathbf x_i)\otimes\beta_i\otimes I_N 
+\Psi(\mathbf x_i)^{-1}\otimes \beta_i^* \otimes I_N+1\otimes \gamma_i \otimes U_i+1\otimes \gamma_i^* \otimes U_i^{-1})\right)^{-1}\right]\right.\\
& & \left.\frac{}{}({X^{N,T_2}_{3,1}, \tilde X^{N,T_2}_{3,1}, \tilde X^{N,T_2}_{3,2},X^{N,T_2}_{3,2}})\right)\\
& &\mbox{}+\1_{[0,t_1]}(t_3)(\tr_m \otimes \mathbb{E} (\tr_N)\otimes\mathbb{E} (\tau_N) )\\
& & \left({}^1\mathsf{flip}^0 {}^0\mathsf{flip}^1\left[R_{2}^{(2),{\mathbf x}}R_1^{\mathbf x}\left(z-1\otimes\xi\otimes I_N-\sum_{i=1}^r(\Psi(t_i)\otimes \beta_i \otimes I_N 
+\Psi(t_i)^{-1}\otimes \beta_i^* \otimes I_N+1\otimes \gamma_i \otimes U_i+1\otimes \gamma_i^* \otimes U_i^{-1})\right)^{-1}\right]\right.\\
& & (X^{N,T_2}_{1,1}, \tilde X^{N,T_2}_{1,1}, \tilde X^{N,T_2}_{1,2},X^{N,T_2}_{1,2})\\
& & \mbox{}+\1_{[t_1,t_2]}(t_3)(\tr_m\otimes \mathbb{E}(\tr_N)\otimes  \mathbb{E} (\tau_N))\\
& & \left({}^1\mathsf{flip}^0 {}^0\mathsf{flip}^1\left[R_{2}^{(3),{\mathbf x}}R_1^{\mathbf x}\left(z-1\otimes\xi\otimes I_N-\sum_{i=1}^r(\Psi(t_i)\otimes\beta_i\otimes I_N
+\Psi(t_i)^{-1}\otimes \beta_i^* \otimes I_N+1\otimes \gamma_i \otimes U_i+1\otimes \gamma_i^* \otimes U_i^{-1})\right)^{-1}\right] \right.\\
& & \left. \left.\frac{}{}(X^{N,T_2}_{2,1}, \tilde X^{N,T_2}_{2,1}, \tilde X^{N,T_2}_{2,2},X^{N,T_2}_{2,2})\right)\right\}\\
& & {\rm d}t_1{\rm d}t_2{\rm d}t_3{\rm d}t_4\\
\\& &+\frac{1}{4}\int_{A_2}  e^{-t_4-t_3-t_2-t_1}\left\{\1_{[t_2,t_4]}(t_3) (\tr_m\otimes \mathbb{E}(\tau_N) \otimes \tau)\frac{}{} \right.\\
& & \left({}^0\mathsf{flip}^1\left[R_{2}^{(1),{\mathbf y}}R_1^{\mathbf y}\left(z-1\otimes\xi\otimes1-\sum_{i=1}^r(\Psi(\mathbf y_i)\otimes\gamma_i \otimes1
+\Psi(\mathbf y_i)^{-1}\otimes \gamma_i^* \otimes1+1\otimes \beta_i \otimes v_i+1\otimes \gamma_i^* \otimes v_i^{-1})\right)^{-1}\right]\right.\\
& & \left.\frac{}{}({X^{N,T_2}_{3,1}, \tilde X^{N,T_2}_{3,1}, \tilde X^{N,T_2}_{3,2},X^{N,T_2}_{3,2}}) \right)\\
& & \mbox{}+\1_{[0,t_1]}(t_3)(\tr_m\otimes \mathbb{E}(\tau_N) \otimes \tau) \\
& & \left({}^0\mathsf{flip}^1\left[R_{2}^{(2),{\mathbf y}}R_1^{\mathbf y}\left(z-1\otimes\xi\otimes1-\sum_{i=1}^r(\Psi(\mathbf y_i)\otimes \gamma_i \otimes1
+\Psi(\mathbf y_i)^{-1}\otimes \gamma_i^* \otimes1+1\otimes \beta_i \otimes v_i+1\otimes \gamma_i^* \otimes v_i^{-1})\right)^{-1}\right]\right.\\
& &(X^{N,T_2}_{1,1}, \tilde X^{N,T_2}_{1,1}, \tilde X^{N,T_2}_{1,2},X^{N,T_2}_{1,2})\\
& &\mbox{}+\1_{[t_1,t_2]}(t_3)(\tr_m\otimes \mathbb{E}(\tau_N) \otimes \tau) \\
& & \left({}^0\mathsf{flip}^1\left[R_{2}^{(3),{\mathbf y}}R_1^{\mathbf y}\left(z-1\otimes \xi\otimes1-\sum_{i=1}^r(\Psi(\mathbf y_i)\otimes \gamma_i \otimes1
+\Psi(\mathbf y_i)^{-1}\otimes \gamma_i^* \otimes1+1\otimes \beta_i \otimes v_i+1\otimes \gamma_i^* \otimes v_i^{-1})\right)^{-1}\right] \right.\\ 
& & \left. \left.\frac{}{}(X^{N,T_2}_{2,1}, \tilde X^{N,T_2}_{2,1}, \tilde X^{N,T_2}_{2,2},X^{N,T_2}_{2,2})\right)\right\}\\
& & {\rm d}t_1{\rm d}t_2{\rm d}t_3{\rm d}t_4\\
& &\mbox{}+\frac{1}{2}\int_0^{+\infty}\!\!\int_0^{t_2} e^{-t_2-t_1}(\tr_m \otimes E\left(\tau_N \right) \otimes \tau)\\
&&{}^0{\sf flip}^1R_1^{\bf y}{}^0{\sf flip}^1\left\{\left[{}^1{\sf flip}^0{}^0{\sf flip}^1R_1^{\bf x}\left(z-1\otimes\xi\otimes1-\sum_{i=1}^r(\Psi(\mathbf x_i)\otimes\beta_i\otimes1
+\Psi(\mathbf x_i)^{-1}\otimes\beta_i^* \otimes1+1\otimes\gamma_i\otimes \Psi(y_i)+1\otimes\gamma_i^*\otimes\Psi(y_i)^{-1})\right)^{-1}\right]\right.\\
&&\left.\left.\frac{}{}(z^1_{t_1},\tilde z^1_{t_1},\tilde z^2_{t_1},z_{t_1}^2)\right\}(z^1_{t_1}(N),\tilde z^1_{t_1}(N),\tilde z^2_{t_1}(N),z_{t_1}^2(N))\right\}{\rm d}t_1{\rm d}t_2.\\
\end{eqnarray*}}\noindent

It remains to use the two formulae above in order to prove \eqref{estimdiffeqno}. Indeed, given that
$$
g_N(z)-g(z)-\frac{E(z)}{N^2}=\frac{\Delta_N(z)}{N^4},
$$
this estimate is proved once we have argued that
$z\mapsto E(z)$ is analytic on the domain of $g$, that $z\mapsto\Delta_N(z)$ is analytic on $\mathbb C\setminus\mathbb R$, that $\lim_{z\to\infty}zE(z)=0$, and that there
exist $M,M'>0,k'',k'\in\mathbb N$ not depending on $N\in\mathbb N$ such that
\begin{equation}\label{mk}
|E(z)|<M+\frac{M}{|\Im z|^{k''}},\ \ |\Delta_N(z)|<M'+\frac{M'}{|\Im z|^{k'}},\quad z\in\mathbb C\setminus\mathbb R.
\end{equation}
We have seen that applying any of the operators $R$ to either of the tensor coordinates of $(z-\mathcal S)^{-1}$ still yields a rational noncommutative function, the only
caveat being that a tensor coordinate might change from $\mathcal A$ to $\mathcal A^{\rm op}$. A priori this could be fatal to our proof: one can find elements $x_n\in
\mathcal A\otimes\mathcal A\simeq\mathcal A\otimes\mathcal A^{\rm op}$ (algebraic tensor products and algebraic vector space identification) whose min norm stays bounded in 
the first tensor product but tends to infinity in the second. This phenomenon, however, is irrelevant for our purposes, because we use only two types of estimates in
all our computations, which hold in both completions: $\|a\otimes b\|\leq\|a\|\|b\|$ and $\|(z-\mathcal S)^{-1}\|\le\frac{1}{|\Im z|}$. The first estimate simply does not
`see' the presence or absence of the op structure in either of the two tensor coordinates, because $\mathcal A$ and $\mathcal A^{\rm op}$ are isometrically anti-isomorphic.
The second holds whenever $\mathcal S=\mathcal S^*$ in any $C^*$-algebra. In all our formulas, $\mathcal S$ is selfadjoint regardless of whether it is seen in 
$\mathcal A\otimes\mathcal A$ or in $\mathcal A\otimes\mathcal A^{\rm op}$. 

Let us consider first $E(z)$. We make use of the explicit formula \eqref{45} in order to obtain the required estimate. By counting the number of terms in it, and the number of factors in
each term, and applying the two estimates mentioned above,
\begin{eqnarray*}
\|\eqref{45}\| & \le & \frac{2\|\Psi(s)-1\|^5}{|\Im z|^2}\|\gamma_j\|+\frac{6\|\Psi(s)-1\|^6}{|\Im z|^3}\left(\|\gamma_j\|^2+\|\gamma_j\|\|\gamma_k\|\right)\\
& & \mbox{}+\frac{4\|\Psi(s)-1\|^7}{|\Im z|^4}\left(5\|\gamma_j\|^2\|\gamma_k\|+\|\gamma_j\|\|\gamma_k\|^2\right)\\
& & \mbox{}+\frac{16\|\Psi(s)-1\|^8}{|\Im z|^5}\|\gamma_j\|^2\|\gamma_k\|^2<\frac{6144(1-|\Im z|^4)(1+\max_j\|\gamma_j\|)^4}{|\Im z|^5(1-|\Im z|)},
\end{eqnarray*}
where $s$ stands for a standard semicircular variable (one is mostly interested in the case when $|\Im z|>0$ is small, so the possibility that the denominator is zero does not bother us;
however, the meaning of the above when $|\Im z|=1$ is obvious). We have also used $\|\Psi(a)^{\pm1}-1\|\le2$ for any $a=a^*$.
Quite obviously, when one acts on the other tensor coordinate, the estimate is
$$
\frac{6144(1-|\Im z|^4)(1+\max_j\|\beta_j\|)^4}{|\Im z|^5(1-|\Im z|)},
$$
which is a polynomial in $\frac{1}{|\Im z|}$. Now let us inspect the expressions under the two integrals in the formula for $E(z)$. As described immediately following \eqref{45}, applying
the $R_1^{\mathbf y}$ from the formula for $E(z)$ above comes down to first taking $\sum_{j,k=1}^r\eqref{45}$, second evaluating $v_i=\Psi({\bf t}_i),1\le i\le r,$ with ${\bf t}_i$
being standard free semicirculars and evaluating ${\sf s}=z^1_{t_1},s=\tilde z^1_{t_1},\tilde{s}=\tilde z^2_{t_1},\tilde{\sf s}=z_{t_1}^2$ (recall the notations in 
\eqref{45} and before), and performing two more steps which we discuss in a moment. Before that, let us state that the spectrum of $\mathcal S$ remains the same
regardless of whether either of the last two tensor coordinates lay in $\mathcal A$ or in $\mathcal A^{\rm op}$. Indeed, this follows from the fact that the distribution of
an $r$-tuple of free selfadjoint elements $(a_1,\dots,a_r)$ does not change regardless of whether they are viewed in $\mathcal A$ or in $\mathcal A^{\rm op}$ (the reader 
can easily see this by recalling that $\tau(a_{i_1}\cdots a_{i_p})=\overline{\tau(a_{i_p}\cdots a_{i_1})}$ for any word $i_1\cdots i_p,i_j\in\{1,\dots r\}$ and at the same time
$\tau(a_{i_1}\cdots a_{i_p})\in\mathbb R$ whenever $a_1,\dots,a_r$ are free wrt $\tau$) and, as $\tau$ is faithful, the spectrum of $\mathcal S$ is determined by its distribution.
Also, by their definition, all of $z^1_{t_1},\tilde z^1_{t_1},\tilde z^2_{t_1},z_{t_1}^2$ have a distribution that is constant (i.e. not depending on $t_1,t_2$)
and equal to the standard semicircular distribution; while they are nor free from each other, precisely one of them appears in each of the resolvents in formula \eqref{45},
as the reader can verify by a direct inspection. Thus it follows that \eqref{45} is defined for all $z$ in the domain of analyticity of $z\mapsto g(z)$, $1\le j,k\le r$. We pursue
with the description of the third step, namely performing the algebraic operations left, right, and between the placeholders present in \eqref{45}, for each pair
$(j,k)\in\{1,\dots,r\}^2,$ with the evaluations ${\bf t} $ and $z,\tilde z$ mentioned in the previous (second) step. At this step, we observe that the results, that is,
the elements that one obtains between each such two placeholders, or left of the leftmost, or right of the rightmost placeholder, are majorized each as described above,
regardless of whether the operators $z^1_{t_1},\tilde z^1_{t_1},\tilde z^2_{t_1},z_{t_1}^2$ are viewed as belonging to $\mathcal A$ or to $\mathcal A^{\rm op}$.
Then one views each of the resulting formulas as being in $\mathcal A$ and replaces the placeholders with copies of theunit $1\in\mathcal A$. The element thus obtained
is majorized by a multiple of $\frac{6144(1-|\Im z|^4)}{|\Im z|^5(1-|\Im z|)}[(1+\max_j\|\beta_j\|)^4+(1+\max_j\|\gamma_j\|)^4]$ in norm. Since $\mathrm{tr}_m\otimes
\tau\otimes\tau$ is a state, hence of norm equal to one, it follows that 
\begin{eqnarray*}
\lefteqn{\Big|(\mathrm{tr}_m\otimes\tau\otimes\tau)\!\Big\{\!
{}^0\mathsf{flip}^1\!\Big[R_1^{\mathbf y}\Big(z-1\otimes\xi\otimes1-\!\sum_{i=1}^r\!\big(\{\Psi({\bf y}_i)\otimes\gamma_i+\Psi({\bf y}_i)^{-1}\otimes\gamma_i^*\}\!\otimes\!1}\\
& & \mbox{}+1\otimes \beta_i \otimes v_i+1\otimes \gamma_i^* \otimes v_i^{-1}\big)\Big)^{-1}\Big]( z^1_{t_1\frac{}{}}\!\!, \tilde z^1_{t_1}, \tilde z^2_{t_1}, z_{t_1}^2)\\
& &\mbox{}+{}^1\mathsf{flip}^0{}^0\mathsf{flip}^1\!\Big[R_1^{\mathbf x}\Big(z-1\otimes\xi\otimes I_N-\!\sum_{i=1}^r\!\big(\{\Psi(\mathbf x_i)\otimes\beta_i
+\Psi(\mathbf x_i)^{-1}\otimes \beta_i^*\} \otimes I_N\\
& & \mbox{}+1\otimes \gamma_i \otimes u_i+1\otimes \gamma_i^* \otimes u_i^{-1})\Big)^{-1}\Big]( z^1_{t_1\frac{}{}}\!\!, \tilde z^1_{t_1}, \tilde z^2_{t_1}, z_{t_1}^2)\Big\}\Big|\\
&<&K\frac{6144(1-|\Im z|^4)}{|\Im z|^5(1-|\Im z|)}[(1+\max_{j=1}^r\|\beta_j\|)^4+(1+\max_{j=1}^r\|\gamma_j\|)^4]
\end{eqnarray*}
for some $K>0$ not depending on $N,m,z,$ but only on $r$ (this is again a polynomial in $\frac{1}{|\Im z|}$). Finally, again by inspecting \eqref{45}, one notices that 
$\lim_{z\to\infty}z\eqref{45}=0$ in norm, so by the above one has $\lim_{z\to\infty}zE(z)=0$ (in fact $E$ is easily seen to be $o(|z|^3)$ as $|z|\to\infty$). By \cite[Satz 9 and 10]{Till} 
(see \cite[Theorem 5.4]{S}), it follows that $E(z)$ is the Cauchy transform of a compactly supported Schwartz distribution which sends the constant functions to zero and is supported
on the set of singularities of $E$, which has been seen to coincide with the set of singularities of $g$, i.e. the spectrum of $\mathcal S$ evaluated in free semicircular systems.

It remains to discuss $\Delta_N$. Fortunately here one only needs to investigate the behavior of $z\mapsto\Delta_N(z)$ close to $\mathbb R$. As described in
the definition \eqref{rrr} of $R_{2}^{(j)},j=1,2,3,$ they are as well compositions of difference-differential operators, flips and evaluations. As seen in Section \ref{diffev}
(see \eqref{frezzy}), each such application (to $(z-\mathcal S)^{-1}$ and derivatives of it) results in a new rational function, possibly in a larger number of variables, possibly in 
a rational function of higher order, possibly with some of the tensor coordinates belonging to the opposite algebra. However, the inverses that occur in any such expression
are exclusively of the form $(z-\mathcal S)^{-1}$, evaluated in various selfadjoint tuples, all other factors in each term being polynomials in Cayley transforms of selfadjoints.
Since any finite succession of applications of difference-differential operators to a rational function like $(z-\mathcal S)^{-1}$ yields a finite sum of finite products
of such resolvents of selfadjoint polynomials, and selfadjoint polynomials, in Cayley transforms of indeterminates, and all evaluations (see the expresion of $\Delta_N$ and of the 
$R$s) are in tuples of selfadjoints, it follows that the expressions under the traces in the definition of $\Delta_N$ are bounded in norm by products of $|\Im z|^{-1}$, one for each
resolvent, and multiples of norms of coefficients $\gamma$ and $\beta$, neither of them depending on $N$. The fact that all of $\tau,{\rm tr}_m,\mathbb E({\rm tr}_N)$
are states, hence of norm one, implies that all the elements under the integrals in the expression of $\Delta_N$ are bounded by the same quantities.
This guarantees the existence of the $M',k'$ required in \eqref{mk}.

This completes the proof of \eqref{estimdiffeqno} and thus, as explained in Section \ref{HTSapproach}, the proof of Theorem \ref{scunitary} and finally, as explained in Section 
\ref{Sec:WU}, the proof of Theorem \ref{mainresult}.


\end{document}